\documentclass{amsart}
\usepackage[utf8]{inputenc}
\usepackage{amsmath}
\usepackage{amssymb}
\usepackage{amsthm}
\usepackage{todonotes}
\usepackage{lipsum}

\usepackage[
backend=biber,
style=alphabetic,
sorting=nyvt
]{biblatex}
\newtheorem{definition}{Definition}
\newtheorem{theorem}{Theorem}
\newtheorem{lemma}{Lemma}
\newtheorem{conjecture}{Conjecture}
\newtheorem*{remark}{Remark}
\newtheorem{corollary}{Corollary}
\numberwithin{equation}{section}
\numberwithin{theorem}{section}
\numberwithin{lemma}{section}
\numberwithin{definition}{section}
\numberwithin{corollary}{section}
\numberwithin{conjecture}{section}

\title[Graded Extensions of generalized Haagerup categories]
{Graded Extensions of generalized Haagerup categories} 

\author{Pinhas Grossman}
\address{School of Mathematics and Statistics, University of New South Wales,
Sydney NSW 2052, Australia}
\email{p.grossman@unsw.edu.au}

\author{Masaki Izumi}
\address{Department of Mathematics\\ Graduate School of Science\\
Kyoto University\\ Sakyo-ku, Kyoto 606-8502\\ Japan}
\email{izumi@math.kyoto-u.ac.jp}

\author{Noah Snyder}
\address{Department of Mathematics, Rawles Hall, Indiana University, Bloomington, Indiana 47405-7106, USA}
\email{nsnyder1@indiana.edu}

\subjclass[2010]{ 
Primary 46L37; Secondary 18D10}
\keywords{ 
subfactors, fusion categories, Cuntz algebras}

\thanks{Supported in part by JSPS KAKENHI Grant Number JP20H01805; ARC grants DP140100732, DP170103265, and DP200100067; and NSF DMS grant number 2000093.}

\begin{document}

\maketitle
\begin{abstract}
    We classify certain $\mathbb{Z}_2 $-graded extensions of generalized Haagerup categories in terms of numerical invariants satisfying polynomial equations. 
    In particular, we construct a number of new examples of fusion categories, including: $\mathbb{Z}_2 $-graded extensions of $\mathbb{Z}_{2n} $ generalized Haagerup categories for all $n \leq 5 $; $\mathbb{Z}_2 \times \mathbb{Z}_2 $-graded extensions 
    of the Asaeda-Haagerup categories; and extensions of the $\mathbb{Z}_2 \times \mathbb{Z}_2 $ generalized Haagerup category by its outer automorphism group $A_4 $.
    The construction uses endomorphism categories of operator algebras, and in particular, free products of Cuntz algebras with free group C$^*$-algebras.
\end{abstract}
\section{Introduction}
A quadratic category is a fusion category whose set of simple objects has exactly two orbits under the (left) tensor product action of the subcategory of invertible objects. Quadratic categories play a prominent role in the classification of small-index subfactors. Indeed, with a notable exception (the Extended Haagerup categories), all known fusion categories can be constructed by starting with either categories coming from quantum groups at roots of unity or starting with quadratic fusion categories, and then applying certain constructions. 

In this paper we study one of these constructions ($G$-extensions) applied to one of the most important families of quadratic categories: the generalized Haagerup categories.  One motivating application of these techniques is to resolve in the positive the open question of whether the Asaeda-Haagerup fusion categories admit extensions by their full Brauer-Picard group, which is the Klein $4$-group.

Generalized Haagerup categories were introduced as a generalization of Haagerup's famous original example appearing in the classification of small index subfactors \cite{MR1686551}, by replacing the group $\mathbb{Z}_3 = \mathrm{Inv}(\mathcal{C})$ of isomorphism classes of invertible objects which appears in the Haagerup subfactor with an arbitrary finite Abelian group. A generalized Haagerup category is tensor generated by a single simple object $X$, and satisfies the following fusion rules (plus some cohomological conditions) \cite{MR3827808}:
$$g \otimes X \cong X \otimes g^{-1}, \ \forall g \in \mathrm{Inv}(\mathcal{C}), \quad X \otimes X\cong 1 \oplus \bigoplus_{g \in \mathrm{Inv}(\mathcal{C})}\limits g \otimes X.$$

Generalized Haagerup categories were classified in \cite{MR3827808} in terms of solutions of certain polynomial equations; moreover, when there is such a solution the category can be realized as a category of endomorphisms of a von Neumann factor completion of a Cuntz algebra.  We will be generalizing this approach to also treat extensions of generalized Haagerup categories, but this generalization will require replacing Cuntz algebras by more complicated algebras.

A $G$-extension of a fusion category $\mathcal{C}$ is a $G$-graded fusion category $\mathcal{D}$ whose trivial component is $\mathcal{C}$.  There is a general obstruction theory for $G$-extensions developed by Etingof-Nikshych-Ostrik using the homotopy type of the Brauer-Picard groupoid of $\mathcal{C}$ \cite{MR2677836}.  As is typical for obstruction theories, this is quite easy to apply when the cohomology groups where the obstructions live are trivial, but if the groups are non-trivial it can be quite difficult to figure out whether the obstruction vanishes or not.  In this paper we will take a much more bare-hands approach, using concrete realizations of our examples as categories of endomorphisms, and explicitly computing structure constants. 

In general, the non-trivially graded parts of a $G$-extension of $\mathcal{C}$ will be non-trivial invertible bimodule categories over $\mathcal{C}$.  In this paper we will be considering the special case of quasi-trivial extensions, where each of these bimodules comes from an outer automorphism of $\mathcal{C} $ (i.e. it is trivial as either a left or right module, but the two actions are twisted by an outer automorphism relative to each other).  


Our first main result says:

\begin{theorem}
Unitary extensions of a generalized Haagerup category $\mathcal{C}$ by an outer action of $\mathbb{Z}_2$ which is trivial on the subcategory of invertible objects are completely classified by solutions to certain polynomial equations. Moreover, when these polynomial equations are satisfied then the extensions may be explicitly realized as categories of endomorphisms of a factor completion of the free product $\mathcal{O}_{n+1} * \mathcal{O}_{n+1} * C^*(\mathbb{F}_3)$ where $\mathcal{O} $ denotes a Cuntz algebra, $\mathbb{F} $ denotes a free group, and  $n$ is the size of $\mathrm{Inv}(\mathcal{C})$. (See Theorems \ref{relations}, \ref{rcrthm}, and \ref{equiv} below for the precise statements).

\end{theorem}

Such outer actions can only exist when the group  $\mathrm{Inv}(\mathcal{C})$ has even order. Generalized Haagerup categories are known to exist for all cyclic groups of size $\leq 10 $ (with multiple distinct examples for certain groups), and we solve the polynomial equations for $\mathbb{Z}_2 $-extensions for all of the examples in this range, thereby constructing new fusion categories in each case. In fact, due to choices in the construction of the extension, we have $4$ different $\mathbb{Z}_2 $-extensions for each example, which are also distinct as tensor categories (some of the choices even lead to different fusion rules). 

We then generalize these techniques to give applications in two further examples of interest.  First, we consider the category $\mathcal{AH}_4$ in the Morita equivalence class of the Asaeda-Haagerup subfactor.  This can be constructed as a degenerate version of a generalized Haagerup category for the group $\mathbb{Z}_4 \times \mathbb{Z}_2$, where the second factor acts trivially and so the group of invertible objects up to isomorphism is $\mathbb{Z}_4$.  In prior work  we calculated the Brauer-Picard groupoid of the Asaeda-Haagerup fusion categories and saw that the Brauer-Picard group is the Klein $4$-group \cite{MR3449240,MR3859276}.  Using Etingof-Nikshych-Ostrik's obstruction theory, it is easy to see that these fusion categories have $\mathbb{Z}_2$-extensions for each subgroup of the Klein $4$-group \cite{MR3354332}, but since the Klein $4$-group is not cyclic the question of whether there is an extension by the full Klein $4$-group is substantially more difficult.  For the original fusion categories $\mathcal{AH}_1$ and $\mathcal{AH}_2$, which arise as the even parts of the Asaeda-Haagerup subfactor, the invertible bimodule categories do not come from outer automorphisms, but for $\mathcal{AH}_4$ all the bimodule categories do come from outer automorphisms.  Thus the problem of finding extensions of $\mathcal{AH}_4$ is very close to the setting of our main result.  This leads to our second result.

\begin{theorem}
The Asaeda-Haagerup fusion category $\mathcal{AH}_4$ has an extension by its Klein $4$-group of outer automorphisms. Moreover, this extension can be explicitly realized as a category of endomorphisms of a factor completion of the algebra
$\mathcal{O}_{9} * \mathcal{O}_{9} * C^*(\mathbb{F}_3)$.
\end{theorem}

By Etingof-Nikshych-Ostrik's theory, we can conclude that the obstruction vanishes, and hence all of the Asaeda-Haagerup fusion categories have extensions by their full Brauer-Picard group; and moreover all such extensions can be easily classified via group cohomology.  These extensions give some new rich and complicated examples of fusion categories.  Homotopy theoretically this can be summarized by saying that the Brauer-Picard $3$-groupoid is homotopy equivalent to the product of Eilenberg-Maclane spaces $K(\mathbb{Z}_2 \times \mathbb{Z}_2,1) \times K(\mathbf{C}^\times,3)$, or equivalently that the Postnikov $k$-invariant vanishes.

Our other application is to the generalized Haagerup category for the group $\mathbb{Z}_2 \times \mathbb{Z}_2$.  This category is related to a conformal inclusion $SU(5)_5 \subset \mathrm{Spin}(24) $; see \cite{MR3764563,2102.09065}. This category is interesting because its Brauer-Picard group is unusually rich: it was shown in \cite{MR4001474} that this group has order $360$, and it was identified as $S_3 \times A_5 $ in \cite{2102.09065}. The outer automorphism subgroup is $A_4$.  We show using similar techniques to our main theorem:

\begin{theorem}
There is an $A_4$-graded extension of the $\mathbb{Z}_2 \times \mathbb{Z}_2$ generalized Haagerup category by its outer automorphism group. Moreover, this extension can be realized as a category of endomorphisms of a factor closure of the algebra $\mathcal{O}_{5} * \mathcal{O}_{5}*\mathcal{O}_{5} *\mathcal{O}_{5}* C^*(\mathbb{F}_{13})$.
\end{theorem}

Again this implies that the relevant obstruction vanishes and hence lets us completely classify all such extensions, of which there are exactly 15 up to equivalence. We also classify all extensions by subgroups of the outer automorphism  group. Thus we determine the extension theory associated to the outer automorphism subgroup of the Brauer-Picard group. It is an interesting problem to determine the extension theory by the entire Brauer-Picard group; however we do not currently see an accessible way to approach this.   

The paper is organized as follows.

In Section 2 we review some background material on fusion categories, extension theory, generalized Haagerup categories, and outer automorphisms.

In Section 3 we give the classification of certain $\mathbb{Z}_2 $-extensions of generalized Haagerup categories.

In Section 4 we look at some examples, including generalized Haagerup categories for cyclic groups, the Asaeda-Haagerup categories, and the generalized Haagerup category for $\mathbb{Z}_2 \times \mathbb{Z}_2 $.

In Section 5 we study the $\mathbb{Z}_2 \times \mathbb{Z}_2 $ generalized Haagerup example further, and classify all of its quasi-trivial extensions.

A long and tedious calculation needed for the argument in Section 5 is deferred to an Appendix.

\textbf{Acknowledgements.} We would like to thank Cain Edie-Michell for pointing out to us Davydov and Nikshych's result \cite[Corollary 8.7]{2006.08022}.

We would like to dedicate this paper to the memory of Vaughan Jones. During the first semester that Noah attended Vaughan's subfactor seminar at UC Berkeley, Pinhas gave a talk on his joint work with Vaughan \cite{MR2257402} in which he drew the intermediate subfactor lattice for the index $6+4\sqrt{2}$ and Vaughan declared with satisfaction ``Now that's a finite quantum group!"  We all miss him, and we'd like to think that these rich extensions might have elicited a similar response.

\section{Background}
\subsection{Fusion categories}
A fusion category over an algebraically closed field $k$ is a rigid semisimple $k$-linear monoidal category with finitely many simple objects up to isomorphism and finite-dimensional morphism spaces, and such that the unit object is simple \cite{MR2183279}. In this paper $k$ will always be the field $ \mathbb{C}$ of complex numbers.

An object $X$ in a fusion category is said to be invertible if there is another object $Y$ such that $X \otimes Y \cong 1$ (where $1$ is the unit object). The invertible objects in a fusion category $\mathcal{C} $ form a tensor subcategory $\mathrm{Inv}(\mathcal{C}) $, and the set of isomorphism classes of invertible objects is a group, by an abuse of notation also sometimes denoted by $\mathrm{Inv}(\mathcal{C}) $.

One can define left and right module categories and bimodule categories over fusion categories, as well as relative tensor products - see \cite{MR2677836} for details. A bimodule category is said to be invertible if its relative tensor product with its opposite bimodule category is equivalent to a trivial bimodule.  Invertible bimodule categories are also called Morita equivalences.

One way that invertible bimodule categories arise is through automorphisms. Given a tensor autoequivalence $\alpha$ of a fusion category $\mathcal{C} $, there is an invertible bimodule category ${}_\mathcal{C} \mathcal{C} {}_{\alpha(\mathcal{C})}  $, where the right action of $\mathcal{C} $ is twisted by $\alpha$. This bimodule is equivalent to the trivial bimodule
${}_\mathcal{C} \mathcal{C} {}_\mathcal{C}$ iff $\alpha $ is inner (isomorphic to conjugation by an invertible object). The set of isomorphism classes of tensor autoequivalences of $\mathcal{C} $, modulo inner autoequivalences, is a group, denoted by $\mathrm{Out}(\mathcal{C}) $.

To any fusion category $\mathcal{C} $, one can associate the Brauer-Picard $3$-groupoid, whose objects are fusion categories Morita equivalent to $\mathcal{C}$, whose $1$-morphisms are Morita equivalences between such categories, whose $2$-morphisms are bimodule equivalences, and whose $3$-morphisms are bimodule natural isomorphisms.  This can be truncated: in particular, the Brauer-Picard groupoid consists just of Morita equivalences modulo equivalence, and the Brauer-Picard group consists of Morita autoequivalences of $\mathcal{C}$ up to equivalence.  Also, by the homotopy hypothesis, one can think of a $3$-groupoid as a homotopy $3$-type (that is, a space in the sense of algebraic topology, whose homotopy groups vanish above $3$).

In this paper, we are primarily concerned with unitary fusion categories. A fusion category is called unitary if it is equipped with a $*$ (sometimes called ``dagger'') structure which makes it into a C$^*$-tensor category (see
\cite{MR1010160} for the definition of a (strict) C$^*$-tensor category). When discussing tensor functors between unitary fusion categories, we assume such functors are also unitary, i.e. compatible with the C$^*$-structure. Unitary fusion categories are closely related to operator algebras; see Section \ref{endcat} below.

\subsection{Extension theory}
Let $\Gamma$ be a finite group. A $\Gamma$-graded fusion category is a fusion category with a direct sum decomposition $$\mathcal{C}=\bigoplus_{g \in \Gamma}\limits \mathcal{C}_g$$ where the $\mathcal{C}_g $ are full Abelian subcategories and the tensor product bifunctor maps $\mathcal{C}_g \times \mathcal{C}_h$ to $\mathcal{C}_{gh}, \ \forall g,h \in \Gamma $. The trivial component $ \mathcal{C}_{e}$ is then a fusion category and all of the graded components $\mathcal{C}_g $ are $\mathcal{C}_e $-$\mathcal{C}_e $ bimodule categories. If the grading is faithful, then these bimodule categories are all invertible \cite{MR2677836}.

\begin{definition}
A $\Gamma$-extension of a fusion category $\mathcal{C} $ is a faithfully $\Gamma$-graded fusion category whose trivial component is tensor equivalent to $\mathcal{C}$.
\end{definition}

{
Whenever we discuss equivalence between two $\Gamma$-extensions of $\mathcal{C}$, we once fix tensor equivalences between $\mathcal{C}$ and the trivial components of the extensions, and then identify them afterward. 

\begin{definition}
We say that two $\Gamma$-extensions $\mathcal{D}$ and $\mathcal{D'}$ of $\mathcal{C}$ are equivalent if 
there exists a tensor equivalence $\mathcal{F}$ from $\mathcal{D}$ to $\mathcal{D}'$ satisfying $\mathcal{F}|_{\mathcal{D}_e}=\mathrm{id}$ and $\mathcal{F}(\mathcal{D}_g)=\mathcal{D'}_g$ for every $g\in \Gamma$. We denote by $\mathrm{Ext}_\Gamma(\mathcal{C})$ the set of equivalence classes of $\Gamma$-extensions of $\mathcal{C}$. 
\end{definition}
}

Note that one can have inequivalent extensions which nonetheless are equivalent as tensor categories.  This can happen either because the equivalence permutes the gradings, or because the equivalence restricts non-trivially to $\mathcal{C}$; see \cite{MR4192836}.

On the other hand, there is an even less flexible definition where in addition to fixing the zero graded part and fixing the grading, you also fix the bimodule categories.  The main statements in \cite{MR2677836} implicitly use this even more restrictive definition.  To correct those results for the above definition of extension, one needs to look at orbits under the action of applying a bimodule autoequivalence to each graded part in a coherent way.  See \cite{2006.08022} for more detail.

One way that $\Gamma$-extensions arise is from categorical group actions: if $\Gamma $ acts on $\mathcal{C} $, then there is a corresponding semidirect product $\mathcal{C} \rtimes \Gamma $, which is a $\Gamma$-extension of $\mathcal{C} $.

\begin{definition}
A $\Gamma$-extension of $\mathcal{C} $ is called trivial if it is equivalent to a semidirect product of a categorical action of $\Gamma$ on $\mathcal{C}$. A $\Gamma$-extension is called quasi-trivial if each graded component contains an invertible object.   
\end{definition}
Equivalently, an extension is quasi-trivial if each of the homogenous components is equivalent to the trivial module as a (left) $\mathcal{C} $-module category.



The following result of Etingof-Nikshych-Ostrik describes extensions in terms of the Brauer-Picard groupoid.
\begin{theorem}[\cite{MR2677836}]\label{ENOth}
A group homomorphism $c$ from a finite group $\Gamma$ into the Brauer-Picard group of $\mathcal{C} $ determines an obstruction class  in $O^3(c)\in H^3(\Gamma,\mathrm{Inv}(\mathcal{Z}(\mathcal{C}))  ) $ for the existence of a $\mathcal{C} $-bimodule quasi-tensor product (defined there) on the $\Gamma$-indexed collection of bimodules coming from the map. If this obstruction vanishes, then the set of such $\mathcal{C} $-bimodule quasi-tensor products  is a torsor for $H^2(\Gamma,\mathrm{Inv}(\mathcal{Z}(\mathcal{C}))) $.

Then each such $\mathcal{C} $-bimodule quasi-tensor product $M$  determines an obstruction class in $O^4(c,M)\in H^4(\Gamma,\mathbb{C}^*) $ for the existence of an associativity constraint. If this obstruction vanishes, then the set of associativity constraints $A$ for the quasi-tensor product forms a torsor over $H^3(\Gamma,\mathbb{C}^*) $.
\end{theorem}

{
The $H^3(\Gamma,\mathbb{C}^*)$ torsor structure can be realized in a concrete manner 
as follows. 
Let $\mathcal{D}$ be a $\Gamma$-extension of $\mathcal{C}$, and let $[\omega]\in H^3(\Gamma,\mathbb{C})$. 
Then we can put 
$$[\omega]\cdot \mathcal{D}=\bigoplus_{g\in \Gamma}\mathcal{D}_g\boxtimes g\subset \mathcal{D}\boxtimes \mathrm{Vec}_\Gamma^\omega.$$
In the context of operator algebras, this procedure corresponds to taking an (outer) tensor product with a $\Gamma$-kernel with obstruction $[\omega]$, which we often use in this work.

The parametrization in Theorem \ref{ENOth} does not classify extensions up to equivalence, in the sense defined above, because two associators $A$ and $A'$ for a given pair $(c,M)$ with $A'_{f,g,h}=\omega(f,g,h)\circ A_{f,g,h}$ and $[\omega]\in H^3(\Gamma,\mathbb{C}^\times)\setminus\{0\}$ may give equivalent  extensions. 
The missing piece for complete classification was obtained recently by Davydov and Nikshych.

\begin{theorem}[{\cite[Corollary 8.7]{2006.08022}}]\label{DNth} 
Let the notation be as above. Then there exists a group homomorphism  $p^1_{(c,M)}:H^1(\Gamma,\mathrm{Inv}(\mathcal{Z}(\mathcal{C})) )\to H^3(\Gamma,\mathbb{C}^\times)$ satisfying the following property: 
Let $A$ and $A'$ be associators for $(c,M)$, and let $\omega\in Z^3(\Gamma,\mathbb{C}^\times)$ with $A_{f,g,h}=\omega(f,g,h)\circ A'_{f,g,h}$. 
Then the two $\Gamma$-extensions of $\mathcal{C}$ arising from $A$ and $A'$ are equivalent if and only if the cohomology class $[\omega]$ is in the image of $p^1_{(c,M)}$. 
In consequence, the equivalence classes of $\Gamma$-extensions of $\mathcal{C}$ with $(c,M)$ form a torsor over $\mathrm{coker}(p^1_{(c,M)})$.  
\end{theorem}
}


In practice it can of course be difficult to compute the obstruction classes for specific examples. One of the motivations of this work is to provide interesting examples of graded extensions.

{
\begin{remark} When $\Gamma$ is a finite group, we have $H^n(\Gamma,\mathbb{C}^\times)=H^n(\Gamma,\mathbb{T})$ for $n\geq 1$ because  $\mathbb{C}^\times \cong \mathbb{R}\times \mathbb{T}$ as trivial $\Gamma$-modules and $H^n(\Gamma,\mathbb{R})=\{0\}$ for $n\geq 1$. 
Thus we mainly discuss $H^n(\Gamma,\mathbb{T})$ as it is more natural from the view point of operator algebras. In fact, there should be a version of Etingof-Nikshych-Ostrik's extension theory in the unitary setting using an appropriate unitary analogue of the Brauer-Picard group where $\mathbb{T}$ appears as $\pi_3$, but we will not require this unitary version of obstruction theory in this paper.
\end{remark}
}
\subsection{The category $\mathrm{End}_0(M) $} \label{endcat}
Let $M$ be a Type $\mathrm{III}$ factor. The $\mathbb{C} $-linear category $\mathrm{End}(M)$ has as objects the normal unital $*$-endomorphisms of $M$, and as morphisms elements of $M$ which intertwine such endomorphisms:
$$\mathrm{Hom}(\rho,\sigma)=\{ t \in M:t \rho(x) = \sigma(x)t, \ \forall x \in M  \} .$$
This can be made into a strict monoidal category by defining
$$\rho \otimes \sigma = \rho \circ \sigma $$
and 
$$t \otimes s =t \rho_1(s)=\sigma_1(s)t, \quad  t \in \mathrm{Hom}(\rho_1,\sigma_1), \ s \in \mathrm{Hom}(\rho_2,\sigma_2) . $$
The identity automorphism is a monoidal unit.

Let $\mathrm{End}_0(M)$ be the full subcategory of $\mathrm{End}(M) $ whose objects are endomorphisms with finite-index (see \cite{MR1027496} for a discussion of index in infinite factors). Then $\mathrm{End}_0(M) $ is still a monoidal category, and it is also rigid and semi-simple with finite-dimensional morphism spaces. Thus any full tensor subcategory of $\mathrm{End}_0(M) $ with finitely many simple objects is a unitary fusion category.  Conversely, every unitary fusion category embeds into $\mathrm{End}_0(M) $ for some $M$ (in fact $M$ can be taken to be any hyperfinite Type III factor) in an essentially unique way.

Recall that a tensor functor from a strict fusion category $\mathcal{C}$ to another strict fusion category $\mathcal{D}$ is a pair $(F,L)$ 
consisting of a functor $F:\mathcal{C}\to \mathcal{D}$ and natural isomorphisms 
$$L_{\rho,\sigma}\in \mathrm{Hom}_\mathcal{D}(F(\rho)\otimes F(\sigma), F(\rho\otimes \sigma))$$ 
satisfying 
$$L_{\rho\otimes \sigma,\tau}\circ (L_{\rho,\sigma}\otimes I_{F(\tau)})
=L_{\rho,\sigma\otimes \tau}\circ (I_{F(\rho)}\otimes L_{\sigma,\tau})$$
for any $\rho,\sigma,\tau\in \mathcal{C}$.  
We may and do assume $F(\mathbf{1}_{\mathcal{C}})=\mathbf{1}_{\mathcal{D}}$ and $L_{\mathbf{1}_\mathcal{C},\rho}=L_{\rho,\mathbf{1}_{\mathcal{C}}}=I_{F(\rho)}$.   
When $\mathcal{C}$ and $\mathcal{D}$ are C$^*$ categories, we further assume that $L_{\rho,\sigma}$ is a unitary.

The following uniqueness result is \cite[Theorem 2.2]{MR3635673}, essentially due to Popa.

\begin{theorem}\label{uniqueness} Let $M$ and $P$ be hyperfinite type III$_1$ factors, 
and let $\mathcal{C}$ and $\mathcal{D}$ be unitary fusion categories embedded in $\mathrm{End}_0(M)$ and $\mathrm{End}_0(P)$ respectively. 
Let $(F,L)$ be a tensor functor from $\mathcal{C}$ to $\mathcal{D}$ that is an equivalence of the two unitary fusion categories $\mathcal{C}$ and $\mathcal{D}$. 
Then there exists a surjective isomorphism $\Phi:M\to P$ and unitaries $U_\rho \in P$ for each object $\rho \in \mathcal{C}$ 
satisfying  
$$F(\rho)=\mathrm{Ad} U_\rho \circ \Phi \circ\rho\circ\Phi^{-1},$$
$$F(t)=U_\sigma\Phi(t)U_\rho^*,\quad X\in (\rho,\sigma),$$
$$L_{\rho,\sigma}=U_{\rho\circ\sigma}\Phi\circ\rho\circ\Phi^{-1}(U_\sigma^*)U_\rho^*=
U_{\rho\circ\sigma}U_\rho^*F(\rho)(U_\sigma^*).$$
\end{theorem}

When discussing the category $\mathrm{End}_0(M) $, it is common to suppress tensor product and ``Hom'' symbols, and to use square brackets to denote isomorphism classes (also called sectors).

\subsection{Generalized Haagerup categories}
A generalized Haagerup category is a unitary fusion category $\mathcal{C} $ which is tensor generated by a simple object $X$ satisfying the fusion rules
$$g \otimes X \cong X \otimes g^{-1}, \ \forall g \in \mathrm{Inv}(\mathcal{C}), \quad X \otimes X\cong 1 \oplus \bigoplus_{g \in \mathrm{Inv}(\mathcal{C})}\limits g \otimes X,$$
and satisfying certain cohomological conditions (see \cite{MR3827808}).

It is shown in \cite{MR3827808} that a generalized Haagerup category can always be realized in a standard form in $\mathrm{End}_0(M)$ as follows.

Let $G=\mathrm{Inv}(\mathcal{C})$. There is a copy of the Cuntz algebra $\mathcal{O}_{|G|+1} $ with generators $\{s\} \cup \{ t_g\}_{g \in G}$ inside $M$, a map 
$$G \rightarrow \mathrm{Aut}(M), \quad g \mapsto \alpha_g,$$ and an irreducible endomorphism $\rho $ of $M$, such that the following relations hold:

\begin{enumerate}
\item $$\alpha_g(s)=s, \quad \alpha_g(t_h)= \epsilon_g(h)t_{h+2g}, \quad \forall g,h \in G $$
$$\rho(s)=\frac{1}{d}s+\sum_{g \in G} \limits \frac{1}{\sqrt{d}}t_g^2 $$
$$\rho(t_g)=\epsilon_{-g}(g)[\eta_{-g}t_{-g}ss^*+\frac{\overline{\eta_g}}{\sqrt{d}}st_{-g} ^*+\sum_{h,k \in G} \limits A_{-g}(h,k)t_{h-g}t_{h+k-g}t_{k-g}^*] ,$$
for structure constants 
$$\epsilon_g(h) \in \{-1,1\}, \quad \eta_g \in \{ 1,e^{\frac{2\pi i }{3}},e^{-\frac{2\pi i} {3} } \}, \quad A_g(h,k) \in  \mathbb{C} $$ satisfying \begin{equation} \label{e1}
\epsilon_{h+k}(g)=\epsilon_h(g)\epsilon_k(g+2h), \ \epsilon_h(0)=1
\end{equation}
\begin{equation}
\eta_{g+2h}=\eta_g
\end{equation}
\begin{equation}\label{e3}
\sum_{h \in G} A_g(h,0)=-\frac{\overline{\eta_g}}{d}
\end{equation}
\begin{equation}\label{e4}
\sum_{h \in G} A_g(h-g,k)\overline{A_{g'}(h-g',k)}=\delta_{g,g'}-\frac{\overline{\eta_g}\eta_{g'}}{d} \delta_{k,0}
\end{equation}
\begin{equation}\label{e5}
A_{g+2h}(p,q)=\epsilon_h(g)\epsilon_h(g+p)\epsilon_h(g+q)\epsilon_h(g+p+q)A_g(p,q)
\end{equation}
\begin{equation}
A_g(h,k)=\overline{A_g(k,h)}
\end{equation}
\begin{align} \label{e8}
A_g(h,k) &=A_g(-k,h-k)\eta_g \epsilon_{-k}(g+h)\epsilon_{-k}(g+k)\epsilon_{-k}(g+h+k) \\
&=A_g(k-h,-h) \overline{\eta_g} \epsilon_{-h}(g+h)\epsilon_{-h}(g+k)\epsilon_{-h}(g+h+k) \nonumber
\end{align}
\begin{align}\label{e9}
A_g(h,k)&=A_{g+h}(h,k) \eta_g \eta_{g+k} \overline{\eta_{g+h} \eta_{g+h+k}} \epsilon_h(g)\epsilon_h(g+k) \\
&= A_{g+k}(h,k)  \overline{\eta_g \eta_{g+h}} \eta_{g+k} \eta_{g+h+k}\epsilon_k(g) \epsilon_k(g+h) \nonumber
\end{align}
\begin{align} \label{e10}
\sum_{l \in G} \limits & A_g(x+y,l)A_{g-p+x}(-x,l+p)A_{g-q+x+y}(-y,l+q) \\
 &=A_g(p+x,q+x+y)A_{g-p}(q+y,p+x+y)  \nonumber \\
 & \times \eta_{g} \eta_{g+q+x} \eta_{g+p+q+y}    \overline{ \eta_{g+p} \eta_{g+x+y} \eta_{g+q+x+y} } \nonumber \\
 &\times \epsilon_p(g-p+x)\epsilon_{p+x}(g-p+q+y)
 \epsilon_q(g-q+x+y)\epsilon_{q+y}(g-q+x) \nonumber \\
 & -\displaystyle \frac{\delta_{x,0}\delta_{y,0,}}{d} \eta_g \eta_{g+p} \eta_{g+q} \nonumber
\end{align}
 
\item  $$\alpha_g (\rho(x)) = \rho (\alpha_{-g}(x)) $$
$$\rho^2(x)=sxs^*+\sum_{g \in G}\limits t_g (\alpha_g( \rho(x))t_g^*, \quad \forall x \in M   $$
\end{enumerate}
(The second condition follows from the first one for $x$ in the Cuntz algebra.)

In such a setup, the full tensor subcategory of $\mathrm{End}_0(M) $ generated by $\rho $ is a generalized Haagerup category if the action of $G$ is outer.

We will also be interested in ``degenerate'' generalized Haagerup categories,  where the action of $G$ on $M$ may not be outer. An example of such a category for $G=\mathbb{Z}_4 \times \mathbb{Z}_2 $ is the Asaeda-Haagerup category $\mathcal{AH}_4 $, where the $ \mathbb{Z}_2$ factor acts trivially; this category is a $\mathbb{Z}_2 $-de-equivariantization of a corresponding generalized Haagerup category.


\subsection{The outer automorphism group} \label{outsection}
Let $(F,L)$ be a tensor autoequivalence of a generalized Haagerup category $\mathcal{C}$ with group of invertible objects $G=\mathrm{Inv}(\mathcal{C})$. 
Then there exists $p\in G$ and $\sigma\in \mathrm{Aut}(G)$ satisfying 
$[F(\alpha_g)]=[\alpha_{\sigma(g)}]$ and $[F(\rho)]=[\alpha_p\rho]$. 
Thus there exist unitaries $v_g,u\in U(M)$ satisfying $$F(\alpha_g)=\mathrm{Ad} (v_{\sigma(g)})\circ\alpha_{\sigma(g)} \text{ and } F(\rho)=\mathrm{Ad}(u)\circ\alpha_p\circ\rho.$$  
Note that $(\{F(\alpha_g)\}_{g\in G},\{L_{g,h}^*\}_{g,h})$ form a cocycle action of $G$ on $M$. 
Since $F(\alpha_g)$ is outer for all $g\neq e$, it is equivalent to a genuine action, and we may assume that $L_{g,h}=1$ for all $g,h\in G$ up to natural transformation. 
Then we have 
$$\mathrm{Ad}(v_g)\circ\alpha_g \circ\mathrm{Ad}(v_h)\circ\alpha_h
=\mathrm{Ad}(v_{g+h}) \circ\alpha_{g+h},$$
and $\mathrm{Ad}(v_g\alpha_g(v_h))=\mathrm{Ad}(v_{gh})$. 
This means that there exists a 2-cocycle $\omega$ in $Z^2(G,\mathbb{T})$ satisfying $v_g\alpha(v_h)=\omega(g,h)v_{gh}$, 
and the cohomology class $[\omega]\in H^2(G,\mathbb{T})$ depends only on the class $[(F,L)]\in \mathrm{Out}(\mathcal{C})$. 
Since the inner autoequivalence $\alpha_g\otimes\cdot\otimes \alpha_g^{-1}$ of $\mathcal{C}$ sends $\rho$ to $\alpha_{2g}\circ\rho$, while it leaves $\alpha_h$ invariant, only the class $[p]\in G/2G$ is an invariant of $[(F,L)]$ too.  
Thus the triple 
$$([\omega],[p],\sigma)\in H^2(G,\mathbb{T})\times G/2G\times \mathrm{Aut}(G)$$
is an invariant of the class $[(F,L)]\in\mathrm{Out}(\mathcal{C})$.  

If the cohomology class of $\omega$ is trivial, we may assume that $\{v_g\}_{g\in G}$ form an $\alpha$-cocycle by modifying $v_g$. 
Since $\alpha$ is outer, every $\alpha$-cocycle is a coboundary, and  there exists $v\in U(M)$ satisfying $v_g=v^*\alpha_g(v)$. 
Thus 
$$F(\alpha_g)=\mathrm{Ad}(v_g)\circ \alpha_g=\mathrm{Ad}(v)^{-1}\circ\alpha_g \circ\mathrm{Ad}(v),$$
and we may assume that $F(\alpha_g)=F(\alpha_{\sigma(g)})$ and $L_{g,h}=1$ for all $g,h\in G$ up to natural transformation. 

The group 
$$(H^2(G,\mathbb{T})\times G/2G)\rtimes \mathrm{Aut}(G)$$
acts on the set of solutions $(\epsilon,\eta,A)$ of the above equations modulo gauge equivalence, and we have an explicit description of $\mathrm{Out}(\mathcal{C})$ in terms of this action. 

\begin{theorem}[{\cite[Theorem 5.9]{MR3827808}}]
Let $\mathcal{C}$ be a generalized Haagerup category given by $(\epsilon,\eta,A)$. 
Then $\mathrm{Out}(\mathcal{C})$ is the stabilizer of $[(\epsilon,\eta,A)]$.
\end{theorem}

For every known example, we have $\mathrm{Out}(\mathcal{C})\subset G/2G\rtimes \mathrm{Aut}(G)$, and we may assume that $F(\alpha_g)=\alpha_{\sigma(g)}$ and $L_{g,h}=1$ for all $g,h\in G$ for every tensor autoequivalence $(F,L)$ of $\mathcal{C}$. 
Assume $\mathcal{C}$ is embedded in $\mathrm{End}_0(M)$ and $\beta\in \mathrm{Aut}(M)$ implements a tensor autoequivalence of $\mathcal{C}$ in this situation.  
Then the above argument shows that by perturbing $\beta$ by an inner automorphism, we may always  assume $\beta\circ \alpha_g\circ \beta^{-1}=\alpha_{\sigma(g)}$.

Recall that the group $\mathrm{Inv}(\mathcal{Z}(\mathcal{C}))$ plays an essential role in the extension theory.
In the case of generalized Haagerup categories satisfying a certain extra assumption - which is satisfied in all of the examples of interest below - we can identify $\mathrm{Inv}(\mathcal{Z}(\mathcal{C}))$ with $$G_2=\{g\in G;\; 2g=0\}$$ (see \cite{1501.07679}), and the action of $\mathrm{Out}(\mathcal{C})$ on $\mathrm{Inv}(\mathcal{Z}(\mathcal{C}))$ is determined by the permutation $\sigma\in \mathrm{Aut}(G)$ associated to each outer automorphism.

We end this section by describing how Theorem \ref{uniqueness} works in the case of generalized Haagerup categories. 
Assume that $\mathcal{C}$ is a generalized Haagerup category given by the Cuntz algebra model $(\alpha,\rho)$. 
Assume we have two embeddings $\mathcal{C}$ in $\mathrm{End}_0(M_i)$, $i=1,2$, where 
$M_1$ and $M_2$ are hyperfinite Type III$_1$ factors. 
More precisely, we have $\alpha^{(i)}_g,\rho^{(i)}\in \mathrm{End}_0(M_i)$ and 
homomorphisms $\iota_i:\mathcal{O}_{|G|+1}\to M_i$ satisfying $$\alpha_g^{(i)}\circ \iota_i=\iota_i\circ \alpha_g. \quad \rho^{(i)}\circ \iota_i=\iota_i\circ \rho_g.$$ 
We apply Theorem \ref{uniqueness} to the monoidal functor $(F,L)$ given by 
$$F(\alpha_g^{(1)})=\alpha_g^{(2)}, \quad F(\rho^{(1)})=F(\rho^{(2)}),$$ $$ \quad F(\iota^{(1)}(v))=\iota^{(2)}(v) \text{ for } v \in (\mu,\nu), \quad  L_{\mu,\nu}=1.$$  
Then we get an isomorphism $\Phi:M_1\to M_2$ and unitaries $u_\mu\in U(M_2)$ such that 
$$\mu^{(2)}=\mathrm{Ad} (u)_\mu\circ \Phi\circ \mu^{(1)}\circ \Phi^{-1},$$ 
$$\iota^{(2)}(v)=u_\nu\Phi(\iota^{(1)}(X))u_\mu^*,$$ 
$$u_{\mu\circ \nu}=u_\mu\Phi\circ \mu^{(1)}\circ \Phi^{-1}(u_\nu).$$ 
For $\mu=\alpha_g$ and $\nu=\alpha_h$, this shows that $\{u_{\alpha_g}\}_{g\in G}$ is a 
$\Phi\circ \alpha^{(1)}\circ \Phi^{-1}$-cocycle, and there exists a unitary $u\in U(M_2)$ satisfying $u_{\alpha_g}=u^*\Phi\circ \alpha^{(1)}\circ \Phi^{-1}(u)$. 
By replacing $\Phi$ with $\mathrm{Ad} (u)\circ \Phi$ if necessary, we may assume that $\alpha^{(2)}=\Phi\circ \alpha^{(1)}\circ \Phi^{-1}$ and $u_{\alpha_g}=1$. 
Under this condition, we have 
$$u_{\alpha_g\circ \rho}=u_{\alpha_g}\alpha^{(2)}_g(u_\rho)=\alpha^{(2)}_g(u_\rho),$$
$$u_{\rho\circ \alpha_{-g}}=u_\rho\Phi\circ \rho^{(2)}\circ \Phi^{-1}(u_{\alpha_{-g}})=u_\rho.$$
Since $\alpha_g\circ \rho=\rho\circ \alpha_{-g}$, we find that $u_\rho$ is fixed by $\alpha^{(1)}_g=\Phi\circ \alpha^{(2)}_g\circ \Phi^{-1}$. 
There is no further argument to simplify the situation. 
In conclusion, this means that when we compare two extensions of $\mathcal{C}$ by using Theorem \ref{uniqueness}, there is a freedom to replace $\rho$ by by $\mathrm{Ad}(u)\circ \rho$ with 
$u$ fixed by $\alpha_g$, while we can always fix the group part $\alpha_g$.

\section{Classification of extensions}
As mentioned in the previous section, for a generalized Haagerup category $ \mathcal{C}$ with group of invertible objects $G=\text{Inv}(\mathcal{C})$, we can identify $ \text{Out}(\mathcal{C})$ with a subgroup of 
$$(H^2(G,\mathbb{T})\times G/2G)\rtimes \mathrm{Aut}(G);$$ and moreover for all known examples, the outer automorphisms are cocycle-free in the sense that $ \text{Out}(\mathcal{C})$ lies in the subgroup $G/2G\rtimes \mathrm{Aut}(G)$. 

We would like to classify $\mathbb{Z}_2 $-graded extensions associated to an outer automorphism which fixes the invertible objects, i.e. which corresponds to the trivial element in $\mathrm{Aut}(G) $. Such an automorphism moves $\rho $ to $\alpha_p \rho $ for some  $p \in G \backslash 2G $ (note that for a given element of $\mathrm{Out}(\mathcal{C}) $, the choice of $p$ is determined only up to an element of $2G $).

As motivation for studying this type of automorphism, we note that it is shown in \cite{MR4001474} that the Brauer-Picard group of the generalized Haagerup subfactor for $\mathbb{Z}_4 $ is isomorphic to $\mathbb{Z}_2 $, and is generated by such an outer automorphism. As we will see below, such outer automorphisms also exist for all known examples of generalized Haagerup categories for even groups.

\subsection{Structure constants and constraints}

Let $\mathcal{C} $ be a generalized Haagerup category realized in standard form in $ \mathrm{End}_0(M)$. We would like to analyze the structure of an arbitrary $\mathbb{Z}_2 $-extension of $\mathcal{C} $ generated by an invertible object (automorphism) $\beta $ such that 
$$[\beta \alpha_g]=[\alpha_g \beta], \quad \forall g \in G $$
$$[\beta \rho] =[\alpha_p \rho \beta], \text{ for some } p \in G \backslash 2G .$$

So we fix $p \in G \backslash 2G$ and assume that $\beta $ is an automorphism of $M$ satisfying these fusion rules. 
We also assume that the automorphism associated to $\beta $ is cocycle-free, so that  may assume
$$\beta \circ \alpha_g = \alpha_g  \circ \beta, \quad \forall g \in G  ,$$
as explained in the previous section.

Choose a unitary $u \in M$ such that $$\beta \circ \rho =\mathrm{Ad}(u) \circ \alpha_p  \circ \rho \circ \beta .$$ 
Note that $u$ is determined up to a scalar since $\rho $ is irreducible.

\begin{lemma}
We have $[\beta^2]=[\alpha_{p+z}] $ for some $z \in G_2 $.
\end{lemma}
\begin{proof}
By assumption $\beta^2 $ is in $\mathcal{C} $, and hence isomorphic to $ \alpha_g$ for some $g \in G$.
We have $$[\alpha_{2g}\rho]=[\alpha_g \rho \alpha_{-g}]=[\beta^2 \rho \beta^{-2}]=[\beta\alpha_p \rho \beta^{-1}]=[\alpha_{2p} \rho].$$
Therefore we have $2g=2p $, and hence $g=p+z $ for some $z \in G_2 $.
\end{proof}

Now choose a unitary $v \in M$ such that $$\beta^2=\mathrm{Ad}(v) \circ \alpha_{p+z}  .$$ 

We first determine the actions of $\alpha_g $ and $\beta$ on $u $ and $v$.

\begin{lemma}
\begin{enumerate}
    \item There are characters $\chi,\mu \in \hat{G} $ such that 
    $$\alpha_g(u)=\chi(g)u, \quad \alpha_g(v)=\mu(g)v, \quad \forall g \in G .$$
    \item We have $$\beta(v)=\nu v ,$$
    where $\nu^2=\mu(p+z) $.
\end{enumerate}
\end{lemma}
\begin{proof}
We have $$\mathrm{Ad}(\alpha_g(u)) \circ \alpha_p \rho=\alpha_g \circ \mathrm{Ad}(u) \circ \alpha_p\rho  \circ \alpha_g$$ $$=\alpha_g \circ \beta  \rho  \beta^{-1} \circ \alpha_g=\beta \rho  \beta^{-1} =\mathrm{Ad}(u) \circ \alpha_p \rho.$$
Since $\alpha_p\rho $ is irreducible, it must be that $\alpha_g(u) $ is a scalar multiple of $u $; call the corresponding character $\chi $.
Similarly, $$\mathrm{Ad}(\alpha_g(v))  \circ \alpha_{p+z}=\alpha_g \circ\beta^2 \circ \alpha_{-g}=\beta^2 =\mathrm{Ad}(v) \circ \alpha_{p+z},$$
so $\alpha_g(v) $ is a scalar multiple of $v$; call the corresponding character $\mu $.

Finally, we have $$ \mathrm{Ad}(\beta(v)) \circ \alpha_{p+z}=\beta(\beta^2)\beta^{-1}=\beta^2,$$
so $\beta(v) $ is also a scalar multiple of $v$; call the corresponding scalar $\nu $.
We have 
$$\nu^2v=\beta^2(v)=(\mathrm{Ad}(v)\circ \alpha_{p+z})(v)=\mu(p+z)v .$$

\end{proof}

Note that we have not found any constraints on $\beta(u) $. Similarly, we have not found any constraints on $\beta(s) $ or $\beta(t_g) $.

We would now like to determine where $\rho $ sends $u$ and $v$.

\begin{lemma}
Replacing $u$ with a scalar multiple if necessary, we may assume that 
$$\rho(u)=u^*\beta(s)s^*+\sum_{g \in G} \limits a(g) u^*\beta(t_{g-p}) u t_g^* ,$$
where $a(g)$ is a function from $G$ to $\mathbb{T} $.

\end{lemma}

\begin{proof}
We have $u \in (\alpha_{p}\rho \beta,\beta \rho ) $, so that
$\rho(u) \in (\rho \alpha_{p}\rho \beta,\rho\beta \rho )=(\alpha_{-p}\rho^2 \beta,\rho\beta \rho ) $. Since $[\alpha_p \rho \beta]=[\beta\rho] $, we also have $$[\rho\beta\rho]=[\alpha_{-p}\rho^2 \beta]=[\alpha_{-p}\beta]\oplus \bigoplus_{g \in G} \limits [\alpha_g \rho \beta],$$
and a basis for $(\alpha_{-p}\rho^2 \beta,\rho\beta \rho )$ is given by
$$\{u^*\beta(s)s^*    \} \cup \{ u^*\beta(t_{g-p})ut_g^*\}_{g \in G}.$$
Multiplying $u$ by a scalar if necessary, we may assume that $u\rho(u)s=\beta(s) $, so that  
$$\rho(u)=u^*\beta(s)s^*+\sum_{g \in G} \limits a(g) u^*\beta(t_{g-p}) u t_g^* .$$
Since $u \rho(u)t_g=a(g)\beta(t_{g-p})u $ is an isometry for each $g$, we must have $|a(g)|=1$.
\end{proof}
Note that we can replace $u$ with $-u $, which would multiply each $a(g) $ by $-1 $.

\begin{lemma}
We have $$\rho(\beta(u))= \chi(p)(u^*\beta(u^*)[vsv^*\beta(s)^*+\sum_{g \in G} \limits a(g) \epsilon_{p+z}(g-p) vt_{g+p}v^* \beta(u) \beta(t_g)^*  ]u )$$

\end{lemma}
\begin{proof}
We have $$\rho(\beta(u))=(\alpha_{-p} \circ \mathrm{Ad}(u^*) \circ \beta \rho)(u)$$ $$=(\mathrm{Ad}(u^*) \circ \beta \rho \alpha_p)(u)=\chi(p)(\mathrm{Ad}(u^*) \circ \beta \rho )(u)$$
$$=\chi(p)u^*\beta(u^*\beta(s)s^*+\sum_{g \in G} \limits a(g) u^*\beta(t_{g-p}) u t_g^*  )u $$
$$=\chi(p)(u^*\beta(u^*)[vsv^*\beta(s)^*+\sum_{g \in G} \limits a(g) \epsilon_{p+z}(g-p) vt_{g+p}v^* \beta(u) \beta(t_g)^*  ]u )$$
(where we have used $\beta^2=\mathrm{Ad}(v) \circ \alpha_{p+z} $).
\end{proof}

\begin{lemma}
We have $$\rho(v)=\xi u^* \beta(u)^* v,$$
where $\xi \in \mathbb{T} $. 
\end{lemma}
\begin{proof}
We have $v \in (\alpha_{p+z},\beta^2 ) $, so $\rho(v) \in (\rho \alpha_{p+z}, \rho \beta^2 ) $, which is a one-dimensional space since $[\rho \alpha_{p+z}] = [\rho \beta^2]$ is irreducible. Therefore, it suffices to show that $u^* \beta(u)^* v \in (\rho \alpha_{p+z}, \rho \beta^2 )$, which can be readily checked:
$$v \in (\rho\alpha_{p+z}=\alpha_{-p+z} \rho,\alpha_{-2p}\beta^2 \rho) $$
$$\beta(u)^* \in (\alpha_{-2p}\beta^2 \rho=\beta^2\rho \alpha_{2p}, \beta\alpha_p \rho \beta \alpha_{2p}=\beta\rho \beta\alpha_{-p})  $$
 $$u^* \in (\beta\rho \beta\alpha_{-p},\rho \beta \alpha_p \beta \alpha_{-p}=\rho \beta^2 ) .$$

\end{proof}

Next, we will check constraints from the relation $\alpha_g \circ \rho=\rho \circ \alpha_{-g}$ on $u$ and $v$.

\begin{lemma}\label{frel} We have
\begin{enumerate}
    \item  $$a(h+2g)=a(h)\epsilon_{g}(h)\epsilon_{g}(h-p)\chi(g), \ \forall g,h \in G .$$
    \item  $$\mu(g)^2 =\chi(g)^2, \ \forall g \in G.$$
\end{enumerate}

\end{lemma}
\begin{proof}
For the first part, we have $$\alpha_g (\rho(u))=\alpha_g(u^*\beta(s)s^*+\sum_{h \in G} \limits a(h) u^*\beta(t_{h-p}) u t_h^* ) $$
$$=\chi(-g)u^*\beta(s)s^*+\sum_{h \in G} \limits a(h) \epsilon_g(h-p)\epsilon_g(h) u^*\beta(t_{2g+h-p}) u t_{2g+h}^*  .$$
On the other hand, $$\rho (\alpha_{-g}(u))=\chi(-g)\rho(u)=\chi(-g)(u^*\beta(s)s^*+\sum_{h \in G} \limits a(h) u^*\beta(t_{h-p}) u t_h^* ).$$
Equating terms gives the desired relation.
For the second part, we have $$\alpha_g (\rho(v))=\alpha_g(\xi u^* \beta(u)^* v)=\xi \chi(-2g) \mu(g)u^* \beta(u)^* v$$
and $$\rho(\alpha_{-g}(v))= \mu(-g)\xi u^* \beta(u)^* v,$$
so we get $\mu(2g)=\chi(2g) $.
\end{proof}

Next, we check constraints from the relation $\beta \circ \rho=\mathrm{Ad}(u)\circ \alpha_p \rho \circ  \beta $ on $v$.

\begin{lemma}\label{frel2}
We have 
$$\mu(p)=\chi(p+z) .$$
\end{lemma}
\begin{proof}
We have
$$\beta(\rho(v))=\beta(\xi u^* \beta(u)^* v )=\xi\chi(-p+z) \beta(u)^*vu^*v^*\beta(v)=\nu \xi \chi(-p+z)\beta(u)^*vu^* ,$$
while 
$$u (\alpha_p \rho (\beta(v))) u^* =\nu \mu(-p)u \rho(v)u^*=\nu \mu(-p)\xi\beta(u)^*vu^*.$$

\end{proof}

What remains is to check constraints coming from the relation
$$\rho^2(x)=sxs^*+\sum_{g \in G} \limits t_g (\alpha_g\rho)(x)t_g^* $$
for $x=v $ and $x=u$.

\begin{lemma} \label{lon1}
We have
\begin{enumerate}
\item $$\xi^2=\chi(p) $$
\item $$a(g)a(g-p)\epsilon_{p+z}(g-2p)\xi=\mu(g), \ \forall g \in G .$$
\end{enumerate}
\end{lemma}
\begin{proof}

We have 
$$\rho^2(v)=\rho(\xi u^* \beta(u)^* v )=\xi^2\rho(u)^*\rho(\beta(u))^*u^*\beta(u)^*v $$ $$=\xi^2(u^*\beta(s)s^*+\sum_{g \in G} \limits a(g) u^*\beta(t_{g-p}) u t_g^*)^* $$ $$\cdot (\chi(p)(u^*\beta(u^*)[vsv^*\beta(s)^*+\sum_{h \in G} \limits a(h) \epsilon_{p+z}(h-p) vt_{h+p}v^* \beta(u) \beta(t_h)^*  ])u )^* u^*\beta(u)^*v$$
$$=\xi^2 \chi(-p)(s\beta(s^*)+\sum_{g \in G}\limits \overline{a(g)} t_gu^*\beta(t_{g-p})^*  )$$ $$\cdot(\beta(s)vs^*+\sum_{h \in G}\limits \overline{a(h)}\epsilon_{p+z}(h-p)\beta(t_h)\beta(u)^*vt_{h+p}^*   )  $$
$$=\xi^2\chi(-p)(svs^*+ \sum_{g \in G}\limits \overline{a(g)}\overline{a(g-p)}\epsilon_{p+z}(g-2p)t_g u^* \beta(u)^*v t_g^*  ) $$
$$=\xi^2\chi(-p)(svs^*+ \sum_{g \in G}\limits \overline{a(g)}\overline{a(g-p)} \overline{\xi}\mu(g)\epsilon_{p+z}(g-2p)t_g (\alpha_g\rho)(v) t_g^*  ) .$$

Setting this equal to $$svs^*+\sum_{g \in G} \limits t_g (\alpha_g\rho)(v)t_g^*$$ gives the relations.
\end{proof}

Finally, we will look at $\rho^2(u) $.

\begin{lemma} \label{lon2}
We have
\begin{enumerate}
\item $$\eta_{g+p}=\eta_g, \quad \forall g \in G $$
\item $$\chi(g)=a(g)^2, \quad \forall g $$
    \item $$a(g)a(-g)=\epsilon_{-g}(g-p) \epsilon_{-g}(g), \quad \forall g.$$
    \item $$A_{g}(h,k)= a(g+h)a(g+k)\overline{a(g+h+k)a(g)}   A_{g-p}(h,k), \quad \forall g,h,k .$$
\end{enumerate}
\end{lemma}
\begin{proof}

We have 
$$\rho^2(u)=\rho(u^*\beta(s)s^*+\sum_{g \in G} \limits a(g) u^*\beta(t_{g-p}) u t_g^* ) $$
$$=\rho(u)^*\rho\beta(s)\rho(s)^*+\sum_{g\in G}\limits a(g) \rho(u)^*\rho\beta(t_{g-p})\rho(u) 
\rho(t_g)^*) $$
$$= (u^*\beta(s)s^*+\sum_{g \in G} \limits a(g) u^*\beta(t_{g-p}) u t_g^* )^*(\mathrm{Ad}(u^*)\beta \rho \alpha_p (s)) \rho(s)^* 
$$ $$+\sum_{g \in G}\limits a(g)(u^*\beta(s)s^*+\sum_{k \in G} \limits a(k) u^*\beta(t_{k-p}) u t_k^* )^* (\mathrm{Ad}(u^*)\beta \rho \alpha_p )(t_{g-p}) $$ $$\cdot(u^*\beta(s)s^*+\sum_{h \in G} \limits a(h) u^*\beta(t_{h-p}) u t_h^* ) \rho(t_g)^*$$
$$= (s\beta(s)^* +\sum_{g \in G} \limits \overline{a(g)} t_gu^* \beta(t_{g-p})^*  )\beta(\rho(s))u\rho(s)^*
$$ $$+\sum_{g \in G}\limits a(g)\epsilon_p(g-p)(s\beta(s)^*+\sum_{k \in G} \limits \overline{a(k)} t_k u^*\beta(t_{k-p})^* ) \beta(\rho(t_{g+p})) $$ $$\cdot(\beta(s)s^*+\sum_{h \in G} \limits a(h) \beta(t_{h-p}) u t_h^* ) \rho(t_g)^*$$
$$=s\beta(s^*\rho(s))u\rho(s)^*
+\sum_{g \in G} \limits \overline{a(g)} t_gu^* \beta(t_{g-p}^*\rho(s) )u\rho(s)^*
$$
$$ 
+\sum_{g,h \in G}\limits a(g)a(h)\epsilon_p(g-p)s\beta(s^*\rho(t_{g+p}) t_{h-p}) u t_h^*  \rho(t_g)^*
$$
$$ 
+\sum_{g,k \in G}\limits a(g)\overline{a(k)}\epsilon_p(g-p)t_ku^*\beta(t_{k-p}^*\rho(t_{g+p}) s)s^*\rho(t_g)^*
$$
$$
+\sum_{g,h,k \in G}\limits a(g)a(h)\overline{a(k)}\epsilon_p(g-p)t_ku^*\beta(t_{k-p}^*\rho(t_{g+p})t_{h-p})ut_h^*\rho(t_g)^*
$$
$$=\frac{1}{d}su\rho(s)^*
+\frac{1}{\sqrt{d}}\sum_{g \in G} \limits \overline{a(g)} t_gu^* \beta(t_{g-p} )u\rho(s)^*
$$
$$ 
+\frac{1}{\sqrt{d}}\sum_{g\in G}\limits \overline{\eta_{g+p}}a(g)a(-g)\epsilon_p(g-p)\epsilon_{-g-p}(g+p)s u t_{-g}^*  \rho(t_g)^*
$$
$$ 
+\sum_{g \in G}\limits \eta_{g+p }a(g)\overline{a(-g)}\epsilon_p(g-p)\epsilon_{-g-p}(g+p) t_{-g}u^*\beta(s)s^*\rho(t_g)^*
$$
$$
+\sum_{g,h,k \in G}\limits  a(g)a(h)\overline{a(k)}\epsilon_p(g-p)\epsilon_{-g-p}(g+p) A_{-g-p}(k+g,h+g)$$ $$\cdot t_ku^*\beta(t_{g+h+k-p})ut_h^*\rho(t_g)^*
$$

This gives
$$\rho^2(u) s=\frac{1}{d^2}su+\frac{1}{\sqrt{d}d} \sum_{g \in G} \limits \overline{a(g)} t_gu^* \beta(t_{g-p} )u$$
$$ 
+\frac{1}{d}\sum_{g\in G}\limits  \eta_g\overline{\eta_{g+p}}a(g)a(-g)\epsilon_{-g}(g)\epsilon_p(g-p)\epsilon_{-g-p}(g+p)s u 
$$
$$
+\frac{1}{\sqrt{d}}\sum_{g,k \in G}\limits \eta_g a(g)a(-g)\overline{a(k)}\epsilon_{-g}(g)\epsilon_p(g-p)\epsilon_{-g-p}(g+p) A_{-g-p}(k+g,0)t_k^*u^*\beta(t_{k-p})u
$$
$$ =\left(\frac{1}{d^2}+\frac{1}{d}\sum_{g \in G} \limits \eta_g\overline{\eta_{g+p}}a(g)a(-g)\epsilon_{-g}(g)\epsilon_p(g-p)\epsilon_{-g-p}(g+p) \right)su
$$ $$+\frac{1}{\sqrt{d}}\sum_{g \in G} \limits \overline{a(g)} \left(\frac{1}{d}+\sum_{k \in G}\limits  \eta_k a(k)a(-k)\epsilon_{-k}(k)\epsilon_p(k-p)\epsilon_{-k-p}(k+p) A_{-k-p}(k+g,0)   \right)      t_gu^* \beta(t_{g-p} )u
$$
Setting this equal to $su $ gives the equation
$$ \sum_{g \in G} \limits \eta_g\overline{\eta_{g+p}}a(g)a(-g)\epsilon_{-g}(g)\epsilon_p(g-p)\epsilon_{-g-p}(g+p)=n,$$
which implies that
\begin{equation} \label{star}
\begin{split}
\eta_g\overline{\eta_{g+p}}a(g)a(-g)\epsilon_{-g}(g)\epsilon_p(g-p)\epsilon_{-g-p}(g+p) \\ =\eta_g\overline{\eta_{g+p}}a(g)a(-g)\epsilon_{-g}(g)\epsilon_{-g}(g-p)=1,
\end{split}
\end{equation}
and
\begin{equation*}
    -\frac{1}{d}=\sum_{k \in G}\limits \eta_{k+p} A_{-k-p}(k+g,0)= \\ \sum_{k \in G}\limits \overline{\eta_{g+p}}^2 A_{g-p}(k+g,0)=\overline{\eta_{g+p}}^2\frac{-\overline{\eta_{g+p}}}{d},
\end{equation*}
which is true (where we have used Eqs. (\ref{e3}) and (\ref{e8}) ).

Similarly, 

$$\rho^2(u)t_l=\frac{1}{d\sqrt{d}}sut_l^*+\frac{1}{d} \sum_{g \in G} \limits \overline{a(g)} t_gu^* \beta(t_{g-p} )ut_l^*$$
$$+ \frac{1}{\sqrt{d}}\sum_{g\in G}\limits \overline{\eta_{g+p}}a(g)a(-g)\epsilon_{-g}(g)\epsilon_p(g-p)\epsilon_{-g-p}(g+p)
\overline{A_{-g}(l+g,0)}s u t_{l}^* $$
$$ 
+\overline{\eta_{-l}}\eta_{-l+p }a(-l)\overline{a(l)}\epsilon_p(-l-p)\epsilon_{l-p}(-l+p)\epsilon_l(-l) t_{l}u^*\beta(s)s^*
$$
$$
+\sum_{g,h,k \in G}\limits  a(g)a(h)\overline{a(k)}\epsilon_p(g-p)\epsilon_{-g-p}(g+p)\epsilon_{-g}(g) $$ $$\cdot A_{-g-p}(k+g,h+g)\overline{A_{-g}(l+g,h+g)}t_ku^*\beta(t_{g+h+k-p})ut_{h+g+l}^*
$$
Collecting terms and applying Eq. (\ref{star}) gives
$$=\frac{1}{\sqrt{d}}(\frac{1}{d}+ \sum_{g \in G} \limits \overline{\eta_g} \overline{A_{-g}(l+g,0)})sut_l^*+a(-l)^2 t_{l}u^*\beta(s)s^*+\frac{1}{d} \sum_{g \in G} \limits \overline{a(g)} t_gu^* \beta(t_{g-p} )ut_l^* 
$$
$$+\sum_{m \in G} \limits ( \sum_{k \in G}\limits \overline{a(k)}  (
\sum_{g\in G}\limits  \overline{a(-g)}\overline{\eta_g}\eta_{g+p} a(m-g) $$ $$\cdot A_{-g-p}(k+g,m)\overline{A_{-g}(l+g,m)})t_ku^*\beta(t_{m+k-p})ut_{m+l}^*)
$$
$$=a(-l)^2 t_{l}u^*\beta(s)s^*
$$
$$+\sum_{m \in G}\limits  \sum_{k \in G}\limits \overline{a(k)}( \frac{\delta_{m,0}}{d}+ (
\sum_{g\in G}\limits  \overline{a(-g)}\overline{\eta_g}\eta_{g+p} a(m-g) $$ $$\cdot A_{-g-p}(k+g,m)\overline{A_{-g}(l+g,m)}))t_ku^*\beta(t_{m+k-p})ut_{m+l}^*)
$$
(where we have used Eqs. (\ref{e3}) and (\ref{e8}) to eliminate the first term).

Comparing with $$t_l \alpha_l\rho(u)=\chi(-l)t_l(u^*\beta(s)s^*+\sum_{g \in G} \limits a(g) u^*\beta(t_{g-p}) u t_g^*)$$ $$=\chi(-l)t_lu^*\beta(s)s^*+\sum_{g \in G}\limits \chi(-l) a(g)t_lu^*\beta(t_{g-p})ut_g^* $$
we get the relations
\begin{equation} \label{star2}
a(l)^2=\chi(l) 
\end{equation}
and
$$\sum_{g\in G}\limits a(l) \overline{a(m+l)} \overline{a(-g)}\overline{\eta_g}\eta_{g+p} a(m-g) A_{-g-p}(k+g,m)\overline{A_{-g}(l+g,m)}))$$  $$= \delta_{k,l}  -\frac{\delta_{m,0}}{d}=
\eta_{k} \overline{ \eta_{l} } \sum_{g \in G} \limits   A_{-k}(k+g,m)\overline{A_{-l}(l+g,m)}$$ $$=
\eta_k^2 \eta_{k+m} \overline{\eta_l^2 \eta_{l+m}} \sum_{g \in G} \limits   \epsilon_{g+k}(-k) \epsilon_{g+k}(-k+m) \epsilon_{g+l}(-l) \epsilon_{g+l}(-l+m)   A_{g}(k+g,m)\overline{A_{g}(l+g,m)}$$ $$=\eta_k^2 \eta_{k+m} \overline{\eta_l^2 \eta_{l+m}}  \epsilon_{k}(-k)\epsilon_k(-k+m)\epsilon_l(-l)\epsilon_l(-l+m)\sum_{g \in G} \limits  A_{-g}(k+g,m)\overline{A_{-g}(l+g,m)}
$$

Setting $k=l $ and $r=k+g$ (and replacing $g$ with $-g $), we get the relation \begin{equation} \label{star3} 
A_{g-p}(r,m)=\overline{a(r+g)a(m+g)\eta_{g+p} }\eta_g a(g)a(m+r+g)  A_{g}(r,m)  .
\end{equation}

Finally, by Eq. (\ref{star2}), we have that
$$a(g)a(-g)=\pm 1 $$
is always real, so by Eq. (\ref{star2}) we must have
\begin{equation}
    \eta_g=\eta_{g+p}, \quad \forall g ,
\end{equation}
which simplifies Eqs. (\ref{star}) and (\ref{star3}).

\end{proof}

Note that although in deriving the last relation in the proof we specialized the equation to $k=l $, the resulting relation makes the equation true for all $k$ and $l$ (which we will need later for reconstruction of the category from these relations). 

Putting this all together, we arrive at the following description of $\beta $:
\begin{theorem} \label{relations}
Let $\mathcal{C} $ be a (possibly degenerate) generalized Haagerup category for $G$ realized in standard form in $\mathrm{End}_0(M)$ with structure constants $(A,\epsilon,\eta) $, and suppose that $\beta \in Aut(M) $ commutes with $\alpha_g,\  g \in G$ and there are elements $p \in G \backslash 2G$ and $z \in G_2 $ such that $$[\beta \rho]=[\alpha_p \rho \beta], \quad [\beta^2]=[\alpha_{p+z}] .$$

Then there exist unitaries $u$ and $v$ in $M$ such that $$\beta \circ \rho=\mathrm{Ad}(u) \circ  \alpha_p \circ \rho \circ \beta, \quad \beta^2 = \mathrm{Ad}(v) \circ \alpha_z ;$$
characters $\chi,\mu  \in \hat{G}$ such that
\begin{equation}
\alpha_g(u)=\chi(g)u, \quad \alpha_g(v)=\mu(g)v ;
\end{equation}
constants $\xi,\nu \in \mathbb{T} $ such that
\begin{equation}
\rho(v)=\xi u^*\beta(u^*)v  , \quad \beta(v)=\nu v ;
\end{equation}
and a function $a: G \rightarrow \mathbb{T} $ such that 
\begin{equation}\label{rhou}
\rho(u)=u^*\beta(s)s^*+\sum_{g \in G} \limits  a(g) u^* \beta(t_{g-p})ut_g^*  ;
\end{equation}

and such that the following identities hold:

\begin{equation}\label{n1}
    \nu^2=\mu(p+z) 
\end{equation}
\begin{equation}\label{n2}
    \xi^2=\chi(p) 
\end{equation}
\begin{equation}\label{n3}
    \mu(g)^2=\chi(g)^2
\end{equation}
\begin{equation}\label{n4}
   \mu(p)=\chi(p+z) 
\end{equation}
\begin{equation}\label{n5}
    a(0)=1 
\end{equation}
\begin{equation}\label{n6}
    \chi(g)=a(g)^2 
\end{equation}
\begin{equation}\label{n7}
    a(h+2g)=a(h)\epsilon_g(h)\epsilon_g(h-p)\chi(g)
\end{equation}
\begin{equation}\label{n8}
    a(g)a(g-p)\epsilon_{p+z}(g-2p)\xi=\mu(g) 
\end{equation}
\begin{equation}\label{n9}
    a(g)a(-g)=\epsilon_{-g}(g-p)\epsilon_{-g}(g)
\end{equation}
\begin{equation}\label{n10}
    A_g(h,k)=a(g+h)a(g+k)\overline{a(g+h+k)a(g)}A_{g-p}(h,k)
\end{equation}

We also must have $\eta_g=\eta_{g+p}, \ \forall g \in G $.

\end{theorem}
\begin{proof}
The relations are collected from the previous lemmas. The only new one is $a(0)=1$, which we can assume by noting that $a(0)^2=\chi(0)=1 $, so that $a(0)=\pm 1 $, and then replacing $u$ by $-u$ if necessary.
\end{proof}

When $g\in G_2$, Eq. (\ref{n7}) implies 
$$\chi(g)=\epsilon_g(h)\epsilon_g(h-p).$$
Then putting $h=0$ and $h=p$, we get 
\begin{equation} \label{n11}
\chi(g)=\epsilon_g(p)=\epsilon_g(-p), \quad \forall g \in G_2 .
\end{equation}

Some of the relations in Theorem \ref{relations} are redundant, and we can organize them in a more efficient way as follows.  

\begin{lemma} Eq. (\ref{n1})-(\ref{n10}) are equivalent to the following equations: 

\begin{equation}\label{n'1}
    \nu^2=\mu(p+z),
\end{equation}
\begin{equation}\label{n'2}
\xi=a(p)\epsilon_{-p}(p)
\end{equation}
\begin{equation}\label{n'3}
\chi(g)=a(g)^2
\end{equation}
\begin{equation}\label{n'4}
\mu(g)=a(g)a(g-p)a(p)\epsilon_{-p}(g)\epsilon_{-p}(p)\epsilon_z(g)
\end{equation}
\begin{equation}\label{n'5}
    a(0)=1 
\end{equation}
\begin{equation}\label{n'6}
    \frac{a(h+2g)}{a(h)a(2g)}=\epsilon_g(h)\epsilon_g(h-p)\epsilon_g(0)\epsilon_g(-p),
\end{equation}
\begin{equation}\label{n'7}
    a(g)a(-g)=\epsilon_{-g}(g-p)\epsilon_{-g}(g)
\end{equation}
\begin{equation}\label{n'8}
    A_g(h,k)=a(g+h)a(g+k)\overline{a(g+h+k)a(g)}A_{g-p}(h,k)
\end{equation}
\end{lemma}

\begin{definition}
We will call a collection of data $(\chi,\mu,\xi,\nu,a(g)) $ satisfying the conditions in Theorem \ref{relations} a set of extension data for $(\mathcal{C},A,\epsilon, p,z) $.
\end{definition}

\subsection{Reconstruction}
We now describe how to reconstruct a $\mathbb{Z}_2 $-graded extension of a generalized Haagerup category $ \mathcal{C}$ from its extension data.

Suppose we are given a set of extension data $(\chi,\mu,\xi,\nu,a(g)) $ for $(\mathcal{C},A,\epsilon,p,z) $. Let $\mathcal{U} =\mathcal{O}_{n+1} * \mathcal{O}_{n+1}*C^*(\mathbb{F}_3)$, which is the universal $C^*$-algebra generated by two copies of $\mathcal{O}_{n+1} $ and three unitaries $u_0$, $u_1 $, and $v$. Intuitively, we think of the first copy of $\mathcal{O}_{n+1} $ as the original Cuntz algebra for $\mathcal{C} $; the second copy as the image of the first copy under the new automorphism $\beta $; and the unitaries $u_0 $, $u_1 $, and $v $ as corresponding to $u $, $\beta(u) $, and $v$ in the previous section, respectively. 

We would like to extend $\rho $ and $ \alpha_g$ to $\mathcal{U}$ such that the original relations 
$$\alpha_g \circ \alpha_h = \alpha_{g+h}, \quad \alpha_g  \circ \rho =\rho \circ \alpha_{-g} $$
and
$$ \rho^2(x)=sxs^*+\sum_{g\in G} \limits t_g \alpha_g (\rho(x))t_g^*$$
continue to hold; then define $\beta $ on $\mathcal{U}$ such that the new relations
$$\beta \circ \alpha_g= \alpha_g \circ \beta, \quad \rho \circ \beta = \mathrm{Ad} (u_0) \circ \alpha_p \circ \rho \circ \beta, \quad \beta^2=\mathrm{Ad}(v) \circ \alpha_{p+z}  $$
also hold; and finally extend everything to a von Neumann algebra closure of $\mathcal{U}$ to get a unitary fusion category.

Let $\Phi_0 $ (resp. $\Phi_1$) be the canonical isomorphism from $\mathcal{O}_{|G|+1} $ onto the first (resp. second) copy of $\mathcal{O}_{|G|+1} $ in $\mathcal{U} $. We set $s^{(k)}=\Phi_k (s)$
and $t_g^{(k)} =\Phi_k(t_g)$.

We define a $G $-action $\tilde{\alpha} $ on $\mathcal{U} $ by
$$\tilde{\alpha}_g(u_k)=\chi(g)u_k, \quad \mbox{for } k=0,1, \quad \tilde{\alpha}_g(v)=\mu(g)v ,$$
$$\tilde{\alpha}_g (\Phi_k(x))=\Phi_k(\alpha_g(x)), \quad x \in \mathcal{O}_{|G|+1} $$
and an endomorphism 
$\tilde{\rho} $ of $\mathcal{U} $
by
$$\tilde{\rho}(\Phi_k(x))=\begin{cases}
\Phi_0(\rho(x)) & \mbox{if } k=0 \\
u_0^*\Phi_1(\rho\alpha_{p}(x)) u_0 & \mbox{if } k=1
\end{cases} , $$
$$\tilde{\rho}(u_0)=u_0^*( s^{(1)}s^{(0)}{}^*+\sum_{g \in G} \limits a(g) t^{(1)}_{g-p} u_0t^{(0)}_{g} {}^* )  ,$$ 
$$\tilde{\rho}(u_1) =\chi(p) u_0^*u_1^*(v s^{(0)}v^*s^{(1)}{}^*+\sum_{g \in G} \limits a(g) \epsilon_{p+z}(g-p) vt^{(0)}_{g+p} v^*u_1t^{(1)}_{g} {}^* )u_0 ,$$
$$\tilde{\rho}(v)= \xi u_0^*u_1^*v .$$

\begin{lemma}
We have
$$\tilde{\alpha}_g \circ \tilde{\rho}= \tilde{\rho} \circ \tilde{\alpha}_{-g} , \quad \forall g \in G$$
\end{lemma}
\begin{proof}
It is easy to see that the relation holds for $x \in \Phi_k(\mathcal{O}_{|G|+1}) $, for $k=0,1 $. For $x \in \{u_0,u_1,v \} $, the relation reduces to a similar calculation as in Lemma \ref{frel}, using Eqs. (\ref{n3}) and (\ref{n7}).



\end{proof}

We define an endomorphism $\tilde{\beta} $ on $\mathcal{U}$  by
$$\tilde{\beta}(\Phi_k(x))=\begin{cases}
\Phi_1(x) & \mbox{if } k=0 \\
v\Phi_0(\alpha_{p+z}(x)) v^* & \mbox{if } k=1
\end{cases} ,$$
$$\tilde{\beta}(u_0)=u_1, \quad \tilde{\beta}(u_1)=\mu(p)vu_0v^* , \quad \tilde{\beta}(v)= \nu v  .$$

\begin{lemma}We have 
\begin{enumerate}
    \item $\tilde{\beta} \circ \tilde{\alpha}_g = \tilde{\alpha}_g \circ \tilde{\beta}$
    \item $\tilde{\beta}^2=\mathrm{Ad} (v) \circ  \tilde{\alpha}_{p+z} $
    \item $\tilde{\beta} \circ \tilde{\rho} = \mathrm{Ad}(u_0)  \circ \tilde{\alpha}_{p } \circ \tilde{\rho} \circ \tilde{ \beta} $
    
\end{enumerate}

\end{lemma}

\begin{proof}
1. Easy.

2. Also straightforward to check, using Eq. (\ref{n1}).

3. Similar calculation as in Lemma \ref{frel2}, using Eq. (\ref{n4}).

\end{proof}

Finally, we need to check that $\tilde{\rho}^2 $ has the correct form.

\begin{lemma}
We have $$\tilde{\rho}^2(x)=s^{(0)}xs^{(0)}{}^*+\sum_{g \in G} \limits t_g^{(0)} (\tilde{\alpha}_g \tilde{\rho})(x)t_g^{(0)}{}^*, \forall x \in \mathcal{U} .$$

\end{lemma}

\begin{proof}
It suffices to show that $$ \tilde{\rho}^2(x)s^{(0)}=s^{(0)}x $$
and
$$\tilde{\rho}^2(x)t_g^{(0)}=t_g^{(0)} (\tilde{\alpha_g} 
\tilde{\rho})(x), \quad \forall g \in G.$$
Again, this is easy to check for $x \in \Phi_k(\mathcal{O}_{|G|+1}) $. 
Note that 
$$\tilde{\rho}(u_0)=u_0^*( s^{(1)}s^{(0)}{}^*+\sum_{g \in G} \limits a(g) t^{(1)}_{g-p} u_0t^{(0)}_{g} {}^* ) $$ $$=u_0^*( \tilde{\beta}(s^{(0)})s^{(0)}{}^*+\sum_{g \in G} \limits a(g) \tilde{\beta}(t^{(0)}_{g-p}) u_0t^{(0)}_{g} {}^* )  $$
and
$$\tilde{\rho}(v)=\xi u_0^*\tilde{\beta}(u_0)^*v .$$
Then then calculation is essentially the same as in Lemmas \ref{lon1} and \ref{lon2}, using Eqs. (\ref{n2}), (\ref{n6}), (\ref{n8}), (\ref{n9}), and (\ref{n10}).
\end{proof}

\begin{lemma}
There is a factor closure of $\mathcal{U} $ to which the endomorphisms $\tilde{\alpha}_g $, $\tilde{\rho} $ and $\tilde{\beta} $ all extend.
\end{lemma}
\begin{proof}

This can be shown by a similar argument to Appendix of \cite{MR3827808}.
\end{proof}

Putting this all together, we get the following reconstruction result.

\begin{theorem}\label{rcrthm}
Let $\mathcal{C} $ be a possibly degenerate generalized Haagerup category for $G$, with structure constants $(A,\epsilon)$. 
Let $p \in G \backslash 2G $ and $z \in G_2 $ be given, and let
$(\chi,\mu,\xi,\nu,a(g)) $ be a set of extension data. Then there is a $\mathbb{Z}_2 $-graded extension of $\mathcal{C} $ which realizes the extension data.

\end{theorem}

\begin{remark}\label{trivext}
Suppose the we have extension data such that everything besides $a(g)$ is trivial (note that in particular this implies that $a(g) \in \{ \pm 1 \}, \ \forall g $). Then we don't need a free product, and we can define $\beta $ directly on the original Cuntz algebra by
${\beta}(s)=s, \quad \beta(t_g)= a(g+p)t_{g+p} .$
We can then verify using Eqs. (\ref{n7}), (\ref{n9}) and (\ref{n10}) that
$\beta $ satisfies the appropriate relations,
namely
$$\beta\circ \rho =\alpha_p\circ \rho \circ \beta, \quad 
\beta^2=\alpha_{p+z}.$$

A necessary condition for this situation to occur is that $\epsilon_k(p)=1$ for all $k\in G_2$. 
Indeed, assume $u=1$. 
Then $\beta\circ \rho \circ \beta^{-1}=\alpha_p\circ \rho$, 
and $\beta(t_0)$ is a multiple of $t_p$. 
Thus for all $k\in G_2$, we get 
$$\alpha_k(\beta(t_0))=\beta(\alpha_k(t_0))=\beta(t_0),$$
which shows $\epsilon_k(p)=1$. 

This will be useful later when we look at the Asaeda-Haagerup categories.
\end{remark}

\subsection{Equivalence}
We have seen that we can describe an extension in terms of extension data. We would like to know when two sets of extension data describe equivalent extensions. 

Suppose we have two extensions, each of the form discussed above, for the same generalized Haagerup category $\mathcal{C} $ with structure constants $ (A,\epsilon)$. Then by Theorem \ref{uniqueness} and the discussion at the end of Section \ref{outsection}, for the purposes of comparing extension data up to unitary equivalence, we may assume without loss of generality that both extensions are realized in the same $\text{End}_0(M) $, with the same group action $\alpha$, but with the choices for $\rho $ possibly differing by an inner perturbation by a unitary fixed by $\alpha$, and with possibly different choices for $\beta $.   

We can easily show that if we replace $\rho$ with $\mathrm{Ad}(w) \circ \rho$, where $w$ is a unitary fixed by $\alpha_g$, the extension data do not change at all. 

So what remains is to check how the choice of $ \beta$ affects the extension data.
 There are two ways we could modify $ \beta$ and still describe the an equivalent extension.

First, we can replace $\beta $ by a different representative of the same isomorphism class $[\beta']=[\beta] $, i.e. $$\beta'=\mathrm{Ad}(w) \circ \beta $$ for some unitary $w$. To keep the relation
$$\alpha_g \circ \beta' = \beta' \circ \alpha_g$$
we require that the $\alpha_g $ act as scalars on $w$, meaning there is a character $\zeta \in \hat{G}$ with 
$$\alpha_g(w)=\zeta(g)w, \quad \forall g \in G .$$
In this case 
we can take $$ u'=w u \rho(w)^* \in (\alpha_p \rho \beta', \beta' \rho ),  \quad \quad v'=w \beta(w) v \in (\alpha_{p+z},\beta'^2 ) $$
as the unitaries for the extension.
Then we have
\begin{enumerate}
\item 
$$\alpha_g(u')=\alpha_g(w u \rho(w)^*)$$ $$=\zeta(g)^2\chi(g) w u \rho(w)^* =\zeta(g)^2\chi(g) u'$$
\item
    $$\alpha_g(v')=\alpha_g(w \beta(w) v )$$ $$=\zeta(g)^2\mu(g) w \beta(w) v=\zeta(g)^2\mu(g) v' $$
    \item
     $$\beta'(v')=\beta'(w \beta(w) v )$$ $$=\mathrm{Ad}(w)(\beta(w \beta(w) v ))$$ $$=w(\beta(w)\beta^2(w)    \beta(v))w^*$$ $$= \nu w \beta(w) (v\alpha_{p+z}(w)   v^*) v w^*$$ $$=\nu \zeta(p+z)w \beta(w) v= \nu \zeta(p+z) v'  $$
     \item
   $$\rho(v')=\rho(w\beta(w) v )=\rho(w) \rho\beta(w) \rho(v)$$ $$=\xi \rho(w)\alpha_{-p} \mathrm{Ad}(u^*)\beta \rho (w) u^* \beta(u)^*v $$ $$=\xi \zeta(p) \rho(w) u^* \beta(\rho(w)u^* ) v  $$ $$=\xi \zeta(p) \rho(w) u^*w^* (w\beta(\rho(w)u^* )w^*)w v $$ $$=\xi \zeta(p)u'^* \beta'(u'^*) \beta'(w)w v  =\xi \zeta(p)u'^* \beta'(u'^*) v'   $$
   \item
   $$\rho(u')=\rho(w)\rho(u)\rho^2(w)^* $$ $$=\rho(w)u^*\beta(s)s^*\rho^2(w)^*+\sum_{g \in G}\limits a(g) \rho(w)u^*\beta(t_{g-p})ut_g^* \rho^2(w)^*
    $$ $$=(\rho(w)u^*w^*) (w\beta(s)w^*) (ws^*\rho^2(w)^*)$$ $$+\sum_{g \in G}\limits a(g) (\rho(w)u^*w^*)(w\beta(t_{g-p})w^*) ( w u t_g^* \rho^2(w)^*) $$ $$=u'^* \beta'(s) s^*+\sum_{g \in G} \limits a(g) u'^* \beta'(t_{g-p}) wu \alpha_g\rho(w^*)t_g^*  $$ $$=u'^* \beta'(s) s^*+\sum_{g \in G} \limits a(g) \zeta(g) u'^* \beta'(t_{g-p}) u't_g^*  $$
\end{enumerate}
Therefore $\chi $ and $ \mu$ are each multiplied by $\zeta^2$, $\xi $ is multiplied by $\zeta(p)  $, $\nu $ 
 is multiplied by $\zeta(p+z)  $, and $a(g) $ is multiplied by $\zeta(g) $.

Second, we can replace $\beta $ by a different object $\beta' $ in the extension which satisfies the same initial assumptions as $\beta$. 
This means that 
$$[\beta']=[\alpha_k\beta] $$
for some $k \in G $, and since 
$$[\beta^2]=[\alpha_{p+z}]= [\beta'^2] =[ \alpha_{2k}\beta^2],$$
we have $$[\alpha_{2k}]=[\mathrm{id}] ,$$
which implies that $k \in G_2 $. On the other hand, for any $k \in G_2 $, we have
$$[(\alpha_k \beta)^2]=[\beta^2]=[\alpha_{p+z}], \quad [\alpha_k  \beta \rho ]=[\alpha_k \alpha_p \rho \beta  ]=[\alpha_p \rho\alpha_{k} \beta ] .$$
Thus $\alpha_k \beta $ satisfies the same assumptions as $\beta $.

In this case, we can still take $u$ as our intertwiner for $(\alpha_p \rho \beta' ,\beta' \rho) $ and $v$ as our intertwiner for $(\alpha_{p+z},\beta'^2 ) $. 
Thus $\chi $ and $\mu $ remain unchanged. On the other hand, we have
\begin{enumerate}
\item
$$\beta'(v)=(\alpha_k \beta)(v) =\alpha_k(\nu v)= \mu(k) \nu v .$$
\item
$$\rho(v)=\xi u^* \beta(u)^*v=\xi u^* (\alpha_k \beta')(u)^*v=\xi \chi(k) u^* \beta'(u)^* v$$
\item
$$\rho(u)=u^*\beta(s)s^*\rho^2(w)^*+\sum_{g \in G}\limits a(g) u^*\beta(t_{g-p})ut_g^*$$ $$=u^*\beta'(s)s^*\rho^2(w)^*+\sum_{g \in G}\limits \epsilon_k(g-p)a(g) u^*\beta'(t_{g-p})ut_g^* $$
\end{enumerate}
Thus $\nu $ is multiplied by $\mu(k) $ and $\xi $ is multiplied by $\chi(k) $. 

For $a(g)$, we need to first normalize the new $a(g)$ by replacing $u$ with $-u$ if necessary, and so $a(g) $ is multiplied by $\epsilon_k(g-p)\epsilon_k(-p)$ in the extension data corresponding to $\beta' $.

Putting this all together, we get the following description of equivalence.

 \begin{theorem} \label{equiv}
 Let $\mathcal{C} $ be a generalized Haagerup category with structure constants $(A,\epsilon) $, and fix $p \in G \backslash 2G $ and $z \in G_2 $.
 Let $(\chi,\mu,\xi,\nu,a(g)) $ and $(\chi',\mu',\xi',\nu',a'(g)) $ be two sets of extension data for $(\mathcal{C},A,\epsilon,p,z) $. Then the corresponding extensions are unitarily equivalent iff there is a character $\zeta \in \hat{G} $ and an element $k \in G_2 $ such that 
 $$\chi'=\zeta^2 \chi, \quad \mu'=\zeta^2 \mu, $$
 $$ \xi'=\zeta(p)\chi(k)\xi,  
 \quad \nu'= \zeta(p+z) \mu(k) \nu$$
 $$a'(g)=\zeta(g)\epsilon_k(g-p)\epsilon_k(p) a(g) $$
 \end{theorem}
\begin{proof}
First note that since the extension data completely determine the $6j$-symbols of the extension, any two extensions which share the same extension data are equivalent.

Now, as we have seen, once $(A,\epsilon) $ is fixed, the only freedom we have for the extension data is the choice of $\beta $, which leads to the relations above.

Conversely, for any character $\zeta $, we can find a unitary $w$ in $M$ such that $\alpha_g(w)=\zeta(g)w $. Therefore we can always vary the extension data by the given relations.
\end{proof}

\begin{remark}
In the degenerate case, where the action of $\alpha $ is not outer, we may not be able realize every character $\zeta$. 
\end{remark}


In the rest of this section, we assume that $A_g(h,k)\neq 0$ for all $g,h,k\in G$, which is true for every known example. 
In this case, $\epsilon$ is a bicharacter on $G_2\times G$. 
Let $(\chi,\mu,\xi,\nu,a)$ and $(\tilde{\chi},\tilde{\mu},\tilde{\xi},\tilde{\nu},\tilde{a})$ be two extension data, and let $b(g)=\tilde{a}(g)/a(g)$. 
Then Eq. (\ref{n10}) shows that $b$ is a character, and we have $$\tilde{\chi}(g)=b(g)^2\chi(g),\quad \tilde{\mu}(g)=b(g)^2\mu(g) $$ $$ \tilde{\xi}=b(p)\xi, \quad \tilde{\nu}=\pm b(p+z)\nu.$$ 
Therefore, to determine the number of extensions with fixed $(p,z)$, Theorem \ref{equiv} shows that we can fix $a$, and in consequence $\chi,\mu$, and $\xi$ too. 
Now the only remaining freedom is multiplying $\nu$ by $-1$. 

Letting $\zeta(g)=\epsilon_k(g-p)\epsilon_k(p)=\epsilon_k(g)$ in Theorem \ref{equiv}, we get $$\nu'=\epsilon_k(p+z)\mu(k)\nu=\epsilon_k(z)\chi(k)\mu(k).$$ 
Let $\tau$ be a character of $G$ defined by 
$$\tau(g)=\frac{\mu(g)\epsilon_z(g)}{\chi(g)}=\frac{a(g-p)a(p)}{a(g)}\epsilon_{-p}(g)\epsilon_{-p}(p).$$
Then since $\chi^2=\mu^2$, we have $\tau^2=1$. 
Now we have 
$$\nu'=\tau(k)\epsilon_{z}(k)\epsilon_k(z)\nu.$$
Note that we always have $\tau(p)=1$, and in fact $\tau$ is trivial for every known example. 

In summary, we get the following classification.

\begin{corollary}\label{coreq}
Assume that there exists extension data for $(\mathcal{C}, A, \epsilon, p,z)$, where  $A_g(h,k)\neq 0$ for all $g,h,k\in G$. 
Then the number of equivalence classes of such extensions  
is 2 if $\tau(k)\epsilon_k(z)\epsilon_z(k)=1$ for all $k\in G_2$, 
and it is 1 if there exists $k\in G_2$ with $\tau(k)\epsilon_k(z)\epsilon_z(k)=-1$. 
\end{corollary}

\begin{corollary}\label{evencyclic} 
Under the assumptions of the above corollary, if $G=\mathbb{Z}_{2n}$, then there exists either 0 or 2 extensions 
for a given $(p,z)$. 
\end{corollary} 

\begin{proof}
In the case of $G=\mathbb{Z}_{2n}$, we have $G_2=\{0,n\}$, and we may always assume $p=1$. 
Since $p$ generates $G$ in this case, $\tau(p)=1$ implies that $\tau$ is trivial. 
Now we have $\tau(k)\epsilon_k(z)\epsilon_z(k)=1$ for every combination of $z$ and $k$. 
Thus there exist exactly 2 extensions for $(1,z)$ once extension data exists. 
\end{proof}

\begin{remark} In our situation, we have 
$$H^2(\mathbb{Z}_2, \mathrm{Inv}(\mathcal{Z}(\mathcal{C})))=H^2(\mathbb{Z}_2, G_2)=G_2,$$
$$H^3(\mathbb{Z}_2,\mathbb{C}^\times)=\mathbb{Z}_2,$$
$$H^1(\mathbb{Z}_2, \mathrm{Inv}(\mathcal{Z}(\mathcal{C})))=H^1(\mathbb{Z}_2, G_2)=\mathrm{Hom}(\mathbb{Z}_2, G_2)=G_2.$$
As in the argument at the end of subsection 5.2, we can see that $z$ corresponds to the element in $H^2(\mathbb{Z}_2, \mathrm{Inv}(\mathcal{Z}(\mathcal{C})))$ in Theorem \ref{ENOth}, and $\tau(k)\epsilon_k(z)\epsilon_z(k)$ corresponds to $p^1_{(c,M)}(k)$ in Theorem \ref{DNth} if $H^3(\mathbb{Z}_2,\mathbb{C}^\times)$ is identified with $\{1,-1\}$.     
\end{remark}

\section{Examples}
\subsection{Cyclic groups}
For an even cyclic group $G=\mathbb{Z}_{2n}$, there are two possible bicharacters on $G_2 \times G=\mathbb{Z}_2 \times\mathbb{Z}_{2n} $, namely the trivial one and $\epsilon_{n}(m)=(-1)^m $.

For all known examples, $\epsilon $ restricts to the nontrivial bicharacter. In particular, there are examples known for each $n \leq 5$ such that $[1]+[\alpha_g \rho] $ admits a Q-system for each $ g \in G$, with two different examples each for $n=3,5$. The Q-systems comprise two orbits under the action of the inner automorphism group of $ \mathcal{C}$, corresponding to whether $ g$ is even or odd. 

It is natural to wonder whether the two orbits are transposed by an outer automorphism of $\mathcal{C} $, and this is indeed the case for all of the the known examples (note that $ H^2(\mathbb{Z}_{2n},\mathbb{T})$ is trivial, so the cocycle-free condition is automatic). It is then natural to ask whether these outer automorphisms realize $\mathbb{Z} _2$-graded extensions of the fusion categories.

We therefore consider extension data for $p=1$. We have $z\in G_2=\{0,n \} $. 

Then we have
$$\chi(p)^m=\chi(1)^n=\chi(n)=\epsilon_n(1)=-1 ,$$

so $\chi(1) $ is a primitive $2n^{th}$ root of unity. We then have $\xi^{2n} =-1$. Then $$\mu(1)= \chi(1) \chi(z)  .$$ From Eq. (\ref{n8}) we then have 
$$a(g)a(g-1)\xi=\epsilon_{p+z}(g-2p)\mu(g)$$ $$=\epsilon_1(g-2)\epsilon_z(g)\xi^{2g} \chi(z)^g=\epsilon_1(g-2)\xi^{2g}  .$$

If we fix $\epsilon_1(2n-1)=-1 $ and $\epsilon_1(g)=1 $ otherwise,
this gives the unique solution
$$a(g)=-\xi^g, \quad 1 \leq g \leq 2n-1 .$$

We can then check Eq.(\ref{n7}) and (\ref{n9}) (which only depend on $\epsilon $) hold, and what remains is to check Eq. (\ref{n10}) using the structure data $A_g(h,k) $. Note that Eq. (\ref{n10}) does not depend on $z$.


\begin{theorem}
For each of the known examples of generalized Haagerup categories for $G=\mathbb{Z}_{2n} , \quad 1 \leq n \leq 5 $, and each odd $p$  and $z \in \{0,n \} $, there are two distinct $\mathbb{Z}_2 $-graded extensions of the form discussed in the previous section.
\end{theorem}
\begin{proof}
We check Eq. (\ref{n10}) with a computer. Then by Corollary \ref{evencyclic}, in each case there are two distinct extensions up to equivalence.
\end{proof}
\begin{remark}
In this paper we are concerned with classifying extensions up to the natural notion of equivalence, but one can also ask whether different extensions give distinct tensor categories. For $\mathbb{Z}_2 $-extensions of generalized Haagerup categories, there is a unique nontrivial homogeneous component, so the only way two different extensions can be tensor equivalent is if they are related by a nontrivial automorphism of the trivial component (that is, of the generalized Haagerup category).

Note that the choice of $z \in \{0,n \}$ for a generalized Haagerup category for an even cyclic group is an invariant of the tensor category (indeed, of the fusion rules) of the extension. It is less clear whether the sign choice in $\nu $ in the extension data is also an invariant of the tensor category.

One can check that once one fixes an extension as above, the extension data is invariant under conjugation by the $\alpha_g$, as well as conjugation by $\beta $. Thus if the outer automorphism group of the generalized Haagerup category is generated by conjugation by $\beta $, then the different extensions are also distinct as tensor categories. This is the case for the generalized Haagerup category for $\mathbb{Z}_4 $. 

Thus at least for $\mathbb{Z}_4$, the $\mathbb{Z}_2 $-graded extensions constructed above give four different fusion categories, and we conjecture that this holds in general for $\mathbb{Z}_{2n} $. 
\end{remark}



\subsection{Asaeda-Haagerup categories}
The Asaeda-Haagerup subfactor was one of the two original ``exotic'' subfactors discovered in \cite{MR1686551} (the other being the Haagerup subfactor, corresponding to a generalized Haagerup category for $ \mathbb{Z}_3$).
It was shown in \cite{MR3859276} that
there are exactly six fusion categories in the Morita equivalence class of the Asaeda-Haagerup categories. Three of these, including the two which are the even parts of the Asaeda-Haagerup subfactor, do not admit any outer automorphisms. The other three are quadratic categories, and one of these, called $\mathcal{AH}_4 $, is a de-equivariantization of a generalized Haagerup category for the group $G=\mathbb{Z}_4 \times \mathbb{Z}_2$. 

The category $\mathcal{AH}_4 $ may be considered a degenerate generalized Haagerup category, coming from a solution to Eq. (\ref{e1})-(\ref{e10}) for $G$ with $\epsilon_{(0,1)}((i,j))=1 $ for all $(i,j) $, which means that $\alpha_{(0,1)} $ acts trivially on the Cuntz algebra (and hence is equal to $\mathrm{id} $). Thus we have $\mathrm{Inv}(\mathcal{AH}_4)\cong \mathbb{Z}_4 $. 

There are $8$ non-isomorphic Q-systems of the form $[\mathrm{id}]+ [\alpha_g \rho] $, which fall into $4$ inner conjugacy classes, with conjugation by $\alpha_{(2,0)} $ acting trivially on the set of Q-systems. The Brauer-Picard group is $$\mathrm{BrPic}(\mathcal{AH}_4)\cong \mathrm{Out}(\mathcal{AH}_4)\cong \mathbb{Z}_2 \times \mathbb{Z}_2$$ and  acts transitively on the inner conjugacy classes of Q-systems. (The Brauer-Picard group had previously been calculated for the original Asaeda-Haagerup categories in \cite{MR3449240} using other methods).

Therefore it is natural to wonder whether $\mathcal{AH}_4 $ can be extended by $\mathrm{Out}(\mathcal{AH}_4)$, and consequently whether all of the Asaeda-Haagerup categories admit $\mathbb{Z}_2 \times \mathbb{Z}_2$-graded extensions.

In \cite{MR3354332} it was shown on abstract grounds that the obstructions for $\mathbb{Z}_2 $-extensions vanish - but those methods do not determine the obstructions for $\mathbb{Z}_2 \times \mathbb{Z}_2$-extensions.

We will show that the $\mathbb{Z}_2 \times \mathbb{Z}_2$ obstructions vanish by directly constructing a $\mathbb{Z}_2 \times \mathbb{Z}_2$-extension using the methods above.

We refer to \cite[Section 4]{MR3859276} for the structure constants $(A,\epsilon)$ of the category $\mathcal{AH}_4$, and note that the bicharacter $\epsilon $ on $G_2 \times G $ is given by
$$\epsilon_{(i,j)}((k,l))= \begin{cases} -1 & (i,l)=(2,1)\\
1 & \text{otherwise} \end{cases}.$$

We will consider extensions for each of $p=(1,0)$ and $p=(0,1) $.

We start with $p=(1,0) $, and let $z=(0,0) $. Note that $\epsilon_k(p)=1 $ for all $k \in G_2 $, so by Remark \ref{trivext}, there is a possibility of realizing an extension on the original Cuntz algebra.

Up to equivalence, there are two solutions for $a(g)$ in Eqs. (\ref{e1})-(\ref{e9}), exactly one of which also solves Eq. (\ref{e10}) (this was checked with Mathematica).

We fix the extension data as $a((0,1))=-1 $ and $a(g)=1 $ otherwise, and then $\nu=\pm 1 $ determine two inequivalent extensions.

We then have $\chi=\mu=\xi=1 $, and if $ \nu=1$ as well, we can represent the extension on the Cuntz algebra $\mathcal{O}_9 $.



Next, we consider $p=(0,1) $ and $z=(0,0) $. 
In this case we have $\epsilon_{(2,0)}(p)=-1 $, so there is no hope of realizing an extension on the original Cuntz algebra.
We find again a unique solution for $a(g)$ up to equivalence:
$$a'((x,y ) )=e^{\frac{x \pi i}{4}} ,$$ 
and again a sign choice in $\nu' $ gives two different extensions.

\begin{remark}
 For each $ p$, we have chosen $z=(0,0) $, and found corresponding extensions. Since $\mathrm{Inv}(\mathcal{Z}(\mathcal{AH}_4)) $ is trivial, there can be at most one quasi-tensor product for a given choice of $p$, so we cannot have additional extensions for other choices of $z$.
 
For example, for $p=(1,0) $ and $z=(2,0)$, there is a solution to Eq. (\ref{e1})-(\ref{e9}), but it does not satisfy (\ref{e10}). 
\end{remark}

We would now like to realize extensions for $p=(1,0) $ and $p=(0,1) $ simultaneously.

Let $ \mathcal{U}=\mathcal{O}_9 *\mathcal{O}_9 *C^*(\mathbb{F}_3)$. We define an automorphism $\tilde{\beta}' $ using $a'(g) $ and a choice of sign for $\nu' $ as in the proof of Theorem \ref{rcrthm}.

We can also define $\beta $ on $\mathcal{O}_9 $ using $a(g) $ with $\nu=1 $.

We now want to extend $\beta $ to $ \mathcal{U}$. We need to preserve the relations 
$$\beta \circ \rho = \alpha_{(1,0)} \circ \rho \circ \beta, \quad \beta \circ \alpha_g = \alpha_g \circ \beta, \quad \beta^2 = \alpha_{(1,0)}$$
which hold on the Cuntz algebra, and we would also like the extension of $\beta $ to commute with $\tilde{\beta}' $.
So we define
$$\tilde{\beta}(\Phi_0(x))=\tilde{\beta}(\Phi_1(x))=\beta(x), \quad  \forall x \in \mathcal{O}_9,$$ $$\quad  \tilde{\beta}(u_0)=cu_0, \quad \tilde{\beta}(u_1)=c u_1, \quad  \tilde{\beta}(v)=c' v .$$

 Then we have $$\tilde{\beta} \circ \tilde{\alpha}_g = \tilde{\alpha}_g \circ \tilde{\beta}, \quad \tilde{\beta} \circ \tilde{\beta}'=\tilde{\beta}' \circ \tilde{\beta} $$
and
$$\tilde{\beta}^2 =\tilde{\alpha}_{(1,0)} $$
if $$c^2=\chi'((1,0))=i, \quad c'^2=\mu'((1,0))=i ,$$
which will now assume.
\begin{lemma}\label{l11}
We have
$$\tilde{\beta} \circ \tilde{\rho} = \tilde{\alpha}_{(1,0)} \circ \tilde{\rho}\circ  \tilde{\beta}$$
if
$$c a'(g-(1,0))a(g-(0,1))a(g)=a'(g), \quad \forall g \in G $$
\end{lemma}
\begin{proof}

It is straightforward to check that the relations hold for $ x \in \Phi_k(\mathcal{O}_9)$.
For $u_0 $ we have
$$ (\tilde{\beta} \circ \tilde{\rho}) (u_0)=\tilde{\beta} (u_0^*( s^{(1)}s^{(0)}{}^*+\sum_{g \in G} \limits a'(g) \tilde{\beta}'({t^{(1)}_{g-(0,1)}}) u_0t^{(0)}_{g} {}^* ) )$$
$$=\overline{c} u_0^* ( s^{(1)}s^{(0)}{}^*+\sum_{g \in G} \limits c a'(g) a(g-(0,1)+(1,0))a(g+(1,0)) \tilde{\beta}'( t^{(1)}_{g-(0,1)+(1,0)}) u_0t^{(0)}_{g+(1,0)} {}^* ) $$
and
 $$ (\tilde{\alpha}_{(1,0)} \circ \tilde{\rho}  \circ \tilde{\beta})(u_0)$$ $$=c \chi'(-(1,0) )(u_0^*( s^{(1)}s^{(0)}{}^*+\sum_{g \in G} \limits a'(g) \tilde{\beta}'(t^{(1)}_{g-(0,1)}) u_0t^{(0)}_{g} {}^* ) ) ,$$
which are equal if 
the relation holds. 

Next we have
$$( \tilde{\beta} \circ \tilde{\rho} )(v)=\tilde{\beta}(\xi u_0^*u_1^*v )=\xi' \overline{c}^2c' u_0^*u_1^*v$$
and
$$( \tilde{\alpha}_{(1,0)}\circ \tilde{\rho} \circ \tilde{\beta}) (v)=\overline{\mu'((1,0))} c' \xi' u_0^*u_1^*v ,$$
which are equal if 
$$\mu'((1,0))=c^2=\chi'((1,0)) ,$$
which is true.
\end{proof}

\begin{theorem}\label{ahthm}
The obstruction in $H^4(\mathbb{Z}_2 \times \mathbb{Z}_2, \mathbb{T}) $ for the existence of $\mathbb{Z}_2 \times \mathbb{Z}_2 $-graded extensions of the Asaeda-Haagerup categories by mutually inequivalent bimodule categories vanishes.
\end{theorem}
\begin{proof}
We can verify that the relation in Lemma \ref{l11} holds for $c=e^{-\frac{3\pi i}{4} }$. Then we can simultaneously extend $\tilde{\beta} $ and $\tilde{\beta}' $ to a factor closure, as in the proof of Theorem \ref{rcrthm}. Since such an extension exists, the obstruction must vanish.
\end{proof}
Since the homotopy type of the Brauer-Picard 2-group is an invariant of Morita equivalence, we get corresponding extensions of all the Asaeda-Haagerup categories.
\begin{corollary}
For each of the Asaeda-Haagerup fusion catgeories, there exist 8 different $\mathbb{Z}_2 \times \mathbb{Z}_2 $-graded extensions of the Asaeda-Haagerup categories by mutually inequivalent bimodule categories.
\end{corollary}
\begin{proof}

Since $\mathrm{Inv}(\mathcal{Z}(\mathcal{C})) $ is trivial, so are $H^n(\mathbb{Z}_2\times \mathbb{Z}_2,\mathrm{Inv}\mathcal{Z}(\mathcal{C}))$ for all $n$.  
Thus there is no choice of quasi-tensor product. By Theorem \ref{ahthm}, the obstructions for extensions vanish. 
Therefore Theorem \ref{ENOth} and Theorem \ref{DNth} show that the set of extensions form a torsor over $H^3(\mathbb{Z}_2 \times \mathbb{Z}_2, \mathbb{T}) \cong (\mathbb{Z}_2)^3$.

\end{proof}

Note that unlike for $\mathcal{AH}_4 $, the group $\mathrm{Out}(\mathcal{C})$ is trivial for $\mathcal{C}=\mathcal{AH}_{1,2,3} $. Therefore the corresponding extensions for those categories are not quasi-trivial, but rather involve bimodule categories that are non-trivial even as module categories (see \cite{MR3449240} and the accompanying text files for a description of these bimodule categories, including dimensions of simple objects and fusion rules).








\begin{conjecture}
Similar $\mathbb{Z}_2 \times \mathbb{Z}_2$-graded extensions exist for generalized Asaeda-Haagerup categories (de-equivariantizations of generalized Haagerup categories for the groups $\mathbb{Z}_{4m} \times \mathbb{Z}_2 $ with $\epsilon_{(0,1)} $ trivial).
\end{conjecture}
For specific values of $m$ the conjecture can in theory be checked by a similar calculation as above - namely, try to find extension data for $p=(1,0), z=(0,0) $ with trivial $\chi$,$\mu$,$\xi $; then for $p=(0,1), z=(0,0) $; then check the relation in Lemma \ref{l11}. However, genralized Asaeda-Haagerup categories are themselves not yet known to exist for $m > 1 $.


\subsection{The group $\mathbb{Z}_2 \times \mathbb{Z}_2$}
It was shown in \cite{MR3827808} that there is a unique generalized Haagerup category $\mathcal{C} $ for $G=\mathbb{Z}_2 \times \mathbb{Z}_2 $. This category is related to a conformal inclusion $SU(5)_5 \subset \mathrm{Spin}(24) $; see \cite{MR3764563,2102.09065}. It was shown in \cite{MR4001474} that the Brauer-Picard group of this category has order $360$, and the group was identified as $S_3 \times A_5 $ in \cite{2102.09065}. The outer automorphism subgroup is $A_4 $.

We would like to classify the quasi-trivial graded-extensions of $\mathcal{C} $, and in particular find $A_4 $-extensions by the entire outer automorphism group. In this subsection we first consider the $\mathbb{Z}_2 $-extensions.

We will use the normalization $$\epsilon_g(h)=\left( \begin{array}{cccc}
1 & 1 & 1 & 1 \\
1 & -1 & -1 & 1 \\
1 & 1 & -1 &- 1\\
1 & -1 & 1 & -1
\end{array} \right),$$
as in \cite{MR3827808} (corresponding to $s=-1 $ there). We label the elements of the group by $\{0,p,q,r\} $ (in that order with respect to the matrices of structure constants). We consider extensions by an automorphism $\beta $ which conjugates $\rho $ to $\alpha_p\rho $.

Then Eq. (\ref{n7}) reduces to 
$$\chi(g)=\epsilon_g(p) ,$$
so $\chi $ is given by the second column of the $\epsilon $ matrix, $(1,-1,1,-1 )$, and then $\xi=\pm i $.

From Eq. (\ref{n8}) we have 
$$a(p)=\bar{\xi}, \quad a(q)a(r)=a(p)\mu(r)\epsilon_{p+z}(r)=a(p)\mu(q)\epsilon_{p+z}(q)   $$
By Theorem \ref{equiv}, without loss of generality we can assume
$$a=(1,-i,t,i) $$
for some
$$t=\pm1=-\mu(q)\epsilon_{p+z}(q)=-\mu(r)\epsilon_{p+z}(r). $$
Checking Eq. \ref{n10} with a computer (or by hand)
gives $t=1$.
Then we have
$$\mu(q)=\epsilon_{p+z}(q)=\chi(q)\epsilon_{p+z}(q)\epsilon_{q}(p)=\chi(q)\epsilon_{z}(q) $$ and similarly  
$$ \mu(r)=\chi(r)\epsilon_z(r).$$ 
Note that we also have $$\mu(0)=\chi(0)\epsilon_z(0), \quad \mu(p)=\chi(p)\epsilon_z(p) $$ 
by Eq. (\ref{n4}).
So
$$\mu = \chi \epsilon_z .$$

Thus $\tau$ in Corollary \ref{coreq} is trivial, and the number of extensions are determined by whether $\epsilon_k(z)\epsilon_z(k)$ can take -1 or not. 
We also have 
$$\nu^2=\mu(p+z)=-\epsilon_z(z).$$

It was shown in \cite{MR3827808} that there is a $\mathbb{Z}_3 $-action on $\mathcal{C} $ which fixes $\rho $ and cyclically permutes $\{ \alpha_p,\alpha_q,\alpha_r\} $. Therefore, similar extensions exist for automorphisms taking $\rho $ to $\alpha_q \rho $ and $\alpha_r \rho$.

Summarizing, we have:

\begin{theorem}\label{z2eqs}
For each $x \in \{p,q,r \} $ and $y \in G $, there is a $\mathbb{Z}_2 $-extension of $\mathcal{C} $ by an automorphism $\beta_{x,y} $ such that $[\beta_{x,y} \rho] = [\alpha_x \rho \beta_{x,y} ]$ and $[\beta_{x,y}^2]=[\alpha_{x+y}] $. Such an extension is unique unless $y=0$, and there exist exactly two extensions for $y=0$.  
\end{theorem}

We defer the general case of quasi-trivial extensions of $\mathcal{C} $ by outer automorphisms to a separate section, since the argument is long and involved.


\section{Quasi-trivial extensions of the generalized Haagerup category for $\mathbb{Z}_2 \times \mathbb{Z}_2 $}

At the end of the previous section we classified quasi-trivial $\mathbb{Z}_2 $-extensions of the generalized Haagerup category for $\mathbb{Z}_2 \times \mathbb{Z}_2 $. In this section we will consider more generally extensions of this category by arbitrary subgroups of the outer automorphism group $A_4 $.

Throughout this section, let $\mathcal{C}$ be the generalized Haagerup category for $\mathbb{Z}_2 \times \mathbb{Z}_2 $, realized in standard form in $\mathrm{End}_0(M) $. We label the group as $G=\{0,p,q,r\}$, and use the same normalization of $\epsilon $ and $A$ as in the previous section.
\subsection{Constraints for a $\mathbb{Z}_2 \times \mathbb{Z}_2 $-extension}
We first consider $ \mathbb{Z}_2 \times \mathbb{Z}_2 $-extensions.

Let us first assume that we have a $ \mathbb{Z}_2 \times \mathbb{Z}_2 $-extension, also realized in $\mathrm{End}_0(M) $, generated by automorphisms $\beta_p $ and $\beta_q $ such that $[\beta_h \rho \beta_h^{-1}]=[\alpha_h \rho] $ for each $h \in \{p,q \}$.
Let $\beta_r=(\beta_p \beta_q)^{-1} $. Then we have
$$[\beta_r \rho \beta_r^{-1}]=[\beta_q^{-1} \beta_p^{-1} \rho \beta_p \beta_q]=[\alpha_r \rho] .$$
We will denote the corresponding unitaries and extension data using subscripts, e.g. $u_h$, $v_h $, $\xi_h $, $a_h$ etc., for $h \in \{p,q,r \} $. 

Then as we have seen, we can without loss of generality assume that $$a_p=(1,-i,1,i) $$ and similarly $$a_q=(1,i,-i,1 ) .$$ Then we have $$a_r=(1,ts ,ti , -si ), $$ where $t$ and $s$ are signs.

\begin{lemma}\label{prevlem}
We have 
$$a_r=(1,1,i,-i) .$$
Also, 
$$\beta_p(u_q)u_p= \beta_r^{-1}(u_r^{*}),$$
and similarly for cyclic permutations of $(p,q,r) $.
\end{lemma}
\begin{proof}
We have $$\mathrm{Ad}(\beta_r^{-1}(u_r^*)) \circ \alpha_r \rho=\beta_r^{-1} \rho \beta_r = \beta_p \beta_q \rho \beta_q^{-1} \beta_p^{-1}$$ $$=\beta_p \circ \mathrm{Ad}(u_q) \circ \alpha_q \rho \beta_p^{-1}=\mathrm{Ad}(\beta_p(u_q)\alpha_q(u_p)) \circ \alpha_{p+q} \rho$$ $$=\mathrm{Ad}(\beta_p(u_q)u_p)\circ \alpha_{r} \rho  ,$$
which implies that $$\beta_p(u_q)u_p =b_p \beta_r^{-1}(u_r^*)$$
for some unitary scalar $b_p $.

Consider the action of $\rho $ on this identity. We have $$b_p=\beta_r^{-1}(u_r)\beta_p(u_q)u_p=\rho(\beta_r^{-1}(u_r)\beta_p(u_q)u_p)$$ $$=(\beta_r^{-1} \beta_r \rho \beta_r^{-1})(u_r) u_p^*\alpha_p(\beta_p(\rho(u_q))) u_p \rho(u_p)$$ $$=  \epsilon_p(q) \beta_r^{-1} (u_r \alpha_r \rho(u_r)u_r^* ) u_p^* \beta_p(u_q^*)\beta_p(u_q \rho(u_q))u_p \rho(u_p)$$
$$=  \epsilon_p(q)\epsilon_r(r) \overline{b_p} \beta_r^{-1} (u_r \rho(u_r)u_r^*) \beta_r^{-1}(u_r) \beta_p(u_q \rho(u_q))u_p \rho(u_p)$$
$$=  \overline{b_p} \beta_r^{-1} (u_r \rho(u_r))  \beta_p(u_q \rho(u_q))u_p \rho(u_p).$$

So we have:

$$b_p^2=\beta_r^{-1} (u_r \rho(u_r)) \beta_p(u_q \rho(u_q))u_p \rho(u_p) $$
$$=\beta_p \beta_q (\beta_r(s)s^*+\sum_{g \in G} \limits a_r(g) \beta_r(t_{g-r}) u_r t_g^*) $$ $$\cdot \beta_p(\beta_q(s)s^*+\sum_{h \in G} \limits a_q(h) \beta_q(t_{h-q}) u_q t_h^*)$$ $$\cdot(\beta_p(s)s^*+\sum_{k \in G} \limits a_p(k) \beta_p(t_{k-p}) u_p t_k^*) $$
$$=(\beta_p\beta_q\beta_r)(s)s^*$$ $$+\sum_{g \in G}\limits a_p(g)a_q(g-p)a_r(g-p-q) (\beta_p\beta_q\beta_r)(t_{g-p-q-r})     (\beta_p\beta_q(u_r))\beta_p(u_q)u_p t_g^* $$
$$=ss^*+\sum_{g \in G}\limits a_p(g)a_q(g-p)a_r(g-r) b_p t_g     t_g^* .$$
So we get $$b_p^2=a_p(g)a_q(g-p)a_r(g-r) b_p=1 , \quad \forall g \in G,$$
which implies that $b_p=1 $ and $s=t=1 $, so that $a_r=(1,1,i,-i) $.

This calculation is invariant under cyclic permutations of $(p,q,r) $.

\end{proof}

We record for later use the relation among $a_p $, $a_q$, $a_r$ that we found in the proof of Lemma \ref{prevlem}, which can be verified directly:
\begin{equation}\label{arel}
a_p(g)a_q(g-p)a_r(g-r)=1, \quad \forall g \in G .
\end{equation}

As we have seen previously, each of the $\mathbb{Z}_2 $-graded extensions can be reconstructed from a Cuntz algebra and three unitaries corresponding to $u_h $, $ v_h$, and $\beta_h(u_h)$. For our $\mathbb{Z}_2 \times \mathbb{Z}_2 $-graded extension, we also have to consider the images under the various $\beta_k $ of each $u_h $ and $v_h $. 

A priori, there are 21 unitaries to consider: $$\{ \beta_k(u_h) \}_{h \in  \{p,q,r\}, \ k \in \mathbb{Z}_2 \times \mathbb{Z}_2} \cup \{ \beta_k(v_h) \}_{h \in  \{p,q,r\}, \ k \in \mathbb{Z}_2 \times \mathbb{Z}_2 \backslash \{ h \}} $$
(where we let $\beta_0=\mathrm{id} $; note that we have $\beta_h(v_h)=\nu_h v_h $). We can then use the relations $$\beta_p \beta_q \beta_r=\mathrm{id} \ \mathrm{ and } \ \beta_h^2=\mathrm{Ad}(v_h) \circ \alpha_{h+z_h} $$ to express $\beta(w) $ as a word in these unitaries and their adjoints for any $w$ on this list. Similarly, we can use the relation $$\rho \circ \beta_h=\mathrm{Ad}(u_h^*) \circ \beta_h  \circ \rho \circ  \alpha_h $$ to simplify $\rho(w) $. Thus, the C$^*$-algebra generated by the Cuntz algebra generators and these unitaries is invariant under the $\alpha_g $, $\beta_h $, and $\rho $.


We first show that 6 of these 21 unitaries can be written in terms of the other 15.

\begin{lemma}\label{reduce}
We have:
\begin{enumerate}
    \item $\beta_p(u_q)=-\epsilon_{z_r}(r)v_r^*\beta_r(u_r)^*  v_ru_p^* $
    \item $\beta_p(v_q)=\overline{\nu_q}\mu_q(r+z_r)v_r^*\beta_r(v_q) v_r $
\end{enumerate}
and similarly for other cyclic permutations of $(p,q,r) $
\end{lemma}

\begin{proof}
\begin{enumerate}
\item We have $$\beta_p(u_q)=\beta_r^{-1}(u_r^*)u_p^*=\beta_r^{-2}(\beta_r(u_r^*))u_p^*$$ $$=(\mathrm{Ad}(v_r) \circ {\alpha_{r+z_r}})^{-1}(\beta_r(u_r^*))u_p^*={\chi_r(r+z_r) }v_r^*(\beta_r(u_r^*)) v_ru_p^*, $$
and $$\chi_r(r+z_r)=\epsilon_{r+z_r}(r)=\epsilon_r(r)\epsilon_{z_r}(r)=-\epsilon_{z_r}(r) .$$
\item We have $$ \beta_p(v_q)=\overline{\nu_q}\beta_p \beta_q(v_q)=\overline{\nu_q}\beta_r^{-1}(v_q)=\overline{\nu}\beta_r^{-2} (\beta_r(v_q))$$ $$=\overline{\nu}(\mathrm{Ad}(v_r) \circ {\alpha_{r+z_r}})^{-1} (\beta_r(v_q))=\overline{\nu_q}\mu_q(r+z_r)v_r^* \beta_r(v_q)v_r .$$
\end{enumerate}
\end{proof}

In light of Lemma \ref{reduce}, we only need to consider the 15 unitaries of the form
$u_h$, $\beta_h(u_h)$, $\beta_{h'}(u_h)$, $v_h$, and $\beta_{h'}(v_h) $, where $h'$ is the successor of $h$ in the cyclic ordering $(p,q,r) $. We introduce the notation
$$u_{i}^{(j)}=\beta_j(u_i), \quad i,j \in \{p,q,r \} .$$

We now derive two further relations among these 15 unitaries.

The first one comes from the fact that $$\mathrm{Ad}(v_p^{(q)})= \mathrm{Ad}(\beta_q(v_p))=\beta_q \circ \mathrm{Ad}(v_p)\circ \beta_q^{-1}$$ $$= \beta_q(\alpha_{p+z_p}\beta_p^2)\beta_q^{-1}=\alpha_{p+z_p}\beta_q \beta_p^2\beta_q^{-1}.$$

\begin{lemma}\label{cond1}
The unitary $v_p v_p^{(q)}v_q v_q^{(r)}v_r v_r^{(p)} $ is a scalar.
\end{lemma}

\begin{proof}
It suffices to show that $\mathrm{Ad}(v_p v_p^{(q)}v_q v_q^{(r)}v_r v_r^{(p)})$ is the identity.
By the previous remark, we have 
$$\mathrm{Ad}(v_p v_p^{(q)})=(\alpha_{p+z_p} \beta_p^2) \alpha_{p+z_p} \beta_q\beta_p^2\beta_q^{-1}=\beta_p^2  \beta_q\beta_p^2\beta_q^{-1},$$
with similar formulas for $\mathrm{Ad}(v_q v_q^{(r)})$ and $ \mathrm{Ad}(v_r v_r^{(p)}) $, 
 
so we have
$$\mathrm{Ad}(v_p v_p^{(q)}v_q v_q^{(r)}v_r v_r^{(p)})
=(\beta_p^2  \beta_q\beta_p^2\beta_q^{-1})( \beta_q^2  \beta_r\beta_q^2\beta_r^{-1} )(\beta_r^2  \beta_p\beta_r^2\beta_p^{-1})$$ $$=\beta_p^2\beta_q\beta_p^2\beta_q\beta_r\beta_q^2\beta_r\beta_p\beta_r^2\beta_p^{-1}
=\beta_p^2\beta_q\beta_p\beta_q\beta_r^2\beta_p^{-1}
=\beta_p^2\beta_q\beta_r\beta_p^{-1}=\beta_p\beta_p^{-1}=\mathrm{id},
$$
where we have used the relation $\beta_p=(\beta_q\beta_r)^{-1} $ four times.
\end{proof}

By renormalizing $v_p$ if necessary, we may and will assume that
\begin{equation}\label{prod1}
v_p v_p^{(q)}v_q v_q^{(r)}v_r v_r^{(p)}=1.
\end{equation}

\begin{lemma}\label{cond2}
We have 
\begin{equation}\label{prod2}
    u_r^{(p)}v_r^*v_q^{(r)*}u_q^{(r)}v_q^*v_p^{(q)*}u_p^{(q)}v_p^*v_r^{(p)*}= -\epsilon_{z_p}(p) \epsilon_{z_q}(q) \epsilon_{z_r}(r). 
\end{equation}
\end{lemma}
\begin{proof}
First note that the relation $\beta_h^2=\mathrm{Ad}(v_h)\circ \alpha_{h+z_h} $ can be rewritten as
$$\beta_h^{-1}= (\mathrm{Ad}(v_h)\circ \alpha_{h+z_h} )^{-1} \beta_h .$$
We then have $$u_r^{(p)}=\beta_p(u_r)=(\beta_q\beta_r)^{-1}(u_r)=\beta_r^{-1}(\beta_q^{-1}(u_r))$$ $$=((\mathrm{Ad}(v_r) \circ \alpha_{r+z_r} )^{-1} \beta_r) (((\mathrm{Ad}(v_q) \circ \alpha_{q+z_q} )^{-1} \beta_q )(u_r))$$ $$=\chi_r(q+z_q+r+z_r) (\mathrm{Ad}(v_r^*v_q^{(r)*}) \circ \beta_r\beta_q)(u_r) $$
$$=\epsilon_{q+z_q+r+z_r} (r) (\mathrm{Ad}(v_r^*v_q^{(r)*})\circ \beta_r)(-\epsilon_{z_p}(p)v_p^*u_p^{(p)*}v_pu_q^*)  $$
$$=-\epsilon_{z_p}(p) \epsilon_{z_r+z_q}(r)\mathrm{Ad}(v_r^*v_q^{(r)*})(\beta_r(v_p)^* \beta_q^{-1}(u_p)^*\beta_r(v_p)\beta_r(u_q)^*) $$
$$=-\epsilon_{z_p}(p) \epsilon_{z_r+z_q}(r)\mathrm{Ad}(v_r^*v_q^{(r)*})((\nu_p \overline{\mu_p(q+z_q)}v_q^*\beta_q(v_p)^*v_q )$$ $$\cdot((\mathrm{Ad}(v_q^*)\circ \alpha_{q+z_q}\beta_q)(u_p)^*)(\overline{\nu_p}\mu_p(q+z_q)v_q^*\beta_q(v_p)v_q ) \beta_r(u_q)^* )$$
$$=-\epsilon_{z_p}(p) \epsilon_{z_r+z_q}(r)\chi_p(q+z_q)\mathrm{Ad}(v_r^*v_q^{(r)*})(v_q^*\beta_q(v_p)^*\beta_q(u_p)^*\beta_q(v_p)v_q  \beta_r(u_q)^* )$$
$$=-\epsilon_{z_p+z_q}(p) \epsilon_{z_r+z_q}(r)v_r^*v_q^{(r)*}v_q^*\beta_q(v_p)^*\beta_q(u_p)^*\beta_q(v_p)v_q \beta_r(u_q)^* v_q^{(r)}v_r $$
$$=-\epsilon_{z_p}(p)\epsilon_{z_q}(q) \epsilon_{z_r}(r)v_r^{(p)}v_pu_p^{(q)*}v_p^{(q)}v_q  u_q^{(r)*} v_q^{(r)}v_r .$$
Rearranging 
gives the condition.
\end{proof}

We next apply $\rho $ and $\beta_h $ to the conditions in Lemmas \ref{cond1} and \ref{cond2} to see if any further constraints arise. It turns out only applying $\beta_h $ to the condition in Lemma \ref{cond1} gives an additional constraint.

\begin{lemma}\label{newcn}
We have $$1=\epsilon_{z_p+z_q+z_r}(z_p)\epsilon_{z_p}(z_p+z_q+z_r) $$ $$=\epsilon_{z_p+z_q+z_r}(z_q)\epsilon_{z_q}(z_p+z_q+z_r)=\epsilon_{z_p+z_q+z_r}(z_r)\epsilon_{z_r}(z_p+z_q+z_r)$$
\end{lemma}
\begin{proof}
We have 
$$1=\beta_p(1)=\beta_p(v_p v_p^{(q)}v_q v_q^{(r)}v_r v_r^{(p)}) $$
$$=\nu_p v_p\beta_r^{-1}(v_p)\beta_p(v_q) (\beta_p\beta_r)(v_q) \beta_p(v_r)(\mathrm{Ad}(v_p)\circ \alpha_{p+z_p})(v_r) $$
$$=\nu_p \mu_r(p+z_p)v_p (v_r^* (\alpha_{r+z_r}\beta_r)( v_p) v_r )(\overline{\nu_q}\mu_q(r+z_r)v_r^*\beta_r(v_q) v_r )$$ $$\cdot(\beta_p\beta_r)(v_q) \beta_p(v_r) v_pv_rv_p^* $$
$$=\nu_p \overline{\nu_q}\mu_r(p+z_p)\mu_p(r+z_r)\mu_q(r+z_r)v_p v_r^* \beta_r(v_p) v_q^{(r)} v_r (\beta_p\beta_r)(v_q) v_r^{(p)} v_pv_rv_p^* $$
$$=\nu_p \overline{\nu_q}\mu_r(p+z_p)\mu_p(r+z_r)\mu_q(r+z_r)v_p v_r^* (\overline{\nu_p}\mu_p(q+z_q)v_q^*\beta_q(v_p)v_q  ) v_q^{(r)} v_r $$ $$\cdot(\mathrm{Ad}(v_r^{(p)}v_p) \alpha_{q+z_r+z_p} \beta_q  )( v_q )    v_r^{(p)} v_pv_rv_p^* $$
$$= \mu_r(p+z_p)\mu_p(p+z_r+z_q)\mu_q(r+z_p+z_q)
$$ $$\cdot v_p v_r^* v_q^*v_p^{(q)}v_q   v_q^{(r)} v_r v_r^{(p)}v_pv_q^* v_q  v_qv_p^* v_r^{(p)*}  v_r^{(p)} v_pv_rv_p^* $$
$$= \mu_r(p+z_p)\mu_p(p+z_r+z_q)\mu_q(p+z_p) .$$
Using the relation
 $$\mu_h(g)=\chi_h(g)\epsilon_{z_h}(g)=\epsilon_g(h)\epsilon_{z_h}(g) ,$$
 we have
 $$1=\mu_r(p+z_p)\mu_p(p+z_r+z_q)\mu_q(p+z_p)
 $$
 $$=\epsilon_{p+z_r+z_q}(p)\epsilon_{p+z_p}(q)\epsilon_{p+z_p}(r)\epsilon_{z_p}(p+z_r+z_q) \epsilon_{z_q}(p+z_p)\epsilon_{z_r}(p+z_p) $$
 $$=\epsilon_{z_p}(z_q+z_r)\epsilon_{z_q+z_r}(z_p) =\epsilon_{z_p}(z_p+z_q+z_r)\epsilon_{z_p+z_q+z_r}(z_p). $$
 Similarly, we can replace $p $ with $q$ or $r$.
\end{proof}
\begin{corollary}\label{zcond}
One of the following occurrs:
\begin{enumerate}
    \item $z_p+z_q+z_r=0 $
    \item Two of $z_p $, $z_q $, $ z_r$ are $0$
    \item $z_p=z_q=z_r $
\end{enumerate}
\end{corollary}
\begin{proof}
Suppose $z_p+z_q+z_r \neq 0 $. Then $$\epsilon_{g}(z_p+z_q+z_r)\epsilon_{z_p+z_q+z_r}(g)=1$$
holds for $g=0$ or $g=z_p+z_q+z_r $. Suppose, e.g. $ z_p$ and $ z_q$ are both nonzero. Then we must have 
$$z_p=z_q=z_p+z_q+z_r =z_r.$$
\end{proof}

\begin{remark}
Conversely, the relations in Lemma \ref{newcn} follow from any of the conditions in Corollary \ref{zcond}.
\end{remark}

Finally, we record the formulas for $\beta_r$ and $\rho $ applied to $u_p^{(q)} $ and $v_p^{(q)} $, which will be needed for reconstruction.

\begin{lemma}\label{rhoform}
We have
\begin{enumerate}
    \item $\beta_r(u_p^{(q)} )=\chi_p(p+z_r+z_q)v_q^{(r)}v_r u_p^{(p)} v_r^*v_q^{(r)*}$
    \item $\beta_r(v_p^{(q)} )=\nu_p \mu_p(p+z_p+z_q)v_q^{(r)}v_r v_p v_r^*v_q^{(r)*}$
    \item $\rho(u_p^{(q)})=\chi_p(q)u_q^*u_p^{(q)*}v_p^{(q)}v_q[  \beta_r(s) v_q^*v_p^{(q)*} \beta_q(s)^*$ \\ $+\sum_{g \in G} \limits  
  \epsilon_{r+z_p+z_q}(g-p)a_p(g) \beta_r(t_{g-p}) v_q^*v_p^{(q)*}u_p^{(q)}\beta_q(t_g)^*]u_q
 $
    \item $\rho(v_p^{(q)})=-\epsilon_{z_q}(q)\mu_p(q)\chi_p(r+z_p+z_q) \xi_p u_q^* u_p^{(q)*}  v_p^{(q)}v_qu_rv_q^*u_q^{(q)}u_q$
\end{enumerate}
\end{lemma}
\begin{proof}
We have already seen similar calculations for $\beta $ in previous proofs.

For $\rho $, we have
$$\rho(u_p^{(q)})=\rho(\beta_q(u_p))=u_q^*(\alpha_q \beta_q\rho)(u_p) u_q$$ $$=\chi_p(q)u_q^*\beta_q(u_p^*\beta_p(s)s^*+\sum_{g \in G} \limits a_p(g) u_p^*\beta_p(t_{g-p}) u_p t_g^*) u_q $$
$$=\chi_p(q)[ u_q^*u_p^{(q)*}v_p^{(q)}v_q \beta_r(s) v_q^*v_p^{(q)*} \beta_q(s)^*u_q$$ $$+\sum_{g \in G} \limits  
  \epsilon_{r+z_p+z_q}(g-p)u_q^*u_p^{(q)*}v_p^{(q)}v_q \beta_r(t_{g-p}) v_q^*v_p^{(q)*}u_p^{(q)}\beta_q(t_g)^*u_q
] $$
and 
$$\rho(v_p^{(q)}=\rho(\beta_q(v_p))=u_q^*(\alpha_q \beta_q\rho)(v_p) u_q $$
$$=\mu_p(q) \xi_p u_q^* \beta_q(u_p^*\beta_p(u_p^*)v_p)u_q$$
$$=\mu_p(q)\chi_p(r+z_p+z_q) \xi_p u_q^* u_p^{(q)*}  v_p^{(q)}v_q u_p^{(q)*}v_q^*v_p^{(q)*} v_p^{(q)}u_q$$
$$=\mu_p(q)\chi_p(r+z_p+z_q) \xi_p u_q^* u_p^{(q)*}  v_p^{(q)}v_q \beta_r(u_p)^*v_q^*u_q$$
$$=-\epsilon_{z_q}(q)\mu_p(q)\chi_p(r+z_p+z_q) \xi_p u_q^* u_p^{(q)*}  v_p^{(q)}v_qu_rv_q^*\beta_q(u_q)v_q   v_q^*u_q$$
$$=-\epsilon_{z_q}(q)\mu_p(q)\chi_p(r+z_p+z_q) \xi_p u_q^* u_p^{(q)*}  v_p^{(q)}v_qu_rv_q^*u_q^{(q)}u_q.$$
\end{proof}

\subsection{Reconstruction}

We now describe how to reconstruct $\mathbb{Z}_2 \times \mathbb{Z}_2 $-graded extensions of the $\mathbb{Z}_2 \times \mathbb{Z}_2 $ generalized Haagerup category, following the calculations of the previous section.

We start with the Cuntz algebra $\mathcal{O}_5 $, together with the endomorphism $\rho $ and $G=\mathbb{Z}_2 \times \mathbb{Z}_2 $ action $\alpha $.

Let $z_p,z_q,z_r \in \mathbb{Z}_2\times \mathbb{Z}_2$ satisfying the conditions of Corollary \ref{zcond} be given. Let $a_p=(1,-i,1,i) $, $a_q=(1,i,-i,1) $, and $a_r=(1,1,i,-i) $. Let $\xi_h =i$, $$\chi_h(g)=\epsilon_g(h), \quad \mu_h(g)=\epsilon_g(h)\epsilon_{z_h}(g)  ,$$ for $h \in \{p,q,r \} $. 

Let $\nu_h \in \{\pm i \}$ for $h=0$ and $\nu_h\in \{\pm 1\}$ for $h\in \{p,q,r\}$ be given for each $h \in \{p,q,r \} $.

For $h \in \{ p,q,r\} $, we will denote by $h'$ its successor in the cylic ordering $(p,q,r) $, and by $h''$ the third element.

Let $\mathcal{U} =\mathcal{O}_{5} * \mathcal{O}_{5}*\mathcal{O}_{5} *\mathcal{O}_{5}* C^*(\mathbb{F}_{13})$, which is the universal $C^*$-algebra generated by four copies of $\mathcal{O}_{5} $ and fifteen unitaries $u_h=u_h^{(0)}$, $u_h^{(h)} $, $u_h^{(h')} $, $v_h=v_h^{(0)}$, $v_{h}^{(h')} $, for  $h \in\{p,q,r \} $; subject to the relations 
\begin{equation}\label{eqi1}
v_p v_p^{(q)}v_q v_q^{(r)}v_r v_r^{(p)}=1 
\end{equation}
and \begin{equation}\label{eqi2}u_r^{(p)}v_r^*v_q^{(r)*}u_q^{(r)}v_q^*v_p^{(q)*}u_p^{(q)}v_p^*v_r^{(p)*}= -\epsilon_{z_p}(p) \epsilon_{z_q}(q) \epsilon_{z_r}(r) 
\end{equation}
(recall that these relations come from Lemmas \ref{cond1} and \ref{cond2}.)

We label the four copies of $\mathcal{O}_5 $ by $ \mathbb{Z}_2\times \mathbb{Z}_2 $, and
denote by $\Phi_h $ the inclusion map of $\mathcal{O}_5 $ into $
\mathcal{U}$ corresponding to $h \in \{0,p,q,r \} $. 

Then we define
$\tilde{\alpha}_g$
on $\mathcal{U} $ by
$$\tilde{\alpha}_g(\Phi_h(x))=\Phi_h(\alpha_g(x)),  \ x \in \mathcal{O}_5 $$
$$\tilde{\alpha}_g(u_h^{(k)})=\chi_h(g)u_h^{(k)}, \quad  \tilde{\alpha}_g(v_h^{(k)})=\mu_h(g)v_h^{(k)}  ,$$
and define $ \tilde{\rho}$ by
$$\tilde{\rho}(\Phi_h(x))=
u_h^*\Phi_h(\rho\alpha_{h}(x)) u_h, \quad \mathrm{(where we let } u_0=1 \mathrm{)}$$
$$\tilde{\rho}(u_h)=u_h^*( s^{(h)}s^{(0)}{}^*+\sum_{g \in G} \limits a_h(g) t^{(h)}_{g+h} u_ht^{(0)}_{g} {}^* )  ,$$ 
$$\tilde{\rho}(u_h^{(h)}) =\chi_h(h) u_h^*u_h^{(h)*}v_h( s^{(0)}v_h^*s^{(h)}{}^*+\sum_{g \in G} \limits a_h(g) \epsilon_{h+z_h}(g+h) t^{(0)}_{g+h} v_h^*u_h^{(h)}t^{(h)}_{g} {}^* )u_h ,$$

$$\rho(u_h^{(h')})=\chi_h(h') u_{h'}^*u_h^{(h')*}v_h^{(h')}v_{h'}[ s^{(h'')} v_{h'}^*v_h^{(h')*} s^{(h')*}$$ $$+\sum_{g \in G} \limits  
  \epsilon_{h''+z_h+z_{h'}}(g+h) t_{g+h}^{h''} v_{h'}^*v_h^{(h')*}u_h^{(h')}t_g^{(h')*}]u_{h'}
 $$

$$\tilde{\rho}(v_h)= \xi_h u_h^*u_h^{(h)*}v_h $$

    $$\rho(v_h^{(h')})=-\epsilon_{z_{h'}}(h')\mu_h(h')\chi_h(h''+z_h+z_{h'}) \xi_h u_{h'}^* u_h^{(h')*}  v_h^{(h')}v_{h'}u_{h''}v_{h'}^*u_{h'}^{(h')}u_{h'}$$

\begin{remark}
The formulas for $\tilde{\rho} $ come from the formulas for $\mathbb{Z}_2 $-extensions in the previous chapter, together with the calculations in Lemma \ref{rhoform} for the unitaries $u_h^{(h')} $ and $v_{h}^{(h')} $. Note that $\chi_h(h)=\epsilon_{h}(h)=-1 $ and $\chi_h(h')=\epsilon_{h'}(h)=1 $ can be used to simplify the formulas for $\tilde{\rho}(u_h^{(h)}) $ and $\tilde{\rho}(u_h^{(h')}) $. Similarly, the scalar cofficients in the formulas for $\tilde{\rho}(v_h) $ and $\tilde{\rho}(v_h^{(h')}) $ can be simplified to $i $ and $i\epsilon_{z_h+z_{h'}}(h'') $, respectively.   
\end{remark}

\begin{lemma}
\begin{enumerate}
    \item The formulas above define a $G$-action $ \tilde{\alpha}$ and an endomorphism $\tilde{\rho} $ on $\mathcal{U} $.
    \item We have $\tilde{\alpha}_g \circ \tilde{\rho}=\tilde{\rho} \circ \tilde{\alpha}_g $.
\end{enumerate}
\end{lemma}
\begin{proof}
To show that the formulas give well-defined maps, we need to check that $\tilde{\alpha}_g $ and $\tilde{\rho} $ preserves the relations (\ref{eqi1}) and (\ref{eqi2}).

Applying $\tilde{\alpha}_g $ to $v_p v_p^{(q)}v_q v_q^{(r)}v_r v_r^{(p)}$ multiplies it by the scalar 
$$\mu_{p}(g)^2\mu_{q}(g)^2\mu_{r}(g)^2=1,$$
and applying $\tilde{\alpha}_g $ to $u_r^{(p)}v_r^*v_q^{(r)*}u_q^{(r)}v_q^*v_p^{(q)*}u_p^{(q)}v_p^*v_r^{(p)*}$ multiplies it by the scalar
$$\chi_{p}(g)\mu_p(g)^2\chi_{q}(g)\mu_{q}(g)^2\chi_{r}(g)\mu_{r}(g)^2u_r^{(p)}$$ $$= \chi_{p}(g)\chi_{q}(g)\chi_{r}(g)=\epsilon_g(p+q+r)=\epsilon_g(0)=1.$$

For $\tilde{\rho} $,
we have $$\tilde{\rho}(v_p v_p^{(q)}v_q v_q^{(r)}v_r v_r^{(p)})$$ $$=(\xi_p u_p^*u_p^{(p)*}v_p)(-\epsilon_{z_{q}}(q)\mu_p(q)\chi_p(r+z_p+z_{q}) \xi_p u_{q}^* u_p^{(q)*}  v_p^{(q)}v_{q}u_{r}v_{q}^*u_{q}^{(q)}u_{q})
$$
$$\cdot(\xi_q u_q^*u_q^{(q)*}v_q)(-\epsilon_{z_{r}}(r)\mu_q(r)\chi_q(p+z_q+z_{r}) \xi_q u_{r}^* u_q^{(r)*}  v_q^{(r)}v_{r}u_{p}v_{r}^*u_{r}^{(r)}u_{r})
$$
$$\cdot(\xi_r u_r^*u_r^{(r)*}v_r)(-\epsilon_{z_{p}}(p)\mu_r(p)\chi_r(q+z_r+z_{p}) \xi_r u_{p}^* u_r^{(p)*}  v_r^{(p)}v_{p}u_{q}v_{p}^*u_{p}^{(p)}u_{p})
$$
The scalar coefficient is
$$-(\xi_p\xi_q\xi_r)^2\epsilon_{z_{p}}(p)\epsilon_{z_{q}}(q) \epsilon_{z_{r}}(r)\epsilon_q(p)\epsilon_{z_p}(q)\epsilon_r(q)\epsilon_{z_q}(r) \epsilon_p(r)\epsilon_{z_r}(p)$$ $$\cdot \epsilon_{r+z_p+z_q}(p) \epsilon_{p+z_r+z_q}(q) \epsilon_{q+z_r+z_p}(r)   $$
$$=-\epsilon_{z_p}(p)\epsilon_{z_q}(q)\epsilon_{z_r}(r) $$
and the product of unitaries, after cancelling inverses, is
$$u_p^*u_p^{(p)*}v_p u_{q}^* u_p^{(q)*} ( v_p^{(q)}v_{q}u_q^{(r)*}  v_q^{(r)}v_{r} u_r^{(p)*}  v_r^{(p)}v_{p})u_{q}v_{p}^*u_{p}^{(p)}u_{p}$$
$$=-\epsilon_{z_{p}}(p)\epsilon_{z_{q}}(q)\epsilon_{z_{r}}(r)(u_p^*u_p^{(p)*}v_p u_{q}^*)( u_{q}v_{p}^*u_{p}^{(p)}u_{p}) $$
$$=-\epsilon_{z_{p}}(p)\epsilon_{z_{q}}(q)\epsilon_{z_{r}}(r) ,$$
where we have used relation (\ref{eqi2}) (after taken the adjoint and a cyclic reordering).
Thus
$$\tilde{\rho}(v_p v_p^{(q)}v_q v_q^{(r)}v_r v_r^{(p)})=1 ,$$
and $\tilde{\rho} $ preserves relation (\ref{eqi1}).

Next, we have
$$\tilde{\rho}(u_h^{(h')}v_h^*v_{h''}^{h)*}) $$
$$=\chi_h(h')[ u_{h'}^*u_h^{(h')*}v_h^{(h')}v_{h'} s^{(h'')} v_{h'}^*v_h^{(h')*} s^{(h')*}u_{h'}$$ $$+\sum_{g \in G} \limits  
  \epsilon_{h''+z_h+z_{h'}}(g+h)u_{h'}^*u_h^{(h')*}v_h^{(h')}v_{h'} t_{g+h}^{h''} v_{h'}^*v_h^{(h')*}u_h^{(h')}t_g^{(h')*}u_{h'}
] $$
$$(\overline{\xi_{h}} v_{h}^*u_h^{(h)}u_h )
(-\epsilon_{z_{h}}(h)\mu_{h''}(h)\chi_{h''}(h'+z_h+z_{h''}) \overline{\xi_{h''}} 
u_{h}^* u_{h}^{(h)*}  v_{h}u_{h'}^*v_{h}^*v_{h''}^{(h)*}u_{h''}^{(h)}u_{h})$$
$$=-\epsilon_{z_h+z_{h''}}(h')u_{h'}^*u_h^{(h')*}v_h^{(h')}v_{h'}[  s^{(h'')} v_{h'}^*v_h^{(h')*} s^{(h')*}$$ $$+\sum_{g \in G} \limits  
  \epsilon_{h''+z_h+z_{h'}}(g+h)t_{g+h}^{h''} v_{h'}^*v_h^{(h')*}u_h^{(h')}t_g^{(h')*}
] v_{h}^*v_{h''}^{(h)*}u_{h''}^{(h)}u_{h}.$$
So then 
$$\tilde{\rho}(u_r^{(p)}v_r^*v_q^{(r)*}u_q^{(r)}v_q^*v_p^{(q)*}u_p^{(q)}v_p^*v_r^{(p)*})$$
$$=(-\epsilon_{z_r+z_{q}}(p)u_{p}^*u_r^{(p)*}v_r^{(p)}v_{p}[  s^{(q)} v_{p}^*v_r^{(p)*} s^{(p)*}$$ $$+\sum_{g \in G} \limits  
  \epsilon_{q+z_r+z_{p}}(g+r)a_r(g)t_{g+r}^{q} v_{p}^*v_r^{(p)*}u_r^{(p)}t_g^{(p)*}
] v_{r}^*v_{q}^{(r)*}u_{q}^{(r)}u_{r})$$
$$\cdot (-\epsilon_{z_q+z_{p}}(r)u_{r}^*u_q^{(r)*}v_q^{(r)}v_{r}[  s^{(p)} v_{r}^*v_q^{(r)*} s^{(r)*}$$ $$+\sum_{g \in G} \limits  
  \epsilon_{p+z_q+z_{r}}(g+q)a_q(g)t_{g+q}^{p} v_{r}^*v_q^{(r)*}u_q^{(r)}t_g^{(r)*}
] v_{q}^*v_{p}^{(q)*}u_{p}^{(q)}u_{q})$$
$$\cdot (-\epsilon_{z_p+z_{r}}(q)u_{q}^*u_p^{(q)*}v_p^{(q)}v_{q}[  s^{(r)} v_{q}^*v_p^{(q)*} s^{(q)*}$$ $$+\sum_{g \in G} \limits  
  \epsilon_{r+z_p+z_{q}}(g+p)a_p(g)t_{g+p}^{r} v_{q}^*v_p^{(q)*}u_p^{(q)}t_g^{(q)*}
] v_{p}^*v_{q}^{(p)*}u_{r}^{(p)}u_{p})$$
$$=-\epsilon_{z_p}(p)\epsilon_{z_q}(q)\epsilon_{z_r}(r) u_{p}^*u_r^{(p)*}v_r^{(p)}v_{p}
$$ $$[  s^{(q)} v_{p}^*v_r^{(p)*} v_{r}^*v_q^{(r)*} v_{q}^*v_p^{(q)*} s^{(q)*}$$ $$+\sum_{g \in G} \limits  
  \epsilon_{p+z_q+z_{r}}(g) \epsilon_{r+z_p+z_{q}}(g+q)
  \epsilon_{q+z_r+z_{p}}(g+r)
  a_r(g)a_q(g+q)a_p(g+r)$$ $$
  t_{g+r}^{(q)} v_{p}^*v_r^{(p)*}u_r^{(p)}v_{r}^*v_q^{(r)*}u_q^{(r)} v_{q}^*v_p^{(q)*}u_p^{(q)} t_{g+r}^{(q)*}
]v_{p}^*v_{q}^{(p)*}u_{r}^{(p)}u_{p}
$$
$$=-\epsilon_{z_p}(p)\epsilon_{z_q}(q)\epsilon_{z_r}(r) u_{p}^*u_r^{(p)*}v_r^{(p)}v_{p}
[  s^{(q)}s^{(q)*}$$ $$+\sum_{g \in G} \limits  
-\epsilon_{z_p}(p)\epsilon_{z_q}(q)\epsilon_{z_r}(r)
  \epsilon_{p+z_q+z_{r}}(g) \epsilon_{r+z_p+z_{q}}(g+q)
  \epsilon_{q+z_r+z_{p}}(g+r)
  $$ $$\cdot a_r(g)a_q(g+q)a_p(g+r)
  t_{g+r}^{(q)} t_{g+r}^{(q)*}
]v_{p}^*v_{q}^{(p)*}u_{r}^{(p)}u_{p},
$$
where we have used relations (\ref{eqi1})  and (\ref{eqi2}) in the last step.

Finally, we have 
$$ -\epsilon_{z_p}(p)\epsilon_{z_q}(q)\epsilon_{z_r}(r)
  \epsilon_{p+z_q+z_{r}}(g) \epsilon_{r+z_p+z_{q}}(g+q)
  \epsilon_{q+z_r+z_{p}}(g+r)
 $$ $$\cdot a_r(g)a_q(g+q)a_p(g+r)$$
$$= a_r(g)a_q(g+q)a_p(g+r)=1,$$
(using relation (\ref{arel})), so we get 
$$\tilde{\rho}(u_r^{(p)}v_r^*v_q^{(r)*}u_q^{(r)}v_q^*v_p^{(q)*}u_p^{(q)}v_p^*v_r^{(p)*})$$
$$=-\epsilon_{z_p}(p)\epsilon_{z_q}(q)\epsilon_{z_r}(r)[s^{(q)}s^{(q)*}+\sum_{g \in G} \limits  
  t_{g+r}^{(q)} t_{g+r}^{(q)*} ]$$
$$=-\epsilon_{z_p}(p)\epsilon_{z_q}(q)\epsilon_{z_r}(r),$$
and $\tilde{\rho} $ preserves relation (\ref{eqi2}).

It is then clear that $\tilde{\alpha} $ is a $G$-action and $\tilde{\rho} $ is an endomorphism of $\mathcal{U} $.

To check that  $\tilde{\rho} \circ  \tilde{\alpha}_g=\tilde{\alpha}_g \circ \tilde{\rho} $, it suffices to check the relations
$$\tilde{\alpha}_g(\tilde{\rho}(u_h^{(h')}))=\tilde{\rho}(\tilde{\alpha}_g)(u_h^{(h')}))=\chi_h(g)\tilde{\rho}(u_h^{(h')})=\epsilon_g(h)\tilde{\rho}(u_h^{(h')}) $$
and
$$\tilde{\alpha}_g(\tilde{\rho})(v_h^{(h')}))=\tilde{\rho}(\tilde{\alpha}_g)(v_h^{(h')}))=\mu_h(g)\tilde{\rho}(v_h^{(h')})=\epsilon_g(h)\epsilon_{z_h}(g)\tilde{\rho}(v_h^{(h')}) $$
which can be easily verified from the formulas for 
$\tilde{\rho}(u_h^{(h')}) $ and $\tilde{\rho}(v_h^{(h')}) $ .

\end{proof}

Next we define automorphisms $\tilde{\beta}_h $ for $h \in \{p,q,r \} $ by the formulas
$$\tilde{\beta}_h(\Phi_k(x))=\begin{cases}
\Phi_h(x) & k=0 \\
 v_h\Phi_{0}(\alpha_{h+z_h}(x))v_h^*  & k=h\\
v_{h''}^*\Phi_{h''}(\alpha_{h''+z_{h''}}(x))v_{h''} & k=h'\\
 v_{h''}^{(h)}v_h\Phi_{h'}(\alpha_{h'+z_h+z_{h''}}(x))v_h^*v_{h''}^{(h)*} & k=h''\\
\end{cases}
$$
$$\tilde{\beta}_h(u_h)=u_h^{(h)}, \quad \tilde{\beta}_h(u_h^{(h)})=\chi_h(h+z_h)v_hu_hv_h^*$$
$$\tilde{\beta}_h(u_h^{(h')})=-\epsilon_{z_{h'}}(h')\chi_h(h''+z_{h''})v_{h''}^*v_{h'}^* u_{h'}^{(h')*} v_{h'}u_{h''}^*v_{h''}$$
$$\tilde{\beta}_h(u_{h'})=-\epsilon_{z_{h''}}(h'')v_{h''}^*u_{h''}^{(h'')*}v_{h''}u_{h}^*  , \quad \tilde{\beta}_h(u_{h'}^{(h')})=\chi_h'(h''+z_{h''})v_{h''}^*u_{h'}^{h''} v_{h''}$$
$$\tilde{\beta}_h(u_{h'}^{(h'')})=\chi_{h'}({h'}+z_{h}+z_{h''})v_{h''}^{(h)}v_h u_{h'}^{(h')} v_h^*v_{h''}^{(h)*})$$
$$\tilde{\beta}_h(u_{h''})=u_{h''}^{(h)}, \quad \tilde{\beta}_h(u_{h''}^{(h'')})=-\epsilon_{z_h}(h)\chi_{h''}(h'+z_h+z_{h''})v_{h''}^{(h)} u_h^{(h)*}v_h u_{h'}^* v_h^*v_{h''}^{(h)*}$$
$$\tilde{\beta}_h(u_{h''}^{(h)})=\chi_{h''}(h+z_h)v_hu_{h''}v_h^*$$
$$\tilde{\beta_h}(v_h)=\nu_h v_h, \quad \tilde{\beta_h}(v_h^{(h')})=\overline{\nu_p}\mu_p(h+z_{h'}+z_{h''})v_{h''}^*v_{h'}^*v_h^{(h')}  v_{h'}    v_{h''} $$
$$\tilde{\beta}_h(v_{h'})= \overline{\nu_{h'}}\mu_{h'}(h''+z_{h''})v_{h''}^*v_{h'}^{(h'')} v_{h''}$$ $$ \tilde{\beta}_h(v_{h'}^{(h'')})=\nu_{h'} \mu_{h'}(h'+z_{h}+z_{h''})v_{h''}^{(h)}v_h v_{h'} v_h^*v_{h''}^{(h)*}$$
$$\tilde{\beta}_h(v_{h''})=v_{h''}^{(h)}, \quad \tilde{\beta}_h(v_{h''}^{(h)})=\mu_{h''}(h+z_h)v_hv_{h''}v_h^*.$$
Once again, the formulas here come from the calculations in the previous subsection, and some of the scalar coefficients can be simplified by calculating $\chi $ and $\mu $ in terms of $\epsilon $.

\begin{lemma} \label{longproof}
The formulas above define automorphisms $ \tilde{\beta}_h$ on $\mathcal{U} $ such that:
\begin{enumerate}
    \item $\tilde{\beta}_h \circ \tilde{\alpha}_g=  \tilde{\alpha}_g\circ\tilde{\beta}_h$
    \item $\tilde{\beta}_h^2=\mathrm{Ad}(v_h) \circ \alpha_{h+{z+h}}$
    \item $\tilde{\beta}_h \circ \tilde{\rho}=\mathrm{Ad}(u_h) \circ \tilde{\alpha}_h\circ \tilde{\rho} \circ \tilde{\beta}_h $.
\end{enumerate}
\end{lemma}

The proof of Lemma \ref{longproof} is straightforward but tedious, so we defer it to an appendix.

\begin{lemma}
We have $\tilde{\rho}^2(x)=sxs^*+\sum_{g \in G} \limits t_g (\alpha_g\rho)(x)t_g^* $ for all $x \in \mathcal{U} $.
\end{lemma}

\begin{proof}
It suffices to check the relation for $x=u_{h}^{(h')} $ and $x=v_h^{(h')} $.
We have 
$$\tilde{\rho}^2(u_{h}^{(h')})=\tilde{\rho}^2\tilde{\beta}_{h'}(u_h)=\tilde{\rho}\mathrm{Ad}(u_{h'}^*)\tilde{\alpha}_{h'}\tilde{\beta}_{h'}\tilde{\rho}(u_h)=\mathrm{Ad}(\tilde{\rho}(u_{h'})^*u_{h'}^*)\tilde{\beta}_{h'}\tilde{\rho}^2(u) $$
$$=\mathrm{Ad}(\tilde{\rho}(u_{h'})^*u_{h'}^*)\tilde{\beta}_{h'} (su_hs^*+\sum_{g \in G} \limits t_g (\alpha_g\rho)(u_h)t_g^*)$$
$$=\mathrm{Ad}(\tilde{\rho}(u_{h'})^*u_{h'}^*) (s^{(h')}u_h^{(h')}s^{(h')*}+\sum_{g \in G} \limits \chi_h(h')t_g^{(h')} u_{h'}(\alpha_g\rho)(u_h^{(h')})u_{h'}^*t_g^{(h')*})$$
$$= (s^{(h')}s^{(0)}{}^*+\sum_{g \in G} \limits a_{h'}(g) t^{(h')}_{g+h'} u_{h'}t^{(0)}_{g} {}^* ) ^* $$
$$\cdot (s^{(h')}u_h^{(h')}s^{(h')*}+\sum_{g \in G} \limits \chi_h(h') t_g^{(h')} u_{h'}(\alpha_g\rho)(u_h^{(h')})u_{h'}^*t_g^{(h')*})$$
$$\cdot (s^{(h')}s^{(0)}{}^*+\sum_{g \in G} \limits a_{h'}(g) t^{(h')}_{g+h'} u_{h'}t^{(0)}_{g} {}^* ) $$
$$=s^{(0)}u_h^{(h')}s^{(0)*}+\sum_{g \in G} \limits \chi_h(h') t_g^{(0)} \alpha_{g+h'}\rho(u_{h}^{(h')})t_g^{(0)*}  $$
$$=s^{(0)}u_h^{(h')}s^{(0)*}+\sum_{g \in G} \limits  t_g^{(0)} \alpha_{g}\rho(u_{h}^{(h')})t_g^{(0)*}  $$
and a similar calculation applies to $v_{h}^{(h')} $.

\end{proof}

\begin{theorem}
For any choice of $ z_p$, $z_q $, $z_r $ satisfying one of the conditions of Corollary \ref{zcond}, and any choice of 
$\nu_h \in \{\pm i \}$ for $z_h=0$ and $\nu_h\in \{\pm 1\}$ for $z_h\in \{p,q,r\}$
, 
there exists a corresponding extension of the $ \mathbb{Z}_2 \times \mathbb{Z}_2$ generalized Haagerup category $\mathcal{C} $ by the $ \mathbb{Z}_2 \times \mathbb{Z}_2$ subgroup of the outer automorphism group.
\end{theorem}

We can count the number of distinct $ \mathbb{Z}_2 \times \mathbb{Z}_2$-extensions given by the above construction as follows. There are 28 triples $(z_p,z_q,z_r) $ which satisfy one of the conditions of Corollary \ref{zcond}. 
\begin{enumerate}
    \item For $(0,0,0)$, there exist exactly 8 extensions. 
    \item For each of $(x,0,0)$, $(0,x,0)$, $(0,0,x)$, $x\in \{p,q,r\}$, there exist exactly  4 extensions. 
    \item For each of $(x,x,0)$, $(x,0,x)$, $(0,x,x)$, $x\in \{p,q,r\}$, there exist exactly 2 extensions. 
    \item For each of $(x,x,x)$, $x\in \{p,q,r\}$, there exist exactly 2 extensions. 
    \item For each of $(p,q,r)$, $(q,r.p)$, $(r,p,q)$, $(p,r,q)$, $(r,q,p)$, $(q,p,r)$, 
    there exists a unique extension. 
\end{enumerate}
The fourth and fifth cases are a little subtle, and we discuss them now. 

Assume $(z_p,z_q,z_r)=(p,p,p)$. 
Then we may assume $\nu_p=\nu_q=1$. 
The only remaining freedom for perturbing $\beta_p$ and $\beta_q$ keeping this condition is to replace $\beta_p$ and $\beta_q$ with $\alpha_x\circ \beta_p$ and $\alpha_y\circ \beta_q$ with $x,y\in \{0,p\}$, up to inner automorphisms. 
This amounts to replacing $\beta_r$ with $\alpha_{x+y}\circ \beta_r$, and multiplying $\nu_r$ by 
$$\epsilon_{x+y}(z_r)\epsilon_{z_r}(x+y)=\epsilon_{x+y}(p)\epsilon_p(x+y),$$
which is always 1 in any combination of $x$ and $y$. 
Thus the two extensions for $\nu_r=1$ and $\nu_r=-1$ are inequivalent. 

In the fifth case, a similar computation shows that the two extensions for $\nu_r=1$ and $\nu_r=-1$ are equivalent. 

\begin{corollary}
There exist exactly 74 different $ \mathbb{Z}_2 \times \mathbb{Z}_2$-graded extensions, up to equivalence.
\end{corollary}

We can interpret our classification result in terms of Theorem \ref{ENOth} and Theorem \ref{DNth} as follows. 

First we show that the freedom for $\nu_h$ corresponds to $H^3(\mathbb{Z}_2\times \mathbb{Z}_2,\mathbb{T})\cong \mathbb{Z}_2^3$, with which we identify $\{1,-1\}^3$.  
Assume a $\mathbb{Z}_2\times \mathbb{Z}_2$-graded extension $(\alpha_g,\rho,\beta_h)$ is realized in $\mathrm{End}_0(M)$. 
We choose another factor $N$ and an outer $\mathbb{Z}_2\times \mathbb{Z}_2$-kernel 
$\sigma:\mathbb{Z}_2\times \mathbb{Z}_2\to \mathrm{Aut}(N)$, which is a map inducing an embedding of $\mathbb{Z}_2\times \mathbb{Z}_2$ into $\mathrm{Out}(N)$.  
We may assume $\sigma_r=(\sigma_p\circ \sigma_q)^{-1}$. 
Then there exist unitaries $w_h\in U(N)$ for $h=p,q,r$ satisfying $\sigma_h^2=\mathrm{Ad}(w_h)$, and there exists $\delta_h\in \{1,-1\}$ satisfying $\sigma_h(w_h)=\delta_h w_h$. 
The triple $(\delta_p,\delta_q,\delta_r)\in \{1,-1\}^3$ is identified with the obstruction of $\sigma$ in $H^3(\mathbb{Z}_2\times \mathbb{Z}_2,\mathbb{T})$. 
Now can get a new extension $(\alpha_g\otimes \mathrm{id},\rho\otimes \mathrm{id},\beta_h\otimes \sigma_h)$ realized in $\mathrm{End}_0(M\otimes N)$, which has the same $(z_p,z_q,z_r)$ as before while $\nu_h$ is replaced by $\delta_h\nu_h$. 
This means that the freedom of $\nu_h$ corresponds to the $H^3(\mathbb{Z}_2\times \mathbb{Z}_2,\mathbb{T})$-torsor structure. 

Now the only remaining freedom is $(z_p,z_q,z_r)$, which should correspond to an element in $$H^2(\mathbb{Z}_2\times \mathbb{Z}_2,\mathrm{Inv}(\mathcal{Z}(\mathcal{C}))=H^2(\mathbb{Z}_2\times \mathbb{Z}_2,G)=G^3.$$
This means that out of 64 possibilities for $M$, only 28 have trivial obstruction $O^4(c,M)\in H^4(\mathbb{Z}_2\times \mathbb{Z}_2,\mathbb{T})$. 

Finally, we have 
$$H^1(\mathbb{Z}_2\times \mathbb{Z}_2,\mathrm{Inv}(\mathcal{Z}(\mathcal{C}))=H^1(\mathbb{Z}_2\times \mathbb{Z}_2,G)=\mathrm{Hom}(\mathbb{Z}_2\times \mathbb{Z}_2,G)=G^2.$$
Since the effect of $p^1_{(c,M)}:H^1(\mathbb{Z}_2\times \mathbb{Z}_2,\mathrm{Inv}(\mathcal{Z}(\mathcal{C}))\to H^3(\mathbb{Z}_2\times \mathbb{Z}_2,\mathbb{T})$ 
should correspond to the freedom of replacing $(\beta_p,\beta_q,\beta_r)$ by 
$$(\alpha_x\circ\beta_p,\alpha_y\circ \beta_q,\alpha_{x+y}\circ \beta_r)$$
up to inner perturbation, we should have 
$$p^1_{(c,M)}(x,y)=(\epsilon_{z_p}(x)\epsilon_x(z_p),\epsilon_{z_q}(y)\epsilon_y(z_q),\epsilon_{z_r}(x+y)\epsilon_{x+y}(z_r)).$$

\subsection{$A_4$-extensions}
We will now consider extensions by the entire outer automorphism group $A_4=(\mathbb{Z}_2\times \mathbb{Z}_2)\rtimes \mathbb{Z}_3$. 

Let $\theta $ be the automorphism of $G$ defined by $\theta(h)=h'$. It was shown in \cite{MR3827808} that since the structure constants $A$ and $\epsilon $ are invariant under $\theta$, the automorphism $\gamma_0 $ of $\mathcal{O}_5 $ defined by 
$$\gamma_0(s)=s, \quad \gamma_0(t_g)=t_{\theta(g)} $$
(and as usual extended to the closure) satisfies 
$$\gamma_0 \circ \alpha_g = \alpha_{\theta(g)} \circ \gamma_0, \quad \gamma_0 \circ \rho = \rho \circ \gamma_0.$$

Let $H=\langle \gamma_0 \rangle \cong \mathbb{Z}_3$. Then we have 
$$H^n(H, \mathrm{Inv}(\mathcal{Z}(\mathcal{C})))=\{0\},\quad \forall n\geq 1,$$
$$H^4(H,\mathbb{T})=\{0\},$$
$$H^3(H,\mathbb{T})=\mathbb{Z}_3.$$
Thus Theorem \ref{ENOth} and Theorem \ref{DNth} show that there exist exactly three $H$-extensions of $\mathcal{C}$. 

One of the three $H$-extensions is generated by $\gamma_0 $, and the other two can be obtained by modifying the associator of the $\mathrm{Vec}_{\mathbb{Z}_3} $ subcategory generated by $\gamma_0 $ by an element of $H^3(\mathbb{Z}_3,\mathbb{T}) $, as in the argument at the end of the previous subsection. (We will refer to this construction as changing the associator of $\gamma_0 $).

\begin{remark}
The same argument works for other order 3 subgroups of $\mathrm{Out}(\mathcal{C})$ too. 
\end{remark}

We would like to extend $\gamma_0 $ to $\mathcal{U} $. Suppose that $\theta(z_h)=z_{\theta(h)}$ for $h \in \{p,q,r \} $.
Define $\tilde{\gamma_0} $
by $$\tilde{\gamma_0}(\Phi_h(x) )=\Phi_{\theta(h)}(\gamma_0(x))$$
$$ \tilde{\gamma_0}(u_{h}^{(k)}) =u_{\theta(h)}^{(\theta(k)
)}, \quad \tilde{\gamma_0}(v_{h}^{(k)}) =v_{\theta(h)}^{(\theta(k)
)}.$$

\begin{lemma}
The above formulas define an automorphism of $\mathcal{U} $, and we have
\begin{enumerate}
    \item $\tilde{\gamma_0}^3=\mathrm{id} $
    \item $\tilde{\gamma_0} \circ \tilde{\alpha}_g =\tilde{\alpha}_{\theta(g)} \circ \tilde{\gamma_0}$
    \item $\tilde{\gamma_0} \circ \tilde{\rho} =\tilde{\rho}\circ \tilde{\gamma_0}$
    \item If $\nu_p=\nu_q=\nu_r $, then $\tilde{\gamma_0} \circ \tilde{\beta}_h=\tilde{\beta}_{\theta(h)} \circ \tilde{\gamma_0}$
\end{enumerate}
\end{lemma}
\begin{proof}
First note that, using the fact that $z_{\theta(h)}=\theta(z_h) $, we can see that $\tilde{\gamma_0} $ preserves the relations (\ref{eqi1}) and (\ref{eqi2}), and is therefore well-defined. It is then clear that $\tilde{\gamma_0} $ is an automorphism of order $3$. 

Then the relations (2)-(4) follow from the invariance of the structure constants under $\theta $.

Since $\epsilon_g(h)=\epsilon_{\theta(g)}(\theta(h)) $ (and therefore
also $\chi_g(h)=\chi_{\theta(g)}(\theta(h))  $ and $\mu_g(h)=\mu_{\theta(g)}(\theta(h))  $, again using the fact that $z_{\theta(h)}=\theta(z_h) $), we can check that $\tilde{\gamma_0} \circ \tilde{\alpha}_g =\tilde{\alpha}_{\theta(g)} \circ \tilde{\gamma_0}$.

Similarly, since $\xi_h=i, \ \forall h \in \{p,q,r \} $ and $a_h(g)=a_{\theta(h)}(\theta(g)) $, we can check that $\tilde{\gamma_0} \circ \tilde{\rho} =\tilde{\rho}\circ \tilde{\gamma_0}$.

And since $\nu_h $ is the same for all $h$, we can check that $\tilde{\gamma_0} \circ \tilde{\beta}_h=\tilde{\beta}_{\theta(h)} \circ \tilde{\gamma_0}$.
\end{proof}

Again by changing the associator of $\gamma_0$, we get triple the number of extensions. 

\begin{theorem}
There exist exactly 15 quasi-trivial extensions of $\mathcal{C} $ by the entire outer automorphism subgroup of the Brauer-Picard group. 
More precisely, 
\begin{enumerate}
    \item For each of the 2 cases $z_p=z_q=z_r=0$ and $\nu_p=\nu_q=\nu_r=\pm i$, there are exactly 3 extensions distinguished by the associators of the invertible objects in the $\theta$-homogeneous part. These 6 extensions form a torsor over  $H^3(A_4,\mathbb{T})$. 
    \item For each of the 3 cases $(z_p,z_q,z_r)=(x,x',x'')$, $x\in \{p,q,r\}$, and $\nu_p=\nu_q=\nu_r=1$, there exist exactly three extensions distinguished by the associators of the invertible objects in the $\theta$-homogeneous part. 
\end{enumerate}
\end{theorem}

\begin{proof}
Let $(\alpha,\rho,\beta_p,\beta_q,\beta_r,\gamma)$ be a $A_4$-extension of $\mathcal{C}$ realized in $\mathrm{End}_0(M)$, where $(\alpha,\rho,\beta_p,\beta_q,\beta_r)$ is a $\mathbb{Z}_2\times \mathbb{Z}_2$-extension as in the previous subsection, and 
$\gamma$ is an invertible object in the $\theta$-homogeneous part. 
Then $\gamma^3\in \mathcal{C}$, and there exists $g\in G$ satisfying $[\gamma^3]=[\alpha_g]$. 
Since $[\gamma^3]$ commutes with $[\gamma]$, we get $[\gamma^3]=[\mathrm{id}]$. 
Since 
$$[\alpha_g\gamma]=[\alpha_{g'+g''}\gamma]=[\alpha_{g'}\gamma\alpha_{g'}^{-1}],$$
the associator of $\gamma$ does not depend on the choice of the invertible object $\gamma$ in the $\theta$-homogeneous part. 
Thus the associator of $\gamma$ is a well-defined invariant of the extension. 

Since $A_4$ is a semi-direct product $G\rtimes H$, and $|G|=|\mathrm{Inv}(\mathcal{Z}(\mathcal{C})))|=4$, $|H|=3$, we have  $H^p(H,H^q(G,\mathbb{T}))=0$ for all $p\geq 1$, $q\geq 1$, and $H^1(H,\mathrm{Inv}(\mathcal{Z}(\mathcal{C}))^H))=0$. 
Thus Lyndon-Hochschild-Serre spectral sequence shows that there exists a split exact sequence 
$$0\to H^3(H,\mathbb{T})\to H^3(A_4,\mathbb{T})\to H^3(G,\mathbb{T})^H\to 0,$$
where $H^3(H,\mathbb{T})\cong \mathbb{Z}_3$ and  $H^3(G,\mathbb{T})^H\cong \mathbb{Z}_2$,  
and 
$$H^1(A_4,\mathrm{Inv}(\mathcal{Z}(\mathcal{C}))))=H^1(G,G)^H=\mathrm{Hom}(G,G)^\theta\cong \mathbb{Z}_2\times \mathbb{Z}_2.$$
Thus the intersection of $H^3(H,\mathbb{T})$ and the image of  $$p^1_{(c,M)}:H^1(A_4,\mathrm{Inv}(\mathcal{Z}(\mathcal{C})))\to H^3(A_4,\mathbb{T})$$ 
is trivial, which means that the set of equivalence classes of $A_4$-extensions of $\mathcal{C}$ has a free $H^3(H,\mathbb{T})$-action through the $H^3(A_4,\mathbb{T})$-action, and it changes the associator of $\gamma$. 
In particular, we get the extensions listed in the theorem. 

Now it suffices to show that there exist exactly 5 extensions with $\gamma$ having trivial associator. 
In this case, we may assume that $\gamma^3=\mathrm{id}$, and the $H$-extension $(\alpha,\rho,\gamma)$ is equivalent to the model $(\alpha,\rho,\gamma_0)$. 
Thus using the uniqueness theorem, we may assume that $\gamma$ acts on the Cuntz algebra $\mathcal{O}_5\subset M$ as $\gamma_0$ by replacing $\rho$ with $\mathrm{Ad}(w)\circ \rho$ with a unitary $w$ fixed by $\alpha_g$ for all $g\in G$. 
Recall that this replacement does not change the extension data of 
$\beta_p$, $\beta_q$, or $\beta_r$. 
Thus we may and do assume that $\gamma$ restricted to $\mathcal{O}_5$ is $\gamma_0$ from the beginning. 

Since $(\alpha,\rho,\beta_p)$, $(\alpha,\rho, \gamma^{-1}\circ \beta_q\circ \gamma)$, $(\alpha,\rho,\gamma\circ \beta_r\circ\gamma^{-1})$ are equivalent extensions, we have 
$z_{h'}=z_h'$ for all $h\in \{p,q,r\}$, and $\nu_p=\nu_q=\nu_r=\pm i$ if $z_h=0$. 
If $z_h\neq 0$, we can arrange $\beta_p$, $\beta_q$, and $\beta_r$ so that $\nu_p=\nu_q=\nu_r=1$. 

As in the case of $\gamma$, we have $[(\beta_p\gamma)^3]=[\mathrm{id}]$, and so 
$$[\beta_p][\gamma\beta_p\gamma^{-1}][\gamma^2\beta_p\gamma^{-2}]=[\mathrm{id}].$$
Thus we may replace $\beta_q$ with $\gamma\circ \beta_p\circ \gamma^{-1}$, and $\beta_r$ with 
$(\beta_p\circ \gamma\circ \beta_q\gamma^{-1})^{-1}$, 
which does not change the extension data of $\beta_p$, $\beta_q$, or $\beta_r$. 

Unfortunately, we can not expect that $(\beta_p\gamma)^3=\mathrm{id}$ holds on the nose, and we should modify $\beta_p$. 
We have $\beta_p^{-1}\circ \gamma\circ \beta_p=\beta_p^{-1}\circ \beta_q\circ \gamma$, and 
\begin{align*}
\beta_p^{-1}\circ \beta_q&=\beta_p^{-2}\circ \beta_r^{-1}=\beta_p^{-2}\circ \beta_r^{-2}\circ \beta_r=\mathrm{Ad}v_p^*\circ \alpha_{p+z_p}\circ \mathrm{Ad}v_r^*\circ \alpha_{r+z_r}\circ \beta_r\\
&=\mathrm{Ad}(v_p^*v_r^*)\circ \alpha_{q+z_q}\circ \beta_r. 
\end{align*}
We set $\beta'_r=\mathrm{Ad}(v_p^*v_r^*)\circ \alpha_{q+z_q}\circ \beta_r$, which satisfies $(\beta'_r\circ \gamma)^3=\mathrm{id}$. 
We set $\beta'_p=\gamma\circ \beta'_r\circ \gamma^{-1}$ and $\beta'_q=\gamma^2\circ \beta'_r\circ \gamma^{-2}$. 
Then 
\begin{equation}\label{c1}
\gamma\circ \beta'_x\circ \gamma^{-1}=\beta'_{x'}
\end{equation}
\begin{equation}\label{c2}
\beta'_x\circ \beta'_{x'}\circ \beta'_{x''}=\mathrm{id}
\end{equation}
hold for all $x\in \{p,q,r\}$. 

Although the extension data of $\beta'_x$ is not necessarily the same as before, it is  completely determined by that of $\beta_x$. 
Since we can work on the new extension data in the previous sections equally well, we assume that Eq.(\ref{c1}),(\ref{c2}) hold for $\beta_x$ instead of $\beta'_x$ to avoid heavy notation. 

The above two equations force that $\gamma(u_x)$ is a multiple of $u_{x'}$, and $\gamma(v_x)$ is a multiple of $v_{x'}$. 
For the latter, we can simply assume that $\gamma(v_x)=v_{x'}$ holds for all $x\in \{p,q,r\}$ by renaming them, while we can still keep Eq.(\ref{prod1}) by normalizing $v_p$. 
Eq.(\ref{rhou}) shows that $\gamma(u_x)=u_{x'}$ holds for all $x\in \{p,q,r\}$. 

Now the action of $\gamma$ on $\mathcal{O}_5\cup\{u_x,u_x^{(x)},u_x^{(x')},v_x,v_x^{(x')}\}$ are completely determined by the data $(z_p,\nu_p)$. 
This means that if two $A_4$-extensions of $\mathcal{C}$ share the same data, they share the same 6j-symbols, and they are equivalent extensions. 
\end{proof}

\appendix
\section{Proof of Lemma \ref{longproof}}
In this Appendix, we prove Lemma \ref{longproof}, which states that the $\tilde{\beta}_h $, as defined in the reconstruction of a $ \mathbb{Z}_2 \times \mathbb{Z}_2$-extension of the $ \mathbb{Z}_2 \times \mathbb{Z}_2$ generalized Haagerup category,  satisfy the appropriate relations. The tedious proof consists of checking the claimed identities of endomorphisms by calculating the images of the various generating unitaries under the left and right hand side of each identity, simplifying if possible using (\ref{eqi1}) and (\ref{eqi2}), and comparing the results.
\begin{proof}
First, we need to show that $\tilde{\beta}_h $ a is well-defined endomorphism. Clearly, $\tilde{\beta}_h $ maps each copy of $ \mathcal{O}_5$ isomorphically onto another Cuntz subalgebra of $\mathcal{U} $. Then we need to check that $\tilde{\beta}_h $ preserves relations (\ref{eqi1}) and (\ref{eqi2}). The relation (\ref{eqi1}) was checked in the proof of Lemma \ref{newcn}. For relation (\ref{eqi2}), we have
$$\tilde{\beta}_p(u_r^{(p)}v_r^*v_q^{(r)*}u_q^{(r)}v_q^*v_p^{(q)*}u_p^{(q)}v_p^*v_r^{(p)*}) $$
$$=(\chi_{r}(p+z_p)v_pu_{r}v_p^*)(v_r^{(p)*})(\nu_{q} \mu_{q}(q+z_{p}+z_{r})v_{r}^{(p)}v_p v_{q} v_p^*v_{r}^{(p)*})^* $$
$$(\chi_{q}({q}+z_{p}+z_{r})v_{r}^{(p)}v_p u_{q}^{(q)} v_p^*v_{r}^{(p)*}) (\overline{\nu_{q}}\mu_{q}(r+z_{r})v_{r}^*v_{q}^{(r)} v_{r})^* (\overline{\nu_p}\mu_p(p+z_{q}+z_{r})v_{r}^*v_{q}^*v_p^{(q)}  v_{q}    v_{r})^*$$
$$(-\epsilon_{z_{q}}(q)\chi_p(r+z_{r})v_{r}^*v_{q}^* u_{q}^{(q)*} v_{q}u_{r}^*v_{r})(\nu_p v_p)^*(\mu_{r}(p+z_p)v_pv_{r}v_p^*)^* $$
The product of unitaries is
$$v_pu_rv_q^*u_q^{(q)}(v_p^*v_r^{(p)*}v_r^*v_q^{(r)*}v_q^*v_p^{(q)*})u_q^{(q)*}v_qu_r^*v_p^*=1 ,$$
using relation (\ref{eqi1}).
The scalar coefficient is
$$=-\epsilon_{z_q}(q)(\mu_r\chi_r)(p+z_p)(\mu_q\chi_q)(q+z_p+z_r)(\chi_p\mu_q)(r+z_r)\mu_p(p+z_q+z_r) $$
$$=-\epsilon_{z_q}(q)\epsilon_{z_r}(p+z_p)\epsilon_{z_q}(q+z_p+z_r)\epsilon_{r+z_r}(p+q)\epsilon_{z_q}(r+z_r)\epsilon_{p+z_q+z_r}(p)\epsilon_{z_p}(p+z_q+z_r) $$
$$-\epsilon_p(p)\epsilon_r(r)\epsilon_{z_p}(p)\epsilon_{z_q}(q)\epsilon_{z_r}(r)\epsilon_{z_p}(z_q+z_r)\epsilon_{z_q+z_r}(z_p)=-\epsilon_{z_p}(p)\epsilon_{z_q}(q)\epsilon_{z_r}(r) ,$$
using the relation in Lemma \ref{newcn}. Thus $\tilde{\beta}_p $ preserves (\ref{eqi2}), and since this calculation 
is invariant under cyclic permutations of $(p,q,r ) $, so do $\tilde{\beta}_q $ and $\tilde{\beta}_r $.

It is straightforward to check that $\tilde{\beta}_h \circ \tilde{\alpha}_g=\tilde{\alpha}_g \circ \tilde{\beta}_h$.

Next, we need to check that $\tilde{\beta}_h^2=\mathrm{Ad}(v_h) \circ \tilde{\alpha}_{h+z_h} $. This relation clearly holds on $\Phi_0(\mathcal{O}_5) $ and $\Phi_h(\mathcal{O}_5) $. We also have
$$\tilde{\beta}_h^2(\Phi_{h'}(x))=\tilde{\beta}_h (v_{h''}^*\Phi_{h''}(\alpha_{h''+z_{h''}}(x))v_{h''})$$
$$ = v_{h''}^{(h)*}
(v_{h''}^{(h)}v_h\Phi_{h'}(\alpha_{h'+z_h+z_{h''}}(\alpha_{h''+z_{h''}}(x)))v_h^*v_{h''}^{(h)*})
v_{h''}^{(h)})$$
$$=v_h \Phi_{h'}(\alpha_{h+z_h})v_h^*=(\mathrm{Ad}(v_h) \circ \tilde{\alpha}_{h+z_h}) (\Phi_{h'}(x)) $$
and
$$\tilde{\beta}_h^2(\Phi_{h''}(x))=\tilde{\beta}_h (v_{h''}^{(h)}v_h\Phi_{h'}(\alpha_{h'+z_h+z_{h''}}(x))v_h^*v_{h''}^{(h)*})$$
$$=(v_hv_{h''}v_h^*)v_h(
v_{h''}^*\Phi_{h''}(\alpha_{h''+z_{h''}}(\alpha_{h'+z_h+z_{h''}}(x)))v_{h''} 
) v_h^*(v_hv_{h''}^*v_h^*)$$
$$=v_h\Phi_{h''}(\alpha_{h+z_h} (x))v_h^* =(\mathrm{Ad}(v_h) \circ \tilde{\alpha}_{h+z_h}) (\Phi_{h''}(x)) .$$

We will now check this relation for all of the unitaries containing a symbol other than $h$.
\begin{itemize}
\item $$\tilde{\beta}_h^2(u_h^{(h')})=\tilde{\beta}_h(-\epsilon_{z_{h'}}(h')\chi_h(h''+z_{h''})v_{h''}^*v_{h'}^* u_{h'}^{(h')*} v_{h'}u_{h''}^*v_{h''}) $$
$$=-\epsilon_{z_{h'}}(h')\chi_h(h''+z_{h''})(v_{h''}^{(h)} )^*(v_{h''}^*v_{h'}^{(h'')} v_{h''})^*$$ $$(\chi_{h'}(h''+z_{h''})v_{h''}^*u_{h'}^{h''} v_{h''})^*(v_{h''}^*v_{h'}^{(h'')} v_{h''})(u_{h''}^{(h)})^*(v_{h''}^{(h)})$$
$$=\epsilon_{z_{h'}}(h')\epsilon_{z_{h''}}(h'') v_{h''}^{(h)*} v_{h''}^*v_{h'}^{(h'')*} (v_{h''} v_{h''}^*)u_{h'}^{(h'')*}(v_{h''} v_{h''}^*)v_{h'}^{(h'')} v_{h''} u_{h''}^{(h)*}v_{h''}^{(h)}$$
$$=\epsilon_{z_{h'}}(h')\epsilon_{z_{h''}}(h'')(v_{h''}^{(h)*} v_{h''}^*v_{h'}^{(h'')*}) u_{h'}^{(h'')*}v_{h'}^{(h'')} v_{h''} u_{h''}^{(h)*}v_{h''}^{(h)}$$
$$=\epsilon_{z_{h'}}(h')\epsilon_{z_{h''}}(h'')v_hv_h^{(h')}v_{h'} u_{h'}^{(h'')*}(v_{h'}^{(h'')} v_{h''} u_{h''}^{(h)*}v_{h''}^{(h)})$$
$$=\epsilon_{z_{h'}}(h')\epsilon_{z_{h''}}(h'')(-\epsilon_{z_h}(h)\epsilon_{z_{h'}}(h')\epsilon_{z_{h''}}(h''))$$ $$\cdot v_h(v_h^{(h')}v_{h'} u_{h'}^{(h'')*}u_{h'}^{(h'')}v_{h'}^*v_{h}^{(h')*})u_{h}^{(h')}v_{h}^*$$
$$=-\epsilon_{z_h}(h)v_hu_{h}^{(h')}v_{h}^*=\chi_h(h+z_h)v_hu_{h}^{(h')}v_{h}^* $$
$$=(\mathrm{Ad}(v_h)\circ \tilde{\alpha}_{h+z_h}) (u_{h}^{(h')}).$$

\item
$$\tilde{\beta}_h^2(u_{h'}^{(h'')} )=\tilde{\beta}_h(  \chi_{h'}({h'}+z_{h}+z_{h''})v_{h''}^{(h)}v_h u_{h'}^{(h')} v_h^*v_{h''}^{(h)*}) $$
$$=\chi_{h'}({h'}+z_{h}+z_{h''})(v_hv_{h''}v_h^*)( v_h)(\chi_h'(h''+z_{h''})v_{h''}^*u_{h'}^{h''} v_{h''})( v_h)^*(v_hv_{h''}v_h^*)^*$$
$$=\chi_{h'}(h+z_h)v_h(v_{h''}v_h^*v_hv_{h''}^*)u_{h'}^{h''} (v_{h''}v_h^*v_hv_{h''}^*)v_h^*$$
$$=\chi_{h'}(h+z_h)v_hu_{h'}^{h''} v_h^*=(\mathrm{Ad}(v_h)\circ \tilde{\alpha}_{h+z_h}) (u_{h'}^{(h'')}).$$

\item
 $$\tilde{\beta}_h^2(u_{h''}^{(h'')} )=\tilde{\beta}_h(  -\epsilon_{z_h}(h)\chi_{h''}(h'+z_h+z_{h''})v_{h''}^{(h)} u_h^{(h)*}v_h u_{h'}^* v_h^*v_{h''}^{(h)*}) $$
 $$=-\epsilon_{z_h}(h)\chi_{h''}(h'+z_h+z_{h''})(v_hv_{h''}v_h^*)(\chi_h(h+z_h)v_hu_hv_h^*)^*(v_h)$$ $$\cdot(-\epsilon_{z_{h''}}(h'')v_{h''}^*u_{h''}^{(h'')*}v_{h''}u_{h}^*)^*( v_h)^*(v_hv_{h''}v_h^*)^*  $$
$$=\epsilon_{z_h}(h)\epsilon_{z_{h''}}(h'')\chi_h(h+z_h)\chi_{h''}(h'+z_h+z_{h''})$$ $$\cdot v_h(v_{h''}v_h^*v_hu_h^*v_h^*v_h u_h v_{h''}^*)u_{h''}^{(h'')}(v_{h''}  v_h^*v_h v_{h''}^* )v_h^*$$
$$=\epsilon_{z_h}(h'') v_hu_{h''}^{(h'')}v_h^*=\chi_{h''}(h+z_{h}) v_hu_{h''}^{(h'')}v_h^*$$
$$=(\mathrm{Ad}(v_h)\circ \tilde{\alpha}_{h+z_h}) (u_{h''}^{(h'')}) $$

\item
$$\tilde{\beta}_h^2(v_{h}^{(h')} )=\tilde{\beta}_h(\overline{\nu_h}\mu_h(h+z_{h'}+z_{h''})v_{h''}^*v_{h'}^*v_h^{(h')}  v_{h'}    v_{h''} ) $$
$$=\overline{\nu_p}\mu_p(h+z_{h'}+z_{h''})(v_{h''}^{(h)})^*(v_{h''}^*v_{h'}^{(h'')} v_{h''})^*$$ $$\cdot(\overline{\nu_h}\mu_h(h+z_{h'}+z_{h''})v_{h''}^*v_{h'}^*v_h^{(h')}  v_{h'}    v_{h''} )(v_{h''}^*v_{h'}^{(h'')} v_{h''})(v_{h''}) $$
$$=\overline{\nu_h}^2 v_{h''}^{(h)*}v_{h''}^*v_{h'}^{(h'')*} (v_{h''} v_{h''}^*)v_{h'}^*v_h^{(h')}  v_{h'}   ( v_{h''}v_{h''}^*)v_{h'}^{(h'')} v_{h''}v_{h''}^{(h)}$$
$$=-\epsilon_{z_h}(z_h) (v_{h''}^{(h)*}v_{h''}^*v_{h'}^{(h'')*} v_{h'}^*)v_h^{(h')} ( v_{h'}  v_{h'}^{(h'')} v_{h''}v_{h''}^{(h)})$$
$$=\epsilon_{h+z_h}(h)\epsilon_{z_h}(h+z_h) v_h v_h^{(h') }(v_h^{(h')} v_h^{(h')*})v_h^*$$
$$=\mu_h(h+z_h) v_h v_h^{(h') }v_h^* =(\mathrm{Ad}(v_h)\circ \tilde{\alpha}_{h+z_h}) (v_{h}^{(h')}) $$

\item
$$\tilde{\beta}_h^2(v_{h'}^{(h'')} )=\tilde{\beta}_h(\nu_{h'} \mu_{h'}(h'+z_{h}+z_{h''})v_{h''}^{(h)}v_h v_{h'} v_h^*v_{h''}^{(h)*}))$$
$$=\nu_{h'} \mu_{h'}(h'+z_{h}+z_{h''})(v_hv_{h''}v_h^*)( v_h)$$ $$\cdot(\overline{\nu_{h'}}\mu_{h'}(h''+z_{h''})v_{h''}^*v_{h'}^{(h'')} v_{h''})( v_h)^*(v_hv_{h''}v_h^*)^* $$
$$=\mu_{h'}(h+z_h)v_h(v_{h''}v_h^* v_hv_{h''}^*)v_{h'}^{(h'')}( v_{h''}v_h^* v_hv_{h''}^*)v_h^*$$
$$=\mu_{h'}(h+z_h)v_hv_{h'}^{(h'')}v_h^*=
(\mathrm{Ad}(v_h)\circ \tilde{\alpha}_{h+z_h}) (v_{h'}^{(h'')}) $$

\item
$$\tilde{\beta}_h^2(u_{h'} )=\tilde{\beta}_h(-\epsilon_{z_{h''}}(h'')v_{h''}^*u_{h''}^{(h'')*}v_{h''}u_{h}^* ) $$
$$=-\epsilon_{z_{h''}}(h'')(v_{h''}^{(h)})^*(-\epsilon_{z_h}(h)\chi_{h''}(h'+z_h+z_{h''})$$ $$\cdot v_{h''}^{(h)} u_h^{(h)*}v_h u_{h'}^* v_h^*v_{h''}^{(h)*})^*(v_{h''}^{(h)})(u_h^{(h)})^*$$
$$=-\epsilon_{z_h}(h') (v_{h''}^{(h)*} v_{h''}^{(h)}) v_hu_{h'}v_h^*(u_h^{(h)}v_{h''}^{(h)*}v_{h''}^{(h)} u_h^{(h)*})
 $$
 $$=\chi_{h'}(h+z_h) v_hu_{h'}v_h^*=(\mathrm{Ad}(v_h)\circ \tilde{\alpha}_{h+z_h}) (u_{h'}) 
 $$
 
\item
 $$\tilde{\beta}_h^2(u_{h'}^{(h')} )=\tilde{\beta}_h(\chi_h'(h''+z_{h''})v_{h''}^*u_{h'}^{h''} v_{h''} )$$
$$=\chi_{h'}(h''+z_{h''})(v_{h''}^{(h)})^*(\chi_{h'}({h'}+z_{h}+z_{h''})v_{h''}^{(h)}v_h u_{h'}^{(h')} v_h^*v_{h''}^{(h)*}))(v_{h''}^{(h)}) $$
$$=\chi_{h'}(h+z_h)(v_{h''}^{(h)*}v_{h''}^{(h)})v_h u_{h'}^{(h')} v_h^*(v_{h''}^{(h)*}v_{h''}^{(h)} )$$
$$=\chi_{h'}(h+z_h)v_h u_{h'}^{(h')} v_h^*=(\mathrm{Ad}(v_h)\circ \tilde{\alpha}_{h+z_h}) (u_{h'}^{(h')}) $$

\item
$$\tilde{\beta}_h^2(u_{h''} )=\tilde{\beta}_h(u_{h''}^{(h)})=\chi_{h''}(h+z_h)v_h u_{h''}v_h^*=(\mathrm{Ad}(v_h)\circ \tilde{\alpha}_{h+z_h}) (u_{h''}) $$

\item
$$\tilde{\beta}_h^2(u_{h''}^{(h)} )=\tilde{\beta}_h( \chi_{h''}(h+z_h)v_hu_{h''}v_h^*)$$
$$=\chi_{h''}(h+z_h)(\nu_h v_h)(u_{h''}^{(h)}) (\nu_h v_h)^*  $$
$$=\chi_{h''}(h+z_h) v_hu_{h''}^{(h)}  v_h^*  =(\mathrm{Ad}(v_h)\circ \tilde{\alpha}_{h+z_h}) (u_{h''}^{(h)})$$

\item
$$\tilde{\beta}_h^2(v_{h'} )=\tilde{\beta}_h(\overline{\nu_{h'}}\mu_{h'}(h''+z_{h''})v_{h''}^*v_{h'}^{(h'')} v_{h''}) $$
$$= \tilde{\beta}_h(\overline{\nu_{h'}}\mu_{h'}(h''+z_{h''})(v_{h''}^{(h)}) (\nu_{h'} \mu_{h'}(h'+z_{h}+z_{h''})v_{h''}^{(h)}v_h v_{h'} v_h^*v_{h''}^{(h)*})) (v_{h''}^{(h)*})) $$
$$=\mu_{h'}(h+z_h)(v_{h''}^{(h)}v_{h''}^{(h)})v_h v_{h'} v_h^*(v_{h''}^{(h)*} v_{h''}^{(h)*} )$$
$$=\mu_{h'}(h+z_h)v_h v_{h'} v_h^*=(\mathrm{Ad}(v_h)\circ \tilde{\alpha}_{h+z_h}) (v_{h'})$$

\item
$$\tilde{\beta}_h^2(v_{h''} )=\tilde{\beta}_h(v_{h''}^{(h)} )=\mu_{h''}(h+z_h)v_hv_{h''}v_h^*=(\mathrm{Ad}(v_h)\circ \tilde{\alpha}_{h+z_h}) (v_{h''})$$
\item
$$\tilde{\beta}_h^2(v_{h''}^{(h)} )=\tilde{\beta}_h(\mu_{h''}(h+z_h)v_hv_{h''}v_h^*)$$
$$=\mu_{h''}(h+z_h)(\nu_h v_h)(v_{h''}^{(h)}) (\nu_h v_h)^*$$
$$=\mu_{h''}(h+z_h)v_h v_{h''}^{(h)} v_h^*=\mathrm{Ad}(v_h)\circ \tilde{\alpha}_{h+z_h} (v_{h''}^{(h)}).$$
\end{itemize}

Finally, we will check the relation $\tilde{\beta}_h  \circ \tilde{\rho} =\mathrm{Ad}(u_h)\circ \tilde{\alpha}_h \circ \tilde{\rho}\circ \tilde{\beta}_h $.

We have 
$$(\tilde{\beta}_h \tilde{\rho})(\Phi_{h'}(x)) =\tilde{\beta}_h (u_{h'}^*\Phi_{h'}(\rho \alpha_{h'}(x))u_{h'} )$$
$$=(u_hv_{h''}^*u_{h''}^{(h'')}v_{h''})
(v_{h''}^*\Phi_{h''}(\alpha_{h''+z_{h''}}(\rho \alpha_{h'}(x)))(v_{h''})
v_{h''}^*u_{h''}^{(h'')*}v_{h''}u_{h}^*)
$$
$$=u_hv_{h''}^*u_{h''}^{(h'')}\Phi_{h''}(\rho \alpha_{h+z_{h''}}(x))u_{h''}^{(h'')*}v_{h''}u_{h}^*
$$
while
$$(\mathrm{Ad}(u_h)\tilde{\alpha}_h \tilde{\rho} \tilde{\beta}_h )(\Phi_{h'}(x))$$
$$=u_h(\tilde{\rho} (v_{h''}^*\Phi_{h''}(\alpha_{h'+z_{h''}}(x))v_{h''})      ) u_h^* $$
$$=(u_h v_{h''}^* u_{h''}^{(h'')}u_{h''}) (u_{h''}^*\Phi_{h''}(\rho \alpha_{h+z_{h''}}(x))u_{h''})(u_{h''}^*u_{h''}^{(h'')*}v_{h''}        u_h^*) $$
$$=u_h v_{h''}^* u_{h''}^{(h'')} \Phi_{h''}(\rho \alpha_{h+z_{h''}}(x))u_{h''}^{(h'')*}v_{h''}        u_h^* $$

Similarly, we have
$$(\tilde{\beta}_h \tilde{\rho})(\Phi_{h''}(x)) $$
$$=\tilde{\beta}_h(u_{h''}^*\Phi_{h''}(\rho \alpha_{h''}(x))u_{h''} ) $$
$$=u_{h''}^{(h)*}
(v_{h''}^{(h)}v_h\Phi_{h'}(\alpha_{h'+z_h+z_{h''}}(\rho \alpha_{h''}(x)))v_h^*v_{h''}^{(h)*})
) u_{h''}^{(h)}$$
$$=u_{h''}^{(h)*}
v_{h''}^{(h)}v_h\Phi_{h'}(\rho \alpha_{h+z_h+z_{h''}}(x))v_h^*v_{h''}^{(h)*}u_{h''}^{(h)}$$
while
$$(\mathrm{Ad}(u_h)\tilde{\alpha}_h \tilde{\rho} \tilde{\beta}_h) (\Phi_{h''}(x)) $$
$$=u_h\tilde{\rho}(v_{h''}^{(h)}v_h\Phi_{h'}(\alpha_{h''+z_h+z_{h''}}(x))v_h^*v_{h''}^{(h)*} ) u_h^* $$
$$=u_h(u_h^*u_{h''}^{(h)*}v_{h''}^{(h)}v_hu_{h'})
(u_{h'}^*\Phi_{h'}(\rho\alpha_{h+z_h+z_{h''}}(x))u_{h'})(u_{h'}^*v_h^*v_{h''}^{(h)*}u_{h''}^{(h)}u_h)
 u_h^* $$
$$=u_{h''}^{(h)*}v_{h''}^{(h)}v_h\Phi_{h'}(\rho\alpha_{h+z_h+z_{h''}}(x))v_h^*v_{h''}^{(h)*}u_{h''}^{(h)}$$
 
Now we will again check the relation on all the unitaries containing a symbol other than $h$.
\begin{itemize}
\item
$$(\tilde{\beta}_h \tilde{\rho})(u_{h'}) $$
$$=\tilde{\beta}_h(u_{h'}^*( s^{(h')}s^{(0)}{}^*+\sum_{g \in G} \limits a_{h'}(g) t^{(h')}_{g+h'} u_{h'}t^{(0)}_{g} {}^* )) $$
$$=u_h v_{h''}^*u_{h''}^{(h'')}v_{h''}(-\epsilon_{z_{h''}}(h'') v_{h''}^*s^{(h'')}v_{h''}s^{(h)}{}^*$$ $$+\sum_{g \in G} \limits \epsilon_{h''+z_{h''}}(g+h')a_{h'}(g) v_{h''}^*t^{(h'')}_{g+h'}v_{h''} v_{h''}^*u_{h''}^{(h'')*}v_{h''}u_{h}^*t^{(h)}_{g} {}^* )$$
$$=u_h v_{h''}^*u_{h''}^{(h'')}(-\epsilon_{z_{h''}}(h'') s^{(h'')}v_{h''}s^{(h)}{}^*$$ $$+\sum_{g \in G} \limits \epsilon_{h''+z_{h''}}(g+h')a_{h'}(g) t^{(h'')}_{g+h'}u_{h''}^{(h'')*}v_{h''}u_{h}^*t^{(h)}_{g} {}^* ),$$
while
$$(\mathrm{Ad}(u_h)\tilde{\alpha}_h \tilde{\rho} \tilde{\beta}_h )(u_{h'}) $$
$$=\chi_{h'}(h)u_h \tilde{\rho} (-\epsilon_{z_{h''}}(h'')v_{h''}^*u_{h''}^{(h'')*}v_{h''}u_{h}^*) u_h^*$$
$$=-\epsilon_{z_{h''}}(h'') u_h$$ $$\cdot 
( ( s^{(0)}v_{h''}^*s^{(h'')}{}^*+\sum_{g \in G} \limits a_{h''}(g) \epsilon_{h''+z_{h''}}(g+h'') t^{(0)}_{g+h''} v_{h''}^*u_{h''}^{(h'')}t^{(h'')}_{g} {}^* )u_{h''}^{(h'')*}v_{h''})^*$$ $$
\cdot u_h^*( s^{(h)}s^{(0)}{}^*+\sum_{g \in G} \limits a_h(g) t^{(h)}_{g+h} u_ht^{(0)}_{g} {}^* )^*  u_h^*$$
$$=-\epsilon_{z_{h''}}(h'') u_h v_{h''}^*u_{h''}^{(h'')}
(s^{(h'')}v_{h''}s^{(h)*}$$ $$+\sum_{g \in G}\limits \overline{a_{h''}(g)a_h(g+h'')} \epsilon_{h''+z_{h''}}(g+h'')t_g^{(h'')}u_{h''}^{(h'')*}v_{h''}u_{h}^* t_{g+h'}^{(h)*}  )  $$
$$=u_h v_{h''}^*u_{h''}^{(h'')}
(-\epsilon_{z_{h''}}(h'') s^{(h'')}v_{h''}s^{(h)*}$$ $$+\sum_{g \in G}\limits a_{h'}(g+h') \epsilon_{h''+z_{h''}}(g)t_g^{(h'')}u_{h''}^{(h'')*}v_{h''}u_{h}^* t_{g+h'}^{(h)*}  )  $$
$$=u_h v_{h''}^*u_{h''}^{(h'')}
(-\epsilon_{z_{h''}}(h'') s^{(h'')}v_{h''}s^{(h)*}$$ $$+\sum_{g \in G}\limits a_{h'}(g) \epsilon_{h''+z_{h''}}(g+h')t_{g+h'}^{(h'')}u_{h''}^{(h'')*}v_{h''}u_{h}^* t_{g}^{(h)*}  )  ,$$
where we have used (\ref{arel}).
\item
$$(\tilde{\beta}_h \tilde{\rho})(u_{h}^{(h')}) $$ 
 $$=\tilde{\beta}_h( u_{h'}^*u_h^{(h')*}v_h^{(h')}v_{h'} [s^{(h'')} v_{h'}^*v_h^{(h')*} s^{(h')*}$$ $$+\sum_{g \in G} \limits  
  \epsilon_{h''+z_h+z_{h'}}(g+h)a_h(g)t_{g+h}^{h''} v_{h'}^*v_h^{(h')*}u_h^{(h')}t_g^{(h')*}]u_{h'}
$$
 $$=(u_hv_{h''}^*u_{h''}^{(h'')}v_{h''})(
 v_{h''}^*u_{h''}v_{h'}^*u_{h'}^{(h')}v_{h'}   v_{h''})(
 v_{h''}^*v_{h'}^*v_h^{(h')}  v_{h'}    v_{h''})(
 v_{h''}^*v_{h'}^{(h'')} v_{h''})$$ 
 $$[-\epsilon_{z_{h'}}(h')\chi_h(h''+z_{h''})(v_{h''}^{(h)}v_hs^{(h')} v_{h}^*v_{h''}^{(h)*})
 (v_{h''}^*v_{(h')}^{(h'')*} v_{h''})$$ $$\cdot
 (v_{h''}^*v_{h'}^*v_{(h)}^{(h')*}v_{h'} v_{h''})
 (v_{h''}^*s^{(h'')*}v_{h''})
 $$ $$
 +\sum_{g \in G} \limits  
  \epsilon_{h''+z_h+z_{h'}}(g+h)\epsilon_{h'+z_h+z_{h''}}(g+h)\epsilon_{h''+z_{h''}}(g)a_h(g)$$ $$
 (v_{h''}^{(h)}v_h t_{g+h}^{(h')} v_{h}^*v_{h''}^{(h)*})
 (v_{h''}^*v_{(h')}^{(h'')*} v_{h''})
 (v_{h''}^*v_{h'}^*v_{(h)}^{(h')*}v_{h'} v_{h''})$$ $$\cdot
 (v_{h''}^*v_{h'}^* u_{h'}^{(h')*} v_{h'}u_{h''}^*v_{h''})
 (v_{h''}^*t_g^{(h'')*}v_{h''})]$$
 $$
 v_{h''}^*u_{h''}^{(h'')*}v_{h''}u_{h}^* $$
 $$=u_hv_{h''}^*u_{h''}^{(h'')}u_{h''}v_{h'}^*u_{h'}^{(h')}
 [\epsilon_{z_{h'}}(h')\epsilon{z_{h''}}(h)s^{(h')} v_{h'} s^{(h'')*})$$ $$+\sum_{g \in G} \limits
-\epsilon_{z_{h'}}(h+g) \epsilon_{z_{h''}}(h)\epsilon_{h'}(g)a_h(g)t_{g+h}^{(h')}  u_{h'}^{(h')*} v_{h'}u_{h''}^*t_g^{(h'')*}]
u_{h''}^{(h'')*}v_{h''}u_{h}^* ,$$
where we have used (\ref{eqi1}) twice,
while
$$(\mathrm{Ad}(u_h)\tilde{\alpha}_h \tilde{\rho} \tilde{\beta}_h )(u_{h}^{(h')}) $$ 
 $$=\chi_h(h)(-\epsilon_{z_{h'}}(h')\chi_h(h''+z_{h''}))u_h\tilde{\rho}( v_{h''}^*v_{h'}^* u_{h'}^{(h')*} v_{h'}u_{h''}^*v_{h''} )    u_h^*$$
 $$=\epsilon_{z_{h'}}(h')\epsilon_{h''+z_{h''}}(h) u_h(v_{h''}^*u_{h''}^{(h'')}u_{h''})(v_{h'}^*u_{h'}^{(h')}u_{h'})$$ $$\cdot(\chi_{h'}(h') u_{h'}^*u_{h'}^{(h')*}(v_{h'} s^{(0)}v_{h'}^*s^{(h')}{}^*+\sum_{g \in G} \limits a_{h'}(g) \epsilon_{h'+z_{h'}}(g+h') v_{h'}t^{(0)}_{g+h'} v_{h'}^*u_{h'}^{(h')}t^{(h')}_{g} {}^* )u_{h'} )^*$$ $$\cdot (u_{h'}^*u_{h'}^{(h')*}v_{h'})(u_{h''}^*( s^{(h'')}s^{(0)}{}^*+\sum_{g \in G} \limits a_{h''}(g) t^{(h'')}_{g+h''} u_{h''}t^{(0)}_{g} {}^* ))^* (u_{h''}^*u_{h''}^{(h'')*}v_{h''})  u_h^*$$
 $$ 
 =-\epsilon_{z_{h'}}(h')\epsilon_{h''+z_{h''}}(h)u_hv_{h''}^*u_{h''}^{(h'')}u_{h''}v_{h'}^*u_{h'}^{(h')}
 $$
 $$[s^{(h')}v_{h'} s^{(h'')*} +\sum_{g \in G} \limits \overline{a_{h'}(g)a_{h''}(g+h')} \epsilon_{h'+z_{h'}}(g+h') t_g^{(h')}u_{h'}^{(h')*}v_{h'}u_{h''}^*t_{g+h}^{(h'')*}    ] u_{h''}^{(h'')*}v_{h''})  u_h^*$$
   $$ 
 =u_hv_{h''}^*u_{h''}^{(h'')}u_{h''}v_{h'}^*u_{h'}^{(h')}
 $$
  $$[\epsilon_{z_{h'}}(h')\epsilon_{z_{h''}}(h)s^{(h')}v_{h'} s^{(h'')*} $$ $$+\sum_{g \in G} \limits -a_h(g) \epsilon_{z_{h'}}(h+g) \epsilon_{z_{h''}}(h)\epsilon_{h'}(g) t_{g+h}^{(h')}u_{h'}^{(h')*}v_{h'}u_{h''}^*t_{g}^{(h'')*}    ] u_{h''}^{(h'')*}v_{h''})  u_h^*,$$
  where again we have used (\ref{arel}).
 \item
 $$(\tilde{\beta}_h \tilde{\rho})(u_{h'}^{(h'')}) $$ 
 $$=\tilde{\beta}_h(\chi_{h'}(h'')u_{h''}^*u_{h'}^{(h'')*}v_{h'}^{(h'')}v_{h''}[  s^{(h)} v_{h''}^*v_{h'}^{(h'')*} s^{(h'')*}$$ $$+\sum_{g \in G} \limits a_{h'}(g) 
  \epsilon_{h+z_{h'}+z_{h''}}(g+h')t_{g+h'}^{(h)} v_{h''}^*v_{h'}^{(h'')*}u_{h'}^{(h'')}t_g^{(h'')*}]u_{h''}   )$$
$$=(u_{h''}^{(h)} )^*( v_{h''}^{(h)}v_h u_{h'}^{(h')} v_h^*v_{h''}^{(h)*})^*(v_{h''}^{(h)}v_h v_{h'} v_h^*v_{h''}^{(h)*} )(v_{h''}^{(h)}) $$ 
$$ [\chi_{h'}({h'}+z_{h}+z_{h''})(v_h s^{(0)}v_h^*) (v_{h''}^{(h)})^*(v_{h''}^{(h)}v_h v_{h'} v_h^*v_{h''}^{(h)*})^* (v_{h''}^{(h)}v_h    s^{(h')*}v_h^*v_{h''}^{(h)*} )$$
 $$+\sum_{g \in G} \limits a_{h'}(g) 
  \epsilon_{h+z_{h'}+z_{h''}}(g+h')
  \epsilon_{h+z_h}(g+h')\epsilon_{h'+z_h+z_{h''}}(g)
  $$ $$(v_ht_{g+h'}^{(0)}v_h^*)(v_{h''}^{(h)})^*(v_{h''}^{(h)}v_h v_{h'} v_h^*v_{h''}^{(h)*})^*(v_{h''}^{(h)}v_h u_{h'}^{(h')} v_h^*v_{h''}^{(h)*}) 
  (v_{h''}^{(h)}v_h    t_g^{(h')*}v_h^*v_{h''}^{(h)*} )](u_{h''}^{(h)})$$
$$=u_{h''}^{(h)*}  v_{h''}^{(h)}v_h u_{h'}^{(h')*}v_{h'}[-\epsilon_{z_h+z_{h''}}(h')s^{(0)}v_{h'}^*s^{(h')}$$ $$+\sum_{g \in G} \limits a_{h'}(g) 
  \epsilon_{z_h+z_{h'}+z_{h''}}(h')
  \epsilon_{h'+z_{h'}}(g)
  t_{g+h'}^{(0)}v_{h'}^*  u_{h'}^{(h')}  t_g^{(h')*}
  ] v_h^*v_{h''}^{(h)*}u_{h''}^{(h)},$$ 
 while
 $$(\mathrm{Ad}(u_h)\tilde{\alpha}_h \tilde{\rho} \tilde{\beta}_h )(u_{h'}^{(h'')}) $$
 $$=u_h\chi_{h'}(h)\tilde{\rho}(\chi_{h'}({h'}+z_{h}+z_{h''})v_{h''}^{(h)}v_h u_{h'}^{(h')} v_h^*v_{h''}^{(h)*} )     u_h^* $$
 $$=-\epsilon_{z_h+z_{h''}}(h')u_h(u_{h}^* u_{h''}^{(h)*}  v_{h''}^{(h)}v_{h}u_{h'}v_{h}^*u_{h}^{(h)}u_{h})(u_h^*u_h^{(h)*}v_h)$$ $$(u_{h'}^*u_{h'}^{(h')*}(v_{h'} s^{(0)}v_{h'}^*s^{(h')}{}^*+\sum_{g \in G} \limits a_{h'}(g) \epsilon_{h'+z_{h'}}(g+h') v_{h'}t^{(0)}_{g+h'} v_{h'}^*u_{h'}^{(h')}t^{(h')}_{g} {}^* )u_{h'})$$ $$(u_h^*u_h^{(h)*}v_h)^*(u_{h}^* u_{h''}^{(h)*}  v_{h''}^{(h)}v_{h}u_{h'}v_{h}^*u_{h}^{(h)}u_{h})^* u_h^*  $$
 $$= u_{h''}^{(h)*}  v_{h''}^{(h)}v_{h}u_{h'}^{(h')*}v_{h'}(-\epsilon_{z_h+z_{h''}}(h') s^{(0)}v_{h'}^*s^{(h')}{}^*$$ $$+\sum_{g \in G} \limits a_{h'}(g)  \epsilon_{z_h+z_{h'}+z_{h''}}(h')
  \epsilon_{h'+z_{h'}}(g) t^{(0)}_{g+h'} v_{h'}^*u_{h'}^{(h')}t^{(h')}_{g} {}^* ))v_h^*v_{h''}^{(h)*}u_{h''}^{(h)}$$ 
\item
 $$(\tilde{\beta}_h \tilde{\rho})(u_{h''}^{(h'')}) $$
 $$=\tilde{\beta}_h(\chi_{h''}(h'') u_{h''}^*u_{h''}^{(h'')*}v_{h''}( s^{(0)}v_{h''}^*s^{(h'')}{}^*$$ $$+\sum_{g \in G} \limits a_{h''}(g) \epsilon_{h''+z_{h''}}(g+h'') t^{(0)}_{g+h''} v_{h''}^*u_{h''}^{(h'')}t^{(h'')}_{g} {}^* )u_{h''}   )$$
 $$=-(u_{h''}^{(h)})^*(v_{h''}^{(h)} u_h^{(h)*}v_h u_{h'}^* v_h^*v_{h''}^{(h)*})^*(v_{h''}^{(h)})$$ $$\cdot [-\epsilon_{z_h}(h)\chi_{h''}(h'+z_h+z_{h''}) s^{(h)}(v_{h''}^{(h)})^*(v_{h''}^{(h)}v_h s^{(h')}v_h^*v_{h''}^{(h)*})$$ $$ +\sum_{g \in G} \limits a_{h''}(g) \epsilon_{h''+z_{h''}}(g+h'')\epsilon_{h'+z_h+z_{h''}}(g)$$ $$\cdot t^{(h)}_{g+h''}(v_{h''}^{(h)})^*(v_{h''}^{(h)} u_h^{(h)*}v_h u_{h'}^* v_h^*v_{h''}^{(h)*})  (v_{h''}^{(h)}v_h t^{(h')}_{g} {}^* v_h^*v_{h''}^{(h)*}) ](u_{h''}^{(h)})$$
 $$=-u_{h''}^{(h)*}v_{h''}^{(h)}v_h u_{h'} v_h^*u_h^{(h)} [-\epsilon_{z_h}(h)\chi_{h''}(h'+z_h+z_{h''}) s^{(h)}v_h s^{(h')}$$ $$ +\sum_{g \in G} \limits a_{h''}(g) \epsilon_{h''+z_{h''}}(g+h'')\epsilon_{h'+z_h+z_{h''}}(g) t^{(h)}_{g+h''} u_h^{(h)*}v_h u_{h'}^* t^{(h')}_{g} {}^*  ]v_h^*v_{h''}^{(h)*}u_{h''}^{(h)}$$
 $$=-\epsilon_{z_h}(h')\epsilon_{z_{h''}}(z_{h''}) u_{h''}^{(h)*}  v_{h''}^{(h)}v_{h}u_{h'}v_{h}^*u_{h}^{(h)} $$
 $$[s^{(h)}v_hs^{(h')}+ \sum_{g \in G} \limits a_{h''}(g) \epsilon_{h+z_h}(g+h')t_{g+h''}^{(h)}u_{h}^{(h)*}v_hu_{h'}^*t_{g}^{(h')} ]v_{h}^* v_{h''}^{(h)*}u_{h''}^{(h)},$$
 while
 $$(\mathrm{Ad}(u_{h})\tilde{\alpha}_h \tilde{\rho} \tilde{\beta}_h )(u_{h''}^{(h'')}) $$ 
 $$=u_h \chi_{h''}(h) \tilde{\rho}( -\epsilon_{z_h}(h)\chi_{h''}(h'+z_h+z_{h''})v_{h''}^{(h)} u_h^{(h)*}v_h u_{h'}^* v_h^*v_{h''}^{(h)*}   )    u_h^*$$ $$=\epsilon_{z_h}(h')\epsilon_{z_{h'}}(z_{h'})u_h(u_{h}^* u_{h''}^{(h)*}  v_{h''}^{(h)}v_{h}u_{h'}v_{h}^*u_{h}^{(h)}u_{h} )$$ $$\cdot(\chi_h(h) u_h^*u_h^{(h)*}v_h$$ $$\cdot( s^{(0)}v_h^*s^{(h)}{}^*+\sum_{g \in G} \limits a_h(g) \epsilon_{h+z_h}(g+h) t^{(0)}_{g+h} v_h^*u_h^{(h)}t^{(h)}_{g} {}^* )u_h)^*(u_h^*u_h^{(h)*}v_h)$$ $$\cdot(u_{h'}^*( s^{(h')}s^{(0)}{}^*+\sum_{g \in G} \limits a_{h'}(g) t^{(h')}_{g+h'} u_{h'}t^{(0)}_{g} {}^* ) ){}^*$$ $$\cdot(u_h^*u_h^{(h)*}v_h)^*(u_{h}^* u_{h''}^{(h)*}  v_{h''}^{(h)}v_{h}u_{h'}v_{h}^*u_{h}^{(h)}u_{h} )^* u_h^*$$
 $$=-\epsilon_{z_h}(h')\epsilon_{z_{h''}}(z_{h''}) u_{h''}^{(h)*}  v_{h''}^{(h)}v_{h}u_{h'}v_{h}^*u_{h}^{(h)} $$ $$\cdot( s^{(0)}v_h^*s^{(h)}{}^*+\sum_{g \in G} \limits a_h(g) \epsilon_{h+z_h}(g+h) t^{(0)}_{g+h} v_h^*u_h^{(h)}t^{(h)}_{g} {}^* )^*$$ $$\cdot( s^{(h')}s^{(0)}{}^*+\sum_{g \in G} \limits a_{h'}(g) t^{(h')}_{g+h'} u_{h'}t^{(0)}_{g} {}^* ){}^* v_{h}^* v_{h''}^{(h)*}u_{h''}^{(h)}$$
 $$=-\epsilon_{z_h}(h')\epsilon_{z_{h''}}(z_{h''}) u_{h''}^{(h)*}  v_{h''}^{(h)}v_{h}u_{h'}v_{h}^*u_{h}^{(h)} $$
 $$[s^{(h)}v_hs^{(h')}+ \sum_{g \in G} \limits \overline{a_h(g)a_{h'}(g+h)} \epsilon_{h+z_h}(g+h)t_g^{(h)}u_{h}^{(h)*}v_hu_{h'}^*t_{g+h''}^{(h')} ]$$ $$\cdot v_{h}^* v_{h''}^{(h)*}u_{h''}^{(h)}$$
 $$=-\epsilon_{z_h}(h')\epsilon_{z_{h''}}(z_{h''}) u_{h''}^{(h)*}  v_{h''}^{(h)}v_{h}u_{h'}v_{h}^*u_{h}^{(h)} $$
 $$[s^{(h)}v_hs^{(h')}+ \sum_{g \in G} \limits a_{h''}(g) \epsilon_{h+z_h}(g+h')t_{g+h''}^{(h)}u_{h}^{(h)*}v_hu_{h'}^*t_{g}^{(h')} ]v_{h}^* v_{h''}^{(h)*}u_{h''}^{(h)}$$
 \item
 $$(\tilde{\beta}_h \tilde{\rho})(u_{h'}^{(h')}) $$
 $$= \tilde{\beta}_h(\chi_{h'}(h') u_{h'}^*u_{h'}^{(h')*}v_{h'}$$ $$\cdot( s^{(0)}v_{h'}^*s^{(h')}{}^*+\sum_{g \in G} \limits a_{h'}(g) \epsilon_{h'+z_{h'}}(g+h') t^{(0)}_{g+h'} v_{h'}^*u_{h'}^{(h')}t^{(h')}_{g} {}^* )u_{h'}) $$
 $$=-(v_{h''}^*u_{h''}^{(h'')*}v_{h''}u_{h}^*)^*(v_{h''}^*u_{h'}^{h''} v_{h''})^*(v_{h''}^*v_{h'}^{(h'')} v_{h''})  $$
 $$[\chi_{h'}(h''+z_{h''})s^{(h)}(v_{h''}^*v_{h'}^{(h'')} v_{h''})^*(v_{h''}^* s^{(h'')*}v_{h''})$$ $$ + \sum_{g \in G} \limits a_{h'}(g) \epsilon_{h'+z_{h'}}(g+h')\epsilon_{h''+z_{h''}(g)}$$ $$\cdot t^{(h)}_{g+h'}(v_{h''}^*v_{h'}^{(h'')} v_{h''})^*(v_{h''}^*u_{h'}^{h''} v_{h''})v_{(h''}^*t_{g}^{(h'')*} v_{h''}  ] (v_{h''}^*u_{h''}^{(h'')*}v_{h''}u_{h}^*)$$
 $$=-u_hv_{h''}^*v_{h'}^{(h'')} v_{h''}  [\chi_{h'}(h''+z_{h''})s^{(h)}v_{h''}^*v_{h'}^{(h'')*} s^{(h'')*}$$ $$ + \sum_{g \in G} \limits a_{h'}(g) \epsilon_{h'+z_{h'}}(g+h')\epsilon_{h''+z_{h''}}(g)t^{(h)}_{g+h'}v_{h''}^*v_{h'}^{(h'')*} u_{h'}^{(h'')}t_{g}^{(h'')*}  ] u_{h''}^{(h'')*}v_{h''}u_{h}^*,$$
 while
 $$(\mathrm{Ad}(u_{h})\tilde{\alpha}_h \tilde{\rho} \tilde{\beta}_h )(u_{h'}^{(h')}) $$  
 $$=u_h \chi_{h'}(h)\tilde{\rho}(\chi_h'(h''+z_{h''})v_{h''}^*u_{h'}^{h''} v_{h''} )u_h^*$$ 
 $$=-\epsilon_{z_{h''}(h')}u_h
 \chi_{h'}(h'')[ u_{h''}^*u_{h'}^{(h'')*}v_{h'}^{(h'')}v_{h''} s^{(h)} v_{h''}^*v_{h'}^{(h'')*} s^{(h'')*}u_{h''}$$ $$+\sum_{g \in G} \limits  
  \epsilon_{h+z_{h'}+z_{h''}}(g+h')a_{h'}(g)u_{h''}^*u_{h'}^{(h'')*}v_{h'}^{(h'')}v_{h''} t_{g+h'}^{h} v_{h''}^*v_{h'}^{(h'')*}u_{h'}^{(h'')}t_g^{(h'')*}u_{h''}]$$ $$\cdot
 (u_{h''}^*u_{h''}^{(h'')*}v_{h''})   u_h^* $$
 $$=-\epsilon_{z_{h''}(h')}u_hu_{h''}^*u_{h'}^{(h'')*}v_{h'}^{(h'')}v_{h''}
 [  s^{(h)} v_{h''}^*v_{h'}^{(h'')*} s^{(h'')*}$$ $$+\sum_{g \in G} \limits  
  \epsilon_{h+z_{h'}+z_{h''}}(g+h')a_{h'}(g) t_{g+h'}^{h} v_{h''}^*v_{h'}^{(h'')*}u_{h'}^{(h'')}t_g^{(h'')*}]
 u_{h''}^{(h'')*}v_{h''}   u_h^* $$
\item
 $$(\tilde{\beta}_h \tilde{\rho})(u_{h''}) $$
 $$=\tilde{\beta}_h(u_{h''}^*( s^{(h'')}s^{(0)}{}^*+\sum_{g \in G} \limits a_{h''}(g) t^{(h'')}_{g+h''} u_{h''}t^{(0)}_{g} {}^* )) $$
 $$=u_{h''}^{(h)*} (v_{h''}^{(h)}v_h s^{(h')}v_h^* v_{h''}^{(h)*}s^{(h)}{}^*$$ $$+\sum_{g \in G} \limits a_{h''}(g) \epsilon_{h'+z_h+z_{h''}}(g+h'') v_{h''}^{(h)}v_ht^{(h')}_{g+h''}v_h^* v_{h''}^{(h)*} u_{h''}^{(h)}t^{(h)}_{g} {}^* ) $$
 $$=u_{h''}^{(h)*} v_{h''}^{(h)}v_h [s^{(h')}v_h^* v_{h''}^{(h)*}s^{(h)}{}^*$$ $$+\sum_{g \in G} \limits a_{h''}(g) \epsilon_{h'+z_h+z_{h''}}(g+h'') t^{(h')}_{g+h''}v_h^* v_{h''}^{(h)*} u_{h''}^{(h)}t^{(h)}_{g} {}^* ) ,$$
 while
 $$(\mathrm{Ad}(u_{h})\tilde{\alpha}_h \tilde{\rho} \tilde{\beta}_h )(u_{h''}) $$
 $$=u_h(\chi_{h''}(h)\tilde{\rho}(u_{h''}^{(h)}   ) )u_h^* $$
 $$=u_h(\chi_{h''}(h)[ u_{h}^*u_{h''}^{(h)*}v_{h''}^{(h)}v_{h} s^{(h')} v_{h}^*v_{h''}^{(h)*} s^{(h)*}u_{h}$$ $$+\sum_{g \in G} \limits  
  \epsilon_{h'+z_{h''}+z_{h}}(g+h'')a_{h''}(g)u_{h}^*u_{h''}^{(h)*}v_{h''}^{(h)}v_{h} t_{g+h''}^{h'} v_{h}^*v_{h''}^{(h)*}u_{h''}^{(h)}t_g^{(h)*}u_{h}
] )u_h^* $$
$$=u_{h''}^{(h)*}v_{h''}^{(h)}v_{h}[ s^{(h')} v_{h}^*v_{h''}^{(h)*} s^{(h)*}$$ $$+\sum_{g \in G} \limits
  \epsilon_{h'+z_{h''}+z_{h}}(g+h'')a_{h''}(g)t_{g+h''}^{h'} v_{h}^*v_{h''}^{(h)*}u_{h''}^{(h)}t_g^{(h)*}]  $$
\item
 $$(\tilde{\beta}_h \tilde{\rho})(u_{h''}^{(h)}) $$
 $$=\tilde{\beta}_h(u_{h}^*u_{h''}^{(h)*}v_{h''}^{(h)}v_{h} [  s^{(h')} v_{h}^*v_{h''}^{(h)*} s^{(h)*}$$ $$+\sum_{g \in G} \limits  
  \epsilon_{h'+z_{h''}+z_{h}}(g+h'')a_{h''}(g)t_{g+h''}^{h'} v_{h}^*v_{h''}^{(h)*}u_{h''}^{(h)}t_g^{(h)*}
] u_{h}  ) $$
$$=(u_h^{(h)})^*(v_hu_{h''}v_h^*)^*(v_hv_{h''}v_h^*)(v_h)$$ $$\cdot [\chi_{h''}(h+z_h) v_{h''}^*s^{(h'')}v_{h''} (v_{h})^*(v_hv_{h''}v_h^*)^*( v_h s^{(0)*} v_h^*)$$ $$+\sum_{g \in G} \limits  
  \epsilon_{h'+z_{h''}+z_{h}}(g+h'')\epsilon_{h''+z_{h''}}(g+h'') \epsilon_{h+z_h}(g)a_{h''}(g)
   $$ $$\cdot v_{h''}^*t_{g+h''}^{(h'')}v_{h''} (v_{h})^*(v_hv_{h''}v_h^*)^*(v_hu_{h''}v_h^*)( v_h t_g^{(h)*} v_h^*)
]u_h^{(h)} $$
$$=\epsilon_{z_h}(h'')u_h^{(h)*}v_h u_{h''}^*  [s^{(h'')}s^{(0)*} +\sum_{g \in G} \limits  
 a_{h''}(g)
   t_{g+h''}^{(h'')}u_{h''} t_g^{(h)*}
]v_h^*u_h^{(h)} ,$$
while
$$(\mathrm{Ad}(u_{h})\tilde{\alpha}_h \tilde{\rho} \tilde{\beta}_h) (u_{h''}^{(h)}) $$
 $$=u_h(\chi_{h''}(h)\tilde{\rho}( \chi_{h''}(h+z_h)v_hu_{h''}v_h^*   ) )u_h^* $$
$$=\epsilon_{z_h}(h'')u_h(u_h^*u_h^{(h)*}v_h)
$$ $$\cdot(u_{h''}^*( s^{(h'')}s^{(0)}{}^*+\sum_{g \in G} \limits a_{h''}(g) t^{(h'')}_{g+h''} u_{h''}t^{(0)}_{g} {}^* ))
(u_h^*u_h^{(h)*}v_h)^* u_h^* $$
$$=\epsilon_{z_h}(h'')u_h^{(h)*}v_h
u_{h''}^*[ s^{(h'')}s^{(0)}{}^*+\sum_{g \in G} \limits a_{h''}(g) t^{(h'')}_{g+h''} u_{h''}t^{(h)}_{g} {}^* ]
v_h^*u_h^{(h)}$$
\item
$$(\tilde{\beta}_h \tilde{\rho})(v_{h}^{(h')})$$
$$=\tilde{\beta}_h(-\epsilon_{z_{h'}}(h')\mu_h(h')\chi_h(h''+z_h+z_{h'}) \xi_h u_{h'}^* u_h^{(h')*}  v_h^{(h')}v_{h'}u_{h''}v_{h'}^*u_{h'}^{(h')}u_{h'} )$$
$$=(-\epsilon_{z_{h'}}(h')\mu_h(h')\chi_h(h''+z_h+z_{h'}\xi_h)(\overline{\nu_h}\mu_h(h+z_{h'}+z_{h''}))$$ $$\cdot(-\epsilon_{z_{h'}}(h')\chi_h(h''+z_{h''}))(\chi_{h'}(h''+z_{h''}))$$ $$\cdot(v_{h''}^*u_{h''}^{(h'')*}v_{h''}u_{h}^*)^*(v_{h''}^*v_{h'}^* u_{h'}^{(h')*} v_{h'}u_{h''}^*v_{h''})^*(v_{h''}^*v_{h'}^*v_h^{(h')}  v_{h'}    v_{h''} )(v_{h''}^*v_{h'}^{(h'')} v_{h''})$$ $$\cdot(u_{h''}^{(h)})(v_{h''}^*v_{h'}^{(h'')} v_{h''})^*(v_{h''}^*u_{h'}^{h''} v_{h''} )(v_{h''}^*u_{h''}^{(h'')*}v_{h''}u_{h}^*) $$ 
$$=-\xi_h\overline{\nu_h}\epsilon_{z_h}(h'+z_{h'}+z_{h''})\epsilon_{z_{h''}}(h')$$ $$\cdot u_h v_{h''}^*u_{h''}^{(h'')}u_{h''}v_{h'}^*u_{h'}^{h'} (v_h^{(h')}  v_{h'} v_{h'}^{(h'')} v_{h''})u_{h''}^{(h)}v_{h''}^*v_{h'}^{(h'')*} u_{h'}^{h''} u_{h''}^{(h'')*}v_{h''}u_{h}^* $$
$$=-\xi_h\overline{\nu_h}\epsilon_{z_h}(h'+z_{h'}+z_{h''})\epsilon_{z_{h''}}(h')$$ $$\cdot u_h v_{h''}^*u_{h''}^{(h'')}u_{h''}v_{h'}^*u_{h'}^{h'} (v_h^*v_{h''}^{(h)*}u_{h''}^{(h)}v_{h''}^*v_{h'}^{(h'')*} u_{h'}^{h''}) u_{h''}^{(h'')*}v_{h''}u_{h}^* $$
$$=\xi_h\overline{\nu_h}\epsilon_{z_h}(h''+z_{h'}+z_{h''})\epsilon_{z_{h'}}(h')\epsilon_{z_{h''}}(h)$$ $$\cdot u_h v_{h''}^*u_{h''}^{(h'')}u_{h''}v_{h'}^*u_{h'}^{h'} u_h^{(h')*}v_h^{(h')}v_{h'} u_{h''}^{(h'')*}v_{h''}u_{h}^* ,$$
where we have used relations  (\ref{eqi1}) and (\ref{eqi2}),
 while
$$(\mathrm{Ad}(u_{h})\tilde{\alpha}_h \tilde{\rho} \tilde{\beta}_h )(v_{h}^{(h')}) $$
 $$=u_h(\mu_{h}(h)\tilde{\rho}( \overline{\nu_h}\mu_h(h+z_{h'}+z_{h''})v_{h''}^*v_{h'}^*v_h^{(h')}  v_{h'}    v_{h''}  )u_h^* $$
 $$=\mu_{h}(h)\overline{\nu_h}\mu_h(h+z_{h'}+z_{h''})(-\epsilon_{z_{h'}}(h')\mu_h(h')\chi_h(h''+z_h+z_{h'}) \xi_h)$$ $$\cdot u_h  (u_{h''}^*u_{h''}^{(h'')*}v_{h''})^*(u_{h'}^*u_{h'}^{(h')*}v_{h'})^*(u_{h'}^* u_h^{(h')*}  v_h^{(h')}v_{h'}u_{h''}v_{h'}^*u_{h'}^{(h')}u_{h'})$$ $$\cdot(u_{h'}^*u_{h'}^{(h')*}v_{h'})(u_{h''}^*u_{h''}^{(h'')*}v_{h''}) u_h^* $$
 $$=\xi_h\overline{\nu_h}\epsilon_{z_h}(h''+z_{h'}+z_{h''})\epsilon_{z_{h'}}(h')\epsilon_{z_{h''}}(h)$$ $$u_h v_{h''}^*u_{h''}^{(h'')}u_{h''}v_{h'}^*u_{h'}^{h'}u_{h}^{(h')*}v_h^{(h')}v_{h'}u_{h''}^{(h'')*}v_{h''} u_h^* $$
\item
$$(\tilde{\beta}_h \tilde{\rho})(v_{h'}^{(h'')})$$
$$=\tilde{\beta}_h(-\epsilon_{z_{h''}}(h'')\mu_{h'}(h'')\chi_{h'}(h+z_{h'}+z_{h''}) \xi_{h'} u_{h''}^* u_{h'}^{(h'')*}  v_{h'}^{(h'')}v_{h''}u_{h}v_{h''}^*u_{h''}^{(h'')}u_{h''} )$$
$$=-\epsilon_{z_{h''}}(h'')\mu_{h'}(h'')\chi_{h'}(h+z_{h'}+z_{h''}) \xi_{h'}(\chi_{h'}({h'}+z_{h}+z_{h''}))$$ $$\cdot(-\epsilon_{z_h}(h)\chi_{h''}(h'+z_h+z_{h''}))(\nu_{h'} \mu_{h'}(h'+z_{h}+z_{h''}))$$ $$\cdot (u_{h''}^{(h)})^*(v_{h''}^{(h)}v_h u_{h'}^{(h')} v_h^*v_{h''}^{(h)*})^*(v_{h''}^{(h)}v_h v_{h'} v_h^*v_{h''}^{(h)*})(v_{h''}^{(h)})(u_h^{(h)})(v_{h''}^{(h)})^*$$ $$\cdot (v_{h''}^{(h)} u_h^{(h)*}v_h u_{h'}^* v_h^*v_{h''}^{(h)*})(u_{h''}^{(h)})
$$
$$=\nu_{h'}\xi_{h'} \epsilon_{z_h+z_{h''}}(h')\epsilon_{z_{h'}}(h''+z_h+z_{h''}) u_{h''}^{(h)*}v_{h''}^{(h)}v_h u_{h'}^{(h')*}  v_{h'}  u_{h'}^* v_h^*v_{h''}^{(h)*}u_{h''}^{(h)},
$$
while
 $$(\mathrm{Ad}(u_{h})\tilde{\alpha}_h \tilde{\rho} \tilde{\beta}_h )(v_{h'}^{(h'')}) $$
 $$=u_{h}(\mu_{h'}(h)\tilde{\rho}( \nu_{h'} \mu_{h'}(h'+z_{h}+z_{h''})v_{h''}^{(h)}v_h v_{h'} v_h^*v_{h''}^{(h)*}) )u_h^* $$
 $$=\mu_{h'}(h) \nu_{h'} \mu_{h'}(h'+z_{h}+z_{h''})\xi_{h'}$$ $$\cdot u_{h}(u_{h}^* u_{h''}^{(h)*}  v_{h''}^{(h)}v_{h}u_{h'}v_{h}^*u_{h}^{(h)}u_{h})(u_h^*u_h^{(h)*}v_h)$$ $$\cdot(u_{h'}^*u_{h'}^{(h')*}v_{h'})(u_h^*u_h^{(h)*}v_h)^*(u_{h}^* u_{h''}^{(h)*}  v_{h''}^{(h)}v_{h}u_{h'}v_{h}^*u_{h}^{(h)}u_{h})^*u_h^*
 $$
 $$=\nu_{h'}\xi_{h'}\epsilon_{z_h+z_{h''}}(h')\epsilon_{z_{h'}}(h''+z_h+z_{h''}) u_{h''}^{(h)*}  v_{h''}^{(h)}v_{h}u_{h'}^{(h')*}v_{h'}
 u_{h'}^*v_h^*v_{h''}^{(h)*}u_{h''}^{(h)}
 $$
 
\item
$$(\tilde{\beta}_h \tilde{\rho})(v_{h''}^{(h)})$$
$$=\tilde{\beta}_h(-\epsilon_{z_{h}}(h)\mu_{h''}(h)\chi_{h''}(h'+z_{h''}+z_{h}) \xi_{h''} u_{h}^* u_{h''}^{(h)*}  v_{h''}^{(h)}v_{h}u_{h'}v_{h}^*u_{h}^{(h)}u_{h}) $$
$$=-\epsilon_{z_{h}}(h)\mu_{h''}(h)\chi_{h''}(h'+z_{h''}+z_{h}) \xi_{h''}\chi_{h''}(h+z_h)$$ $$\cdot \chi_h(h+z_h)\mu_{h''}(h+z_h)(-\epsilon_{z_{h''}}(h''))$$ $$\cdot (u_h^{(h)})^*(v_hu_{h''}v_h^*)^*(v_hv_{h''}v_h^*)(v_h)(v_{h''}^*u_{h''}^{(h'')*}v_{h''}u_{h}^*)(v_h)^*(v_hu_hv_h^*)(u_h^{(h)}) $$ 
$$=\xi_{h''} \epsilon_{z_h}(h'')\epsilon_{z_{h''}}(z_h)  u_h^{(h)*}v_hu_{h''}^*u_{h''}^{(h'')*}v_{h''}v_h^*u_h^{(h)} $$ 
while
$$(\mathrm{Ad}(u_{h})\tilde{\alpha}_h \tilde{\rho} \tilde{\beta}_h )(v_{h''}^{(h)}) $$
$$=u_h \mu_{h''}(h)\tilde{\rho}(\mu_{h''}(h+z_h)v_hv_{h''}v_h^*)  u_h^* $$
$$=\mu_{h''}(h)\mu_{h''}(h+z_h)\xi_{h''} u_h (u_h^*u_h^{(h)*}v_h)(u_{h''}^*u_{h''}^{(h'')*}v_{h''}) (u_h^*u_h^{(h)*}v_h)^*   u_h^* $$
$$\xi_{h''} \epsilon_{z_h}(h'')\epsilon_{z_{h''}}(z_h) u_h^{(h)*}v_h u_{h''}^*u_{h''}^{(h'')*}v_{h''}v_h^* u_h^{(h)}$$
\item
$$(\tilde{\beta}_h \tilde{\rho})(v_{h'})$$
$$=\tilde{\beta}_h(\xi_{h'} u_{h'}^*u_{h'}^{(h')*}v_{h'}) $$
$$=\xi_{h'}(-\epsilon_{z_{h''}}(h'')v_{h''}^*u_{h''}^{(h'')*}v_{h''}u_{h}^*)^*(\chi_{h'}(h''+z_{h''})v_{h''}^*u_{h'}^{h''} v_{h''})^*$$ $$\cdot(\overline{\nu_{h'}}\mu_{h'}(h''+z_{h''})v_{h''}^*v_{h'}^{(h'')} v_{h''}) $$
 $$=-\xi_{h'}\overline{\nu_{h'}}\epsilon_{z_{h'}}(h''+z_{h''})\epsilon_{z_{h''}}(h'')u_h v_{h''}^*v_{h'}^{(h'')} v_{h''},$$
 while
$$(\mathrm{Ad}(u_{h})\tilde{\alpha}_h \tilde{\rho} \tilde{\beta}_h )(v_{h'}) $$
$$=u_h \mu_{h'}(h)\tilde{\rho}(\overline{\nu_{h'}}\mu_{h'}(h''+z_{h''})v_{h''}^*v_{h'}^{(h'')} v_{h''})   u_h^* $$
 $$=\mu_{h'}(h)\overline{\nu_{h'}}\mu_{h'}(h''+z_{h''}) (-\epsilon_{z_{h''}}(h'')\mu_{h'}(h'')\chi_{h'}(h+z_{h'}+z_{h''})\xi_{h'})$$ $$\cdot
 u_h(u_{h''}^*u_{h''}^{(h'')*}v_{h''})^*(u_{h''}^* u_{h'}^{(h'')*}  v_{h'}^{(h'')}v_{h''}u_{h}v_{h''}^*u_{h''}^{(h'')}u_{h''})(u_{h''}^*u_{h''}^{(h'')*}v_{h''})u_{h}^*$$
 $$=
 -\xi_{h'}\overline{\nu_{h'}}\epsilon_{z_{h'}}(h''+z_{h''})\epsilon_{z_{h''}}(h'')
 u_hv_{h''}^*u_{h''}^{(h'')} v_{h'}^{(h'')}v_{h''}$$
\item
 $$(\tilde{\beta}_h \tilde{\rho})(v_{h''})=\tilde{\beta}_h(\xi_{h''} u_{h''}^*u_{h''}^{(h'')*}v_{h''}) $$
 $$=\xi_{h''}u_{h''}^{(h)*}(-\epsilon_{z_h}(h)\chi_{h''}(h'+z_h+z_{h''})v_{h''}^{(h)} u_h^{(h)*}v_h u_{h'}^* v_h^*v_{h''}^{(h)*})^*   v_{h''}^{(h)}  $$
 $$=\xi_{h''} \epsilon_{z_h}(h')\epsilon_{z_{h''}}(h'')u_{h''}^{(h)*} v_{h''}^{(h)} v_hu_{h'}v_h^*u_{h}^{(h')},$$
while
$$(\mathrm{Ad}(u_{h})\tilde{\alpha}_h \tilde{\rho} \tilde{\beta}_h) (v_{h''}) =u_h\mu_{h''}(h) \tilde{\rho}(v_{h''}^{(h)}) u_h^* $$
 $$=\mu_{h''}(h) u_h(-\epsilon_{z_{h}}(h)\mu_{h''}(h)\chi_{h''}(h'+z_{h''}+z_{h}) \xi_{h''} u_{h}^* u_{h''}^{(h)*}  v_{h''}^{(h)}v_{h}u_{h'}v_{h}^*u_{h}^{(h)}u_{h})u_h^*  $$
 $$=\xi_{h''}\epsilon_{z_h}(h')\epsilon_{z_{h''}}(h'') u_{h''}^{(h)*}  v_{h''}^{(h)}v_{h}u_{h'}v_{h}^*u_{h}^{(h)}$$
 \end{itemize}
 \end{proof}
\printbibliography
\end{document}